\documentclass[a4paper,11pt,reqno]{amsart}
\usepackage[english]{babel}

\usepackage[T1]{fontenc}
\usepackage{lmodern}
\usepackage[utf8]{inputenc}

\usepackage{amsfonts} 
\usepackage{amsmath} 
\usepackage{enumerate}

\usepackage[colorlinks]{hyperref}
\hypersetup{
  linkcolor=[RGB]{192,0,0},
  citecolor=[RGB]{0, 120, 0},
  urlcolor=[rgb]{0.6, 0.2, 0.2}
}

\newcommand{\cosimo}[1]{{\color{purple} #1}}
\newcommand{\stefano}[1]{{\color{blue} #1}}

\usepackage[alphabetic]{amsrefs} 
\usepackage{amssymb} 
\usepackage[nameinlink]{cleveref}
\usepackage{stmaryrd}

\usepackage{amsthm} 

\usepackage{xcolor} 

\usepackage{fix-cm}

\usepackage{bm}
\usepackage{listings}

\usepackage{physics}

\usepackage{fancyhdr}

\usepackage{array}
\usepackage{braket}
\usepackage{scalefnt}
\usepackage{lipsum}
\usepackage{nicefrac, xfrac}

\usepackage{scalerel}

\usepackage{mathtools}
\usepackage{thmtools}
\usepackage{graphicx}
\usepackage[normalem]{ulem}
\usepackage{tabularx}

\usepackage{paralist}
\usepackage{tensor} 
\usepackage{comment}
\usepackage{nicematrix}

\usepackage{tikz}
\usepackage{tikz-cd}
\usepackage{tikz-3dplot}
\usetikzlibrary{cd}
\usetikzlibrary{arrows}
\usetikzlibrary{matrix}
\usetikzlibrary{positioning}

\tikzcdset{arrow style=math font}

\DeclareMathAlphabet{\mathbb}{U}{msb}{m}{n}

\NiceMatrixOptions{xdots/shorten=0.6em}

\usepackage{xcolor}
\definecolor{red}{rgb}{1,0,0}
\definecolor{darkred}{RGB}{192,0,0}

\newcommand{\binomial}[2]{\left( \begin{array}{c}\hspace{-5pt} #1 \\ \hspace{-5pt}#2 \end{array}\hspace{-5pt}\right)}

\newcommand{\transpose}[1]{\tensor[^{\mathrm{t}}]{#1}{}}

\usepackage[textwidth=0.8in]{todonotes}
\newcommand{\bfa}{\mathbf{a}}

\newcommand{\bfp}{\mathbf{p}}

\newcommand{\bfx}{\mathbf{x}}

\newcommand{\calD}{\mathcal{D}}

\newcommand{\calH}{\mathcal{H}}
\newcommand{\calI}{\mathcal{I}}

\newcommand{\calK}{\mathcal{K}}
\newcommand{\calL}{\mathcal{L}}

\newcommand{\calO}{\mathcal{O}}

\newcommand{\calR}{\mathcal{R}}

\newcommand{\bbC}{\mathbb{C}}

\newcommand{\bbN}{\mathbb{N}}

\newcommand{\bbP}{\mathbb{P}}

\newcommand{\bbR}{\mathbb{R}}

\newcommand{\bbT}{\mathbb{T}}

\newcommand{\mfS}{\mathfrak{S}}

\newcommand{\rmi}{\mathrm{i}}

\newcommand{\rmm}{\mathrm{m}}

\newcommand{\rmp}{\mathrm{p}}

\newcommand{\rmr}{\mathrm{r}}
\newcommand{\rms}{\mathrm{s}}
\newcommand{\rmt}{\mathrm{t}}

\newcommand{\bfbeta}{\boldsymbol{\beta}}

\renewcommand{\tilde}[1]{\widetilde{#1}}
\renewcommand{\hat}[1]{\widehat{#1}}

\newcommand{\eps}{\varepsilon}

\newcommand{\Lap}{\mathop{}\!\Delta}

\DeclareMathOperator{\Ann}{Ann}

\DeclareMathOperator{\Hom}{Hom}

\DeclareMathOperator{\irk}{irk}
\DeclareMathOperator{\Isot}{Isot}

\DeclareMathOperator{\Proj}{Proj}

\DeclareMathOperator{\rk}{rk}

\DeclareMathOperator{\Sym}{Sym}

\DeclareMathOperator{\expdim}{expdim}
\DeclareMathOperator{\sat}{sat}

\DeclarePairedDelimiter{\pint}{\lfloor}{\rfloor}
\DeclarePairedDelimiter{\pa}{\langle}{\rangle}

\DeclarePairedDelimiter{\pq}{[}{]}

\newcommand{\Mat}{\mathrm{Mat}}
\newcommand{\GL}{\mathrm{GL}}
\newcommand{\Oa}{\mathrm{O}}

\newcommand{\SO}{\mathrm{SO}}

\newcommand{\mfsl}{\mathfrak{sl}}

\newcommand{\mfso}{\mathfrak{so}}

\usepackage[left=1in, right=1in, top=1in, bottom=1in, includefoot, headheight=13.6pt]{geometry}

\usepackage{enumitem}

\usepackage{adjustbox}

\allowdisplaybreaks

\numberwithin{equation}{section}
\theoremstyle{definition}
\newtheorem{defn}[equation]{Definition}
\theoremstyle{plain}
\newtheorem{teo}[defn]{Theorem}
\newtheorem{prop}[defn]{Proposition}

\newtheorem{lem}[defn]{Lemma}

\newtheorem{cor}[defn]{Corollary}
\theoremstyle{remark}
\newtheorem{remark}[defn]{Remark}
\newenvironment{rem}
{\pushQED{\qed}\remark}
{\popQED\endremark}
\newtheorem{example}[defn]{Example}
\newenvironment{exam}
{\pushQED{\qed}\example}
{\popQED\endexample}
\allowdisplaybreaks

\title{Isotropic rank of harmonic polynomials}

\date{}

\author{Stefano Canino}
\author{Cosimo Flavi}
\address{{\normalfont (Stefano Canino)},
\normalfont \scshape\fontfamily{ptm}\selectfont
	Dipartimento di Matematica, Università degli Studi di Trento, \normalfont{Via Sommarive 14, 38123 Povo (Trento), Italy.}}
\email{stefano.canino@unitn.it}
\address{{\normalfont (Cosimo Flavi)},
\normalfont \scshape\fontfamily{ptm}\selectfont
	Wydział Matematyki, Informatyki i Mechaniki, Uniwersytet Warszawski, \normalfont{ul.~Stefana Banacha 2, 02-097 Warsaw, Poland.}}
\email{c.flavi@ew.edu.pl}

\begin{document}
\begin{abstract}
    Any homogeneous harmonic polynomial can be decomposed as a sum of powers of isotropic linear forms, that is, linear forms whose coefficients are the coordinates of isotropic points. The minimum size of such decompositions for a harmonic polynomial is called its isotropic rank. As with the Waring rank, the problem of determining the isotropic rank of a given harmonic form is very hard. We determine the isotropic rank of a general harmonic form providing a full classification of the dimensions of secant varieties of the variety of $d$-powers of isotropic linear forms in $n+1$ variables, for every $n,d\in\bbN$, thus obtaining the analogue of the widely-celebrated Alexander-Hirschowitz theorem. Moreover, we completely solve the problem of determining the isotropic rank for the following classes of harmonic forms: ternary forms, quadrics and monomials.
\end{abstract}

\maketitle

\section{Introduction}
Harmonic functions are commonly defined as the functions that are annihilated by the Laplace operator
\[\Lap=\pdv[2]{}{x_0}+\cdots+\pdv[2]{}{x_n}\] and play an important role in several branches of science. In particular, the Laplace equation $\Lap f=0$, whose solutions are exactly harmonic functions, is one of the most important differential equations in mathematics.
First appearing in the 18{\textsuperscript{th}} century in \cites{Eul61,dal61}, it became one of the pillars of modern physics thanks to the work of P.-S.~de Laplace in \cite{Lap99}. This work provided the basis for the theory of celestial mechanics and gravitational potential, giving rise to the theory of harmonic functions and potential theory, which developed in the subsequent centuries. Many applications of the Laplace equation are related to fluid dynamics (see \cite{Kli72}) and electrodynamics (see \cite{Jac99}). For more information on the theory of harmonic functions in analysis and mathematical physics, we refer the reader to \cites{Eva10,AW13}.

\subsection{Harmonic polynomials and their applications}
This paper focuses on homogeneous harmonic polynomials and their decompositions as sums of powers of isotropic linear forms. These are linear forms whose coefficients correspond to an isotropic point with respect to the standard quadratic form. We refer to these decompositions as \textit{isotropic decompositions}, and the minimum size of an isotropic decomposition of a given harmonic form $h$ is called the \textit{isotropic rank} of $h$ and it is denoted by $\irk h$.
We provide an overview of the isotropic rank of several classes of harmonic forms. The main result is given by the isotropic rank of a \textit{generic} harmonic form of any degree and in any number of variables. Furthermore, we determine the isotropic rank of binary and ternary harmonic forms, harmonic quadratic forms, corresponding to the family of trace-zero matrices, and harmonic monomials, providing the \textit{harmonic} version of the solution of the Waring problem for monomials, solved in \cite{CCG12}.

Harmonic polynomials form a finite-dimensional, rotation-invariant subspace of the space of all polynomials. This structure  gives them deep significance in several branches of modern mathematics (see \cite{ABR01}), so connecting Algebra, Analysis, Physical Mathematics, and Geometry.

From an algebraic and representation-theoretic perspective, the space of harmonic polynomials of a fixed degree forms an irreducible $\SO_{n+1}(\bbC)$-module, as shown in \cite{Wey97}. This property places harmonic polynomials at the heart of harmonic analysis on Euclidean spaces and spheres. There, they provide the angular part of homogeneous polynomials and enable the Fischer decomposition, which expresses any homogeneous polynomial as a sum of a harmonic component and a radial term. This decomposition is fundamental to many aspects of invariant theory, spherical designs, and zonal polynomials in multivariate analysis (see \cite{DX14}).

Harmonic polynomials have applications in potential theory and geophysics. For example, they are used to model fluid flow and gravitational fields (see \cite{SF14}), as well as in numerical analysis and approximation theory (see \cite{HMPS14}).

One of the most common uses concerns the \textit{spherical harmonics}, which are the restriction to the unit sphere of a homogeneous harmonic polynomial in $\bbR^3$ (see \cite{Jac99}). Spherical harmonics indeed represent the three-dimensional case of harmonic polynomials and find several applications. In real-time lighting and computer graphics, spherical harmonics are used to approximate irradiance and precomputed radiance transfer efficiently (see \cite{RH01}). In Geodesy and Earth sciences, spherical harmonics are used for models describing the Earth's gravitational and potential fields
(see \cite{GS23}). They are also crucial in diffusion magnetic resonance imaging (dMRI) and High Angular Resolution Diffusion Imaging (HARDI) (see \cite{TCC07}).
On the computational side, efficient algorithms introduced in \cites{DH94,Moh99} enable high-resolution forward and inverse transforms, making detailed analyses of spherical signals. 
In climate and atmospheric science, spherical harmonic expansions are widely used to represent global fields such as temperature, pressure, and wind velocity on the Earth’s surface (see \cites{Swa89,HJ92}). 
In cosmology, they are used to analyse the cosmic microwave background (CMB) radiation (see \cite{Teg97}).
Finally, in acoustics and wave theory, spherical harmonics provide the natural basis for representing radiating sources and modal decompositions (see \cites{KFCS82,MI86}).

Hyperspherical harmonics generalize spherical harmonics to higher dimensions and are widely used in few-body quantum and nuclear physics. Foundational work in \cite{Ave89} introduces the hyperspherical harmonic formalism, establishing it as a powerful tool for representing and solving multi-particle quantum systems. Later, in \cite{MDGKV20}, the method is applied to nuclear interactions, demonstrating its effectiveness for describing correlations and interactions in few-nucleon systems.
More recently, hyperspherical harmonics have found applications in both quantum few-body modeling and acoustic and engineering problems (see \cites{Tim23,Szw23}).

\subsection{Isotropic rank}
A classical notion of rank appearing in classical literature is the Waring rank of homogeneous polynomials. In particular, the Waring rank of a homogeneous polynomial of degree $d$ is defined as the minimum size of a decomposition of $f$ as a sum of $d$-th powers of linear forms. Given a homogeneous polynomial $f$ we denote by $\rk f$ its Waring rank.
From a more modern point of view, the concept of rank can be generalised to an arbitrary non-degenerate irreducible projective variety $X\subseteq\bbP^n$ in the context of \textit{secant varieties}. Considering the set
\begin{equation*}
\sigma_r^{\circ}(X)=\bigcup_{P_1,\dots,P_r\in X}\pa{P_1,\dots,P_r},
\end{equation*}
the \textit{$X$-rank} of a point $P\in \bbP^n$ is the minimum number $r$ such that $P\in\sigma_r^{\circ}(X)$. The \textit{$r$-th secant variety} of $X$ is the Zariski closure $\sigma_r(X)\coloneqq\smash{\overline{\sigma_r^{\circ}(X)}}\subseteq\bbP^n$. In particular, we define the \textit{generic $X$-rank} as the minimum number $r$ such that $\sigma_r(X)=\bbP^n$. Note that, if $\rk_{\mathrm{gen}}$ is the generic $X$-rank, then $\sigma_{\rk_{\mathrm{gen}}}(X)$ is an open dense Zariski subset of points having $X$-rank equal to $\rk_{\mathrm{gen}}$.
For more information on the Waring rank and, more in general, on secant varieties, we refer the reader to \cites{BC19,BCC+18,BGI11,CGO14,Lan12,LO13}.

For any $X\subset\bbP^n$ as above, there is an expected dimension of $\sigma_r(X)$ given by
$$\expdim\sigma_r(X)=\min\{r(\dim X+1)-1,n\}.$$
In general, we have $\dim\sigma_r(X)\leq\expdim \sigma_r(X)$ and, if the inequality is strict, we say that \textit{$\sigma_r(X)$ is defective}, or that \textit{$X$ is $r$-defective}. If $X$ is $r$-defective, the number $$\delta_r\coloneqq\expdim\sigma_r(X)-\dim\sigma_r(X)$$ is called the \textit{$r$-defect of $X$}. Determining the dimensions of the secant varieties of a projective variety is in general a difficult problem. In the literature, many cases are studied and several results have been provided. So far, one of the most complete classifications is given by the Alexander-Hirschowitz theorem (see \cite{AH95}). This famous result provides, for every $n,d\in\bbN_{\geq 1}$,  the dimension of all the secant varieties of the $(n,d)$-Veronese variety, i.e.~the embedding of $\bbP^n$ with $\calO_{\bbP^n}(d)$, corresponding to the variety of the $d$-th powers of linear forms in $n+1$ variables. In this case, the $X$-rank coincides with the Waring rank. Other results concerning the dimensions of secant varieties are given for the Grassmannian (\cites{AOP12,BC23,BO24,Bor13, BB11, CGG05}), for Segre varieties (\cites{AOP09,CGG05b,LM04}), for Segre-Veronese varieties (\cites{AB13,ABGO24,Bal25,LMR22}), for spinor varieties (\cites{Ang11,Man09}).

Although many of these results can be stated over an arbitrary field, in this paper we work over the field of complex numbers $\bbC$.
In the following, we define for any $n,d\in\bbN_{\geq 1}$ the variety $\Isot_{n,d}$ as the image, via the $(n,d)$-Veronese embedding, of the non-degenerate quadric  hypersurface in $\bbP^n$ corresponding to the equation $x_0^2+\cdots+x_n^2=0$, so that $\Isot_{n,d}$ is the set of $d$-th powers of isotropic linear forms in $n+1$ variables. We give the full classification of the dimensions of the secant varieties of $\Isot_{n,d}$ and this gives the analogue of the Alexander-Hirschowitz theorem for harmonic polynomials. Denoting by $\calH_{n,d}$ the space of harmonic forms of degree $d$ in $n+1$ variables, which is an irreducible $\SO_{n+1}(\bbC)$-module by \cite{GW98}*{Theorem 5.2.4}, it is easy to prove that $\calH_{n,d}=\langle\Isot_{n,d}\rangle$ (see \autoref{lem_isotropic_points_harmonic_forms}). Therefore, the notion of isotropic rank is indeed well defined. We also observe that the choice of the Laplace operator $\Delta$ is completely arbitrary and it could be changed with any other non-degenerate quadratic form. Consequently, the space of harmonic polynomials would change as well, but all the results would be equivalent. In light of that, for the sake of clarity and without loss of generality, here we state our main results for the specific case of the Laplace operator.
\begin{teo}\label{theo:AH_harmonics}
Let $n,d,r\in\bbN_{\geq 1}$. If $d\geq 3$, then
$$\dim \sigma_r(\Isot_{n,d})=\min\{rn-1,\dim\calH_{n,d}-1\},$$
while, if $d=2$, then
$$\dim\sigma_r(\Isot_{n,2})=\begin{cases}rn-\dbinom{r-1}{2}-1, & \text{ if $1\leq r\leq n$},\\[2ex]
\dim\calH_{n,2}-1, & \text{ if $r\geq n+1$}.\end{cases}$$
In particular, the isotropic rank of a generic harmonic polynomial of $\calH_{n,2}$ is $n+1$, and the isotropic rank of a generic harmonic polynomial of $\calH_{n,d}$, with $d\geq 3$, is 
$$\left\lceil\frac{1}{n}\Biggl(\binomial{n+d}{n}-\binomial{n+d-2}{n}\Biggr)\right\rceil.$$
\end{teo}
The proof of \Cref{theo:AH_harmonics}, which is given below in its equivalent formulation of \Cref{teo:post_fin}, is based on the differential Horace method, introduced with a modern point of view in \cite{Hir85} and used by J.~Alexander and A.~Hirschowitz in \cite{AH95} their theorem. The differential Horace method is treated in detail in \Cref{horace}. A simplified proof of the Alexander-Hirschowitz theorem is provided in \cite{BO08}, where the authors also give a more elementary proof for the case $d=3$. We use indeed the same strategy for the case of cubic harmonic forms (see \autoref{thm:post-3}).

The dimension of the space of harmonic binary forms of degree $d$ is equal to $2$ for every $d\in\bbN$. In particular, the elements $(x_0+\rmi x_1)^d$ and $(x_0-\rmi x_1)^d$ are powers of isotropic linear forms and give a basis of $\calH_{1,d}$. This gives the following immediate statement for binary harmonic forms.
\begin{prop}\label{prop:harmonic_binary_forms}
    If $h\in\calH_{1,d}$ is not a power of an isotropic linear form, then $\irk h=2$.
\end{prop}
The case of ternary harmonic forms is not trivial, but it is strictly connected to the Waring rank of binary forms. This is due to the fact that the Lie algebra $\mfso_3(\bbC)$ is canonically isomorphic to $\mfsl_2(\bbC)$. Using the uniqueness of the irreducible representations of $\mfsl_2(\bbC)$ (see e.g.~\cite{FH91}*{Section 11.1}), we have that the ring of ternary harmonic forms is isomorphic to the ring of binary forms. This strong connection, explained in detail in \Cref{sec:ternary_forms}, allows to determine the isotropic rank of every ternary harmonic form, by using the Sylvester algorithm (see \cite{Syl51}) in the improved version provided by G.~Comas and M.~Seiguer in \cite{CS11}*{Theorem 2}.
\begin{teo}\label{teo:harmonic_Comas_Seiguer}
    Let $d\in\bbN$. If $2\leq r\leq d$, then:
    \begin{enumerate}[label=(\arabic*), widest=*,topsep=2pt,itemsep=2pt]
\item $\sigma_r(\Isot_{2,d})\setminus\sigma_{r-1}(\Isot_{2,d})=\sigma_r^\circ(\Isot_{2,d})\cup\sigma_{2d-r+2}^\circ(\Isot_{2,d})$;
\item $\sigma_r^\circ(\Isot_{2,d})=\sigma_r(\Isot_{2,d})\setminus\sigma_{2d-r+2}(\Isot_{2,d})$;
\item $\sigma_{2d-r+2}^\circ(\Isot_{2,d})=\sigma_{2d-r+2}(\Isot_{2,d})\setminus \sigma_{r-1}(\Isot_{2,d})$.
    \end{enumerate}
If $r=d+1$, then:
    \begin{enumerate}[label=(\arabic*), widest=*,start=4,topsep=2pt,itemsep=2pt]
\item $\sigma_r^\circ(\Isot_{2,d})=\sigma_r(\Isot_{2,d})\setminus \sigma_{r-1}(\Isot_{2,d})$.
\end{enumerate}
\end{teo}
Another class of harmonic forms which we analyse is the one of quadratic forms. While for the case of Waring rank the problem is easy, as the rank of a quadratic form is equal to the rank of the matrix associated, the harmonic case is not so immediate.  
A harmonic quadratic form corresponds to a symmetric trace-zero matrix. Inspired by the definition for harmonic forms, the isotropic rank of such a matrix is defined as the minimum number of symmetric trace-zero rank one matrices whose sum gives that matrix. Surprisingly, almost every trace-zero matrix has isotropic rank equal to its rank. Therefore, almost every trace-zero matrix of rank $r$ can be written as a sum of rank one matrices whose columns, or rows, correspond to isotropic points. We describe explictely what we mean by \lq\lq almost every\rq\rq\,: the only matrices having isotropic rank higher than their rank are the nilpotent matrices whose Jordan form consists of a $3$-order block and other blocks of order $2$ and $1$.
\begin{teo}\label{teo: quadric equiv}
Let $H\in\bbC^{(n+1)\times (n+1)}$ be a symmetric matrix such that $\tr(H)=0$. If $H$ is nilpotent and $\rk H^2=1$, then $\irk H=\rk H+2$. Otherwise, $\irk H=\rk H$.
\end{teo}
Note that the condition on the Jordan blocks is equivalent to the one stated in the theorem. The strategy for obtaining the complete classification for the isotropic rank of trace-zero matrices is a full analysis on all the possible canonical Jordan forms assumed by a trace-zero matrix. The first case is the class of non-degenerate matrices, for which the two ranks are equal, coherently with \autoref{theo:AH_harmonics}. Successively, the normal forms provided by F.~R.~Gantmacher in \cite{Gan98b}, related to Jordan canonical forms, allow to analyse each of the possible cases.

A natural question is how much the isotropic rank of a harmonic form can be greater than its Waring rank. We answer this by proving that
\[
\rk h\leq\irk h\leq2\rk h
\]
for every $h\in\calH_{n,d}$ (see \autoref{teo:two_times_rank}). Furthermore, the upper bound is sharp for an infinite family of forms, as showed by the analysis of the isotropic rank of the harmonic monomials, which is the last case we deal with.
\begin{teo}\label{teo:harmonics_monomials}
    Let $\rmm=\ell_0^{a_0}\cdots\ell_r^{a_r}\in\calH_{n,d}$ be a harmonic monomial, where $1\leq a_0\leq\cdots\leq a_r$. If $a_0=1$, 
    $\ell_0$ is not isotropic, and $\ell_1,\dots,\ell_r$ are isotropic, then
\[
\irk \rmm=2\rk\rmm=2\prod_{i=2}^r(a_i+1).
\]
Otherwise,
\[
\irk\rmm=\rk\rmm=\prod_{i=2}^r(a_i+1).
\]
\end{teo}
Also for monomials, we observe that their Waring rank equals the isotropic rank for almost all the cases. The only exception is given when one and only one of the linear forms appearing in the product forming the monomial is non-isotropic. While for the cases where the two ranks coincide it is sufficient to determine an explicit decomposition made of isotropic forms, the case where these are different is a bit more demanding. The argument is treated in detail in  \Cref{sec:monomials}.

The methodology introduced in this paper gives a new and strong correlation between representation theory and tensor theory. For this reason, it is thus desirable that what has been done here for harmonic polynomials can be done for other classes of tensors linked with other groups.

\section{Apolarity Theory and harmonic polynomials}
In this section, we briefly recall some basic facts about apolarity theory and we introduce the main topic of our investigation, i.e.~harmonic polynomials. Throughout the paper, by subscheme we always mean a closed subscheme and by variety we mean an irreducible projective integral separated scheme of finite type.
\subsection{Apolarity}
The apolarity theory is a classical topic, which was first introduced in \cite{Rey70} by T.~Reye, where the author uses the term \textit{apolar} for the
first time, and then by J.~Rosanes in \cite{Ros73} (see
also \cite{Dol12}*{p.~75} for historical details). W.~F.~Meyer further developed this theory in \cite{Mey83}.
For more recent references, the reader can consult \cites{DK93,IK99,Dol12}.

Let $V$ be a vector space over $\bbC$, with $\dim V=n+1$. 
We denote by
$\Sym V$ and $\Sym V^*$
the symmetric algebra of $V$ and its dual $V^*$, respectively. Recall that $\Sym V$ and $\Sym V^*$ are graded algebras via the following decompositions: 
$$\Sym V=\bigoplus_{d\in\bbN}S^dV,\qquad \Sym V^*=\bigoplus_{d\in\bbN}S^dV^*,$$
where $S^dV$ and $S^dV^*$ are the $d$-th symmetric powers of $V$ and $V^*$, respectively.
For every $d\in\bbN$, the bilinear map defined as 
\begin{equation}
\label{rel_contraction_pairing}
\begin{tikzcd}[row sep=0pt,column sep=1pc]
 \circ\colon S^dV^*\times S^dV\arrow{r} & \bbC \\
  {\hphantom{\circ\colon{}}}  (\eta_1\cdots\eta_d, v_1\cdots v_d) \arrow[mapsto]{r} & \displaystyle{\sum_{\sigma\in\mfS_d}} \eta_1(v_{\sigma(1)})\cdots\eta_d(v_{\sigma(d)})
\end{tikzcd}
\end{equation}
is a perfect pairing and it is called the \textit{contraction pairing} or \textit{polar pairing} (see, e.g.~\cite{Dol12}*{Section 1.1.1}). More generally, for every $k,d\in\bbN$, with $k\leq d$, the \textit{$(k,d)$-partial polarization map} is defined as
$$
\begin{tikzcd}[row sep=0pt,column sep=1pc]
 \circ\colon S^kV^*\times S^dV\arrow{r} & S^{d-k}V\hphantom{.} \\
  {\hphantom{\circ\colon{}}}  (\eta_1\cdots\eta_k, v_1\cdots v_d) \arrow[mapsto]{r} & \displaystyle{\sum_{1\leq i_1<\cdots<i_k\leq d}} (\eta_1\cdots\eta_k)\circ (v_{i_1}\cdots v_{i_k})\prod_{j\neq i_1,\dots,i_k}v_j.
\end{tikzcd}
$$
The symmetric algebras $\Sym V$ and $\Sym V^*$ can be identified with polynomial rings. Indeed,  denoting by $\bbC[V]$ the ring of polynomial functions on $V$, that is, the commutative $\bbC$-algebra generated by $V^*$,  we have the canonical isomorphisms 
\[
\bbC[V]\cong\Sym V^*,\qquad \bbC[V^*]\cong \Sym V,
\]
(see e.g.~\cite{CGLM08}*{section 3.1}).
Consider a basis $(x_0,\dots,x_n)$ of $V$ and its dual basis $(\alpha_0,\dots,\alpha_n)$ of $V^*$. 
We define the two rings
\[
\calR_n\coloneqq\bbC[x_0,\dots,x_n],\qquad \calD_n\coloneqq\bbC[\alpha_0,\dots,\alpha_n],
\]
endowed with the standard grading, and we denote by $\calR_{n,d}$ the degree $d$ part of $\calR_n$ and by $\calD_{n,d}$ the degree $d$ part of $\calD_n$. We have $\calR_{n,d}\cong S^dV$ and $\calD_{n,d}\cong  S^dV^*$ as vector spaces. 
These isomorphisms agree also with the multiplication and we have
$$
\calR_{n}\cong\Sym V\cong\bbC[V^*],\qquad \calD_{n}\cong\Sym V^*\cong\bbC[V].
$$
 In these terms, the apolar action of $\calD_{n,1}$ on $\calR_{n,1}$ is
\[
\begin{tikzcd}[row sep=0pt,column sep=1pc]
 \circ\colon \calD_{n,1}\times\calR_{n,1}\arrow{r} & \bbC\hphantom{.} \\
  {\hphantom{\circ\colon{}}} (\alpha_i,x_j) \arrow[mapsto]{r} & \displaystyle\pdv{x_j}{x_i}.
\end{tikzcd}
\]
The apolar action 
$\circ: \calD_n\times\calR_{n}\to\calR_{n}$ is obtained by extension and using the standard and formal properties of differentiation, i.e.~linearity and the Leibniz rule. In light of that, we think of $\calR_n$ as a ring of polynomials and of $\calD_n$ as a ring of polynomial derivations, and we will often denote a generic element of $\calR_n$ by $f$ and a generic element of $\calD_n$ by $\varphi$.
We recall now the definitions of perps and inverse systems and some of their properties.
{
\begin{defn}
For any subspace $V\subseteq \calR_{n,d}$, the \textit{perp} of $V$ is the space
$$V^\perp\coloneqq\Set{\varphi\in\calD_{n,d}| \varphi\circ f=0,\,\forall f\in V}.$$
For any subspace $W\subseteq \calD_{n,d}$, the \textit{perp} of $W$ is the space
$$W^\perp\coloneqq\Set{f\in\calR_{n,d}| \varphi\circ f=0,\,\forall\varphi\in W}.$$
\end{defn}
}
\noindent Note that, since $\circ$ is a perfect pairing, we have \[
\dim V^\perp=\dim\calR_{n,d}-\dim V,
\] 
and analogously for $W$.
\begin{defn}
For any ideal $I\subseteq\calR_n$, the \textit{inverse system} of $I$ is the $\calR_n$-submodule of $\calD_n$ defined as
$$I^{-1}\coloneqq\Set{\varphi\in\calD_n| \varphi\circ f=0,\,\forall f\in I}.$$
For any ideal $J\subseteq\calD_n$, the \textit{inverse system} of $J$ is the $\calD_n$-submodule of $\calR_n$ defined as
$$J^{-1}\coloneqq\Set{f\in\calR_n|\varphi\circ f=0,\,\forall \varphi\in J}.$$
\end{defn}
Perps and inverse systems behave well with respect to the intersection and the restriction to homogeneous degree. Throughout the paper, if $I$ is an ideal of a graded ring, we denote by $I_d$ its homogeneous part of degree $d$.
\begin{lem}[\cite{Ger96}*{Proposition 2.5 and Proposition 2.6}]
\label{lem:invprop}
Let $I$ and $J$ be two ideals of $\calR_n$ or of $\calD_n$. 
Then:
\begin{enumerate}
    \item $(I^{-1})_d=(I_d)^\perp$ for every $d\in\bbN$;
    \item $(I\cap J)^{-1}=I^{-1}+J^{-1}$.  
\end{enumerate}
\end{lem}
To conclude this section, we recall the apolarity lemma, which gives a tool for computing the Waring rank of a homogeneous polynomial. If $Y$ is a variety and $X\subseteq Y$ is a closed subscheme of $Y$, we denote by $I_{X,Y}$ the ideal of $X$ in $Y$; if $Y=\bbP^n$, we just write $I_X$.
\begin{lem}[Apolarity lemma \cite{IK99}*{Lemma 1.15}]
\label{Lem Apo}
Let $f\in\calR_{n,d}$ be a homogeneous polynomial and $(f)^{-1}\subseteq\calD_n$ its inverse system. Let $$\ell_i=a_{i,0}x_0+\dots+a_{i,n}x_n\in\calR_{n,1},\qquad\varphi_i=a_{i,0}\alpha_0+\dots+a_{i,n}\alpha_n\in\calD_{n,1}$$ for $i=1,\dots,r$. Then, there exist $c_1,\dots,c_r\in\bbC$ such that $f=c_1\ell_1^d+\dots+c_r\ell_r^d$ if and only if $I_X\subset(f)^{-1}$, where $X=\{\varphi_1,\dots,\varphi_r\}$. In particular, the Waring rank of $f$ is equal to the minimal length of a reduced 0-dimensional scheme $X\subset \bbP(\calD_{n,1})$ such that $I_X\subset(f)^{-1}$.
\end{lem}
 \subsection{Harmonic forms}
 Using the notions of apolarity introduced above, we can give a definition of harmonic polynomials, which is analogous to the classical notion of harmonic functions with the Laplace operator.
\begin{defn}
Let $\omega_n\in\calD_{n,2}$ be a non-degenerate quadratic form, that is, $\rk{\omega_n}=n+1$. A form $f\in\calR_{n,2}$ is called \textit{$\omega_n$-harmonic} if $\omega_n\circ f=0$. The inverse system
$$\calH_{n}^{\omega_n}\coloneqq(\omega_{n})^{-1}=\Set{f\in\calR_n|\omega_{n}\circ f=0}\subseteq \calR_n$$
is called the \textit{space of $\omega_n$-harmonic polynomials}.
\end{defn}
We denote by $\calH_{n,d}^{\omega_n}$ the degree $d$ part of $\calH_n^{\omega_n}$, so that
\[
\calH_{n}^{\omega_n}=\bigoplus_{d\in\bbN}\calH_{n,d}^{\omega_n}.
\]
Since 
$$\calD_{n,2}\cong S^2(V^*)\subset V^*\otimes V^*\cong\Hom(V,V^*),$$
and the two isomorphisms are canonical, any non-degenerate quadratic form $\omega_n\in\calD_{n,2}$ canonically induces an isomorphism of vector spaces
$$L_{\omega_n}\colon\calR_{n,1}\to\calD_{n,1},$$
which can also be extended to a ring isomorphism $L_{\omega_n}\colon\calR_{n}\to\calD_{n}$. The quadratic form $L_{\omega_n}^{-1}(\omega_n)$ is called the \textit{dual quadratic form of $\omega_n$}. Moreover, there is also an induced non-degenerate bilinear form
$$
\begin{tikzcd}[row sep=0pt,column sep=1pc]
 \langle\ ,\,\rangle_{\omega_n}\colon \calR_{n,1}\times\calR_{n,1}\arrow[r] & \bbC\hphantom{.} \\
  {\hphantom{\langle\ ,\,\rangle_{\omega_n}\colon{}}}  (\ell_1, \ell_2) \arrow[mapsto]{r} & L_{\omega_n}(\ell_1)\circ\ell_2.
\end{tikzcd}
$$
For more details on this isomorphism and its consequence, we refer the reader to \cite{Dol12}*{Section 1.4.5}. We define $$\calK_n^{\omega_n}\coloneqq L_{\omega_n}(\calH_n^{\omega_n}).$$
Even though $\calK_n^{\omega_n}$ and $\calH_n^{\omega_n}$ are canonically isomorphic, we prefer to keep them separate for the sake of clarity and, for this reason, we refer to $\calK_n^{\omega_n}$ as the space of \emph{$\omega_n$-harmonic derivations}.
Note that an equivalent definition for $\calK_n^{\omega_n}$ is $$\calK_{n}^{{\omega_n}}=\bigl(L_{\omega_n}^{-1}(\omega_n)\bigr)^{-1}=\Set{\varphi\in\calD_n|\varphi\circ L_{\omega_n}^{-1}(\omega_n)=0}\subseteq \calD_n.$$
As for $\calH_n^{\omega_n}$, we denote by $\calK_{n,d}^{\omega_n}$ the degree $d$ part of $\calK_n^{\omega_n}$. When it is clear from the context, we omit the symbol of the considered quadric form, just writing $\calH_{n,d}$, or $\calK_{n,d}$, respectively, and we say \textit{harmonic} instead of \textit{$\omega_n$-harmonic}.
\begin{rem}
Fix a non-degenerate quadratic form $\omega_n\in\calD_{n,2}$ and let $q_n\in\calR_{n,2}$ its dual quadratic form. A classical result states that
$$\calR_{n,d}=q_n\calR_{n,d-2}\oplus\calH_{n,d}$$
see \cite{GW98}*{Corollary 5.2.5} or \cite{Fla24}*{Proposition 3.11} for a direct proof. As a consequence, we have that
$$\calH_n\cong\calR_n/(q_n)$$
and $\calH_n$ is a graded ring whose grading is given by
$$\calH_n=\bigoplus_{d=0}^\infty\calH_{n,d}.$$
Note that $\calH_n$ is the coordinate ring of a non-degenerate quadric $Q_n=\{q_n=0\}\subseteq\bbP^n$. Analogous formulas also hold for $\calD_{n}$ and $\calK_n$.
\end{rem}
We want to understand how the apolarity action restricts to harmonic polynomials. For this purpose, we observe that, for any $h\in\calH_{n,d}$ and for any $\varphi\in\calD_{n,d}$, written as  $\varphi=\psi_{d-2}\omega_n+\psi_d$ with $\psi_{d-2}\in\calD_{n,d-2}$ and $\psi_{d}\in\calD_{n,d}$,
we have
$$\varphi\circ h=(\psi_{d-2}\omega_n+\psi_d)\circ h=\psi_{d-2}\circ(\omega_{n}\circ h)+\psi_d\circ h=\psi_d\circ h.$$
As an immediate consequence 
we get the following theorem.
\begin{teo}\label{teo:polar_harmonic}
The apolar action of $\calD_n$ on $\calR_d$ induces a map
$$
\circ\colon\calK_{n}\times\calH_{n} \to  \calH_n$$
such that, for any $d\in\bbN$, $\circ\colon\calK_{n,d}\times\calH_{n,d}\to\bbC$ is a perfect pairing. 
\end{teo}
We refer to $\circ\colon\calK_{n}\times\calH_{n} \to  \calH_n$ as \emph{harmonic apolar action}. Note that, as for the classic apolar action, it is possible to define inverse systems and perps for harmonic apolar action too. From now on, we use for inverse systems and perps referred to the harmonic apolar action the same notation introduced for their classic counterparts and, when needed, we specify which apolar action we are using to avoid confusion.
\section{Isotropic rank of harmonic polynomials}
In this section we introduce the notion of isotropic rank, an analogue of the Waring rank for the class of harmonic polynomials. In order to define it, we first introduce isotropic linear forms and analyse some of their properties.
\begin{defn}\label{defn:orthogonal}
    Two linear forms $\ell_1,\ell_2\in\calR_{n,1}$ are said \textit{$\omega_n$-orthogonal} if $\omega_n\circ(\ell_1\ell_2)=0$. 
\end{defn}
\begin{defn}
For any linear form $\ell_0\in\calR_{n,1}$, the \textit{$\omega_n$-orthogonal space} to $\ell_0$ is the space
\[
\ell_0^{\perp}\coloneqq\Set{\ell\in\calR_{n,1}|\omega_n\circ(\ell_0\ell)=0}.
\]
\end{defn}
\begin{defn}\label{defn:isotropic}
Let $\omega_n\in\calD_{n,2}$ be a non-degenerate quadratic form. A linear form $\ell\in\calR_{n,1}$ is \textit{$\omega_n$-isotropic} if $\omega_n\circ\ell^2=0$. 
\end{defn}
Note that, since $\omega_n\circ(\ell_1\ell_2)=2\langle\ell_1,\ell_2\rangle_{\omega_n}$, the linear forms $\ell_1$ and $\ell_2$ are $\omega_n$-orthogonal if and only if they are orthogonal with respect to $\langle\ ,\,\rangle_{\omega_n}$, and similarly for $\omega_n$-isotropic linear forms.  When there is no risk of confusion on the quadratic form, we write \textit{orthogonal} instead of \textit{$\omega_n$-orthogonal} and \textit{isotropic} instead of \textit{$\omega_n$-isotropic}.
The following properties are classically known. For a detailed proof, see
\cite{Fla24}*{Lemma 4.10}.
\begin{lem} 
\label{lem_isotropic_points_harmonic_forms}
Let $\omega_n\in\calD_{n,2}$ be a non-degenerate quadratic form. Then the following hold:
\begin{enumerate}[label=(\arabic*), widest=*,nosep]
\item a linear form $\ell\in\calR_{n,1}$ is $\omega_n$-isotropic if and only if $\ell^d$ is $\omega_n$-harmonic for every $d\geq 2$;
\item the space $\calH_{n,d}^{\omega_n}$ is generated by the $d$-th powers of $\omega_n$-isotropic linear forms, that is,
\[
\calH_{n,d}^{\omega_n}=\Bigl\langle\Set{\ell^d\in\calR_{n,d}|\omega_n\circ\ell^2=0}\Bigr\rangle.
\]
\end{enumerate}
\end{lem}
Thanks to \autoref{lem_isotropic_points_harmonic_forms}, it is natural to give the following notion of rank.
\begin{defn}\label{def:isot_rank}
    Let $\omega_n\in\calD_{n,2}$ be a non-degenerate quadratic form and let $h\in\calH_{n,d}$. The \textit{$\omega_n$-isotropic rank} of $h$ is
    \[
    \irk_{\omega_n}(h)=\min\Set{r\in\bbN|h=\sum_{i=1}^r\ell_i^d: \ell_i\in\calR_{n,1},\, \omega_n\circ\ell_i^2=0,\,\forall i=1,\dots,r}.
    \]
\end{defn}
When there is no risk of confusion on the quadratic form, we write $\irk h$ to denote the isotropic rank of a harmonic form $h$. As an immediate consequence of \Cref{Lem Apo}, we obtain the following harmonic version of apolarity lemma.
\begin{lem}
\label{lem:apo_harm}
Let $h\in\calH^{\omega_n}_{n,d}$ be a harmonic polynomial and $(h)^{-1}\subseteq\calK^{\omega_n}_n$ its inverse system with respect to the harmonic apolar action. Let $(x_0,\dots,x_n)$ be a basis of $\calR_{n,1}$ and let $(\alpha_0,\dots,\alpha_n)$ be its dual basis such that $\omega_n=\alpha_0^2+\dots+\alpha_n^2$. Let \[
\ell_i=a_{i,0}x_0+\dots+a_{i,n}x_n\in\calR_{n,1}
\]
be a $\omega_n$-isotropic linear forms and let \[
\varphi_i=a_{i,0}\alpha_0+\dots+a_{i,n}\alpha_n\in\Omega_n\subset\bbP(\calD_{n,1}),
\]
for every $i=1,\dots,r$, where $\Omega_n$ is the quadric defined by $\omega_n$. Then, there exist $c_1,\dots,c_r\in\bbC$ such that $f=c_1\ell_1^d+\dots+c_r\ell_r^d$ if and only if $I_{X,\Omega_n}\subset(h)^{-1}$, where $X=\{\varphi_1,\dots,\varphi_r\}$. In particular, the isotropic rank of $h$ is equal to the minimal length of a reduced 0-dimensional scheme $X\subset \Omega_n$ such that $I_{X,\Omega_n}\subset(h)^{-1}$.
\end{lem}
The following remark will allow us to compute the isotropic rank in several cases.
\begin{rem}\label{rem: apofacile}
The second part of \Cref{lem:apo_harm} can be stated equivalently in the setting of classical apolarity as follows: the isotropic rank of $h\in\calH_{n,d}^{\omega_n}$ is the minimal length of a reduced 0-dimensional scheme $X\subset\bbP(\calD_{n,1})$ such that $I_X\subset(h)^{-1}\subset\calD_{n,1}$ and $\omega_n\in I_X$.
\end{rem}
Given $n,d\in\bbN$ we denote by $\nu_{n,d}$ the $(n,d)$-Veronese embedding
\[
\begin{tikzcd}[row sep=0pt,column sep=1pc]
 \nu_{n,d}\colon\bbP(\calR_{n,1})\arrow[r] & \bbP(\calR_{n,d})\hphantom{,}\\
  {\hphantom{\nu_{n,d}\colon{}}} [\ell] \ar[r,mapsto] & {[{\ell}^d]},
\end{tikzcd}
\]
and by $V_{n,d}\coloneqq\nu_{n,d}\bigl(\bbP(\calR_{n,1})\bigr)$ the $(n,d)$-Veronese variety. 
As $\bbP^n$ parameterises linear forms and $V_{n,d}$ parameterises their powers, it is possible to define a variety of isotropic linear forms and a variety consisting in their powers.
\begin{defn}
Let $\omega_n\in\calD_{n,2}$ be a non-degenerate quadratic form and $q_n\in\calR_{n,2}$ its dual quadratic form. The \textit{variety of $\omega_n$-isotropic linear forms} is the variety defined by $q_n=0$ and it is denoted by $\Isot_n^{\omega_n}$.  The \textit{d-th $\omega_n$-isotropic Veronese variety} is defined as 
$$\Isot_{n,d}^{\omega_n}\coloneqq\nu_{n,d}(\Isot_{n}^{\omega_n}).$$
\end{defn}
Again, when there is no risk of confusion on the quadratic form $\omega_n$, we write just $\Isot_n$ and $\Isot_{n,d}$ instead of $\Isot_n^{\omega_n}$ and $\Isot_{n,d}^{\omega_n}$.
Note that, by definition, $\Isot_{n,d}$ is the variety of the powers of isotropic linear forms and, in particular, $\Isot_{n,d}$ is smooth and $\dim(\Isot_{n,d})=n-1$. Observe that, according to \Cref{lem_isotropic_points_harmonic_forms}, $\Isot_{n,d}$ is degenerate in $\bbP(\calR_{n,d})$, but it is non-degenerate in $\bbP(\calH_{n,d})$. For this reason, we consider $\Isot_{n,d}$ as a subvariety of $\bbP(\calH_{n,d})$.


For any projective variety $X\subseteq\bbP^n$,  $\hat X$ denotes  the affine cone of $X$.
\begin{prop}
\label{prop:tangent_space_isotropic_powers}
For every $\ell_0\in\Isot_n$, the tangent space at $\ell_0^d$ to the variety $\Isot_{n,d}$ is given by
\[
T_{\ell_0^d}\bigl(\hat{\Isot}_{n,d}\bigr)=\Set{\ell_0^{d-1}\rmr|\rmr\in\ell_0^\perp}\simeq\ell_0^{\perp}.
\]
\end{prop}
\begin{proof}
Since for any point $\ell_0\in\hat{\Isot}_n$ there exists a linear subspace contained in $\Isot_n$ and containing $\ell_0$, there exists a line contained in $\hat{\Isot}_n$ passing through $\ell_0$, that is, there exists a curve $\gamma\colon\bbC\to\hat{\Isot}_n$ defined as
\[
\gamma(t)\coloneqq(\ell_0+t\ell)^d
\]
for every $t\in\bbC$, for some $\ell\in\calR_{n,1}$.
Since $\ell_0+t\ell\in\Isot_n$ for every $t\in\bbC$, we must have
\[
\omega_n\circ(\ell_0+t\ell)^2=0.
\] 
In particular, we get
\[
(\omega_n\circ \ell_0^2)+2t\bigl(\omega_n\circ (\ell_0\ell)\bigr)+t^2(\omega_n\circ \ell^2)=0
\]
for every $t\in\bbC$. Then, since $(\omega_n\circ \ell_0^2)=0$, we must have
\[
2\bigl(\omega_n\circ (\ell_0\ell)\bigr)+t(\omega_n\circ \ell^2)=0
\]
for every $t\in\bbC$
which implies that $\omega_n\circ (\ell_0\ell)=0$ and $\omega_n\circ \ell^2=0$,
that is, $\ell\in\ell_0^{\perp}$ and $\ell\in\hat{\Isot}_n$. In particular, since
\[
\eval{\dv{\gamma}{t}}_{t=0}=\eval{\dv{t}((u_1+t\ell)^d)}_{t=0}=\eval{(\ell_0+t\ell)^{d-1}\ell}_{t=0}=\ell_0^{d-1}\ell,
\]
we have $$\Set{\ell_0^{d-1}\ell|\ell\in\ell_0^\perp\cap\hat{\Isot}_{n,d}}\subseteq T_{\ell_0^d}\bigl(\hat{\Isot}_{n,d}\bigr).$$ Finally, since $\dim\ell_0^\perp=\dim T_{\ell_0^d}\bigl(\hat{\Isot}_{n,d}\bigr)=n $ we get the statement.
\end{proof}

Clearly, for every polynomial $h\in\calH_{n,d}$, we have
\[
\rk h\leq \irk h.
\]
Asking which are the harmonic polynomials whose Waring rank is strictly less than their isotropic rank is a natural question. We will consider it in detail in the following sections. Moreover, in general, the Waring rank of a harmonic polynomial also gives an upper bound on the isotropic rank.
\begin{teo}\label{teo:two_times_rank}
Let $\omega_n\in\calD_{n,2}$ a non-degenerate quadratic form and let $h\in\calH_{n,d}^{\omega_n}$. Then
\[
\rk h\leq \irk_{\omega_n} h\leq 2\rk h.
\]
\end{teo}
\begin{proof}
Let $\Omega_n\subset\bbP^n$ the quadric defined by $\omega_n$. Let $h\in\calH_{n,d}^{\omega_n}$ and let us assume that $\rk h=r$. Then there exists $\ell_1,\dots,\ell_r\in\calR_{n,1}$ such that
    \[
    h=\ell_1^d+\cdots+\ell_r^d.
    \]
Let $X\subset\bbP^n$ be the set of points corresponding to such a decomposition and let $I_X\subset (h)^{-1}$ the corresponding apolar ideal. Given a point $A\in X$, the locus of the points $B\in\bbP^n$ such that the line $AB$ is tangent to $\Omega_n$ is closed. As a consequence, there exists $P\in\bbP^n$ such that none of the lines passing through $P$ and one of the point of $X$ is tangent to $\Omega_n$. Let $\pi\colon\bbP^n\to \bbP^{n-1}$ be the projection from the point $P$ to a generic hyperplane. Then we have $r_0\coloneqq\abs{\pi(X)}\leq r$ and the cone $\pi^{-1}\bigl(\pi(X)\bigr)$ of vertex $P$ and base $\pi(X)$ consists in a set of $r_0$ distinct lines passing through the point $P\not\in \Omega_n$. In particular, they intersect $\Omega_n$ in pairwise distinct points. It follows that $\Omega_n\cap\pi^{-1}\bigl(\pi(X)\bigr)$ is reduced and 
    \[
    \abs{\Omega_n\cap\pi^{-1}\bigl(\pi(X)\bigr)}=2r_0\leq 2r.
    \]
    It remains to prove that
    \[
    I\Bigl(\Omega_n\cap\pi^{-1}\bigl(\pi(X)\bigr)\Bigr)\subseteq (h)^{-1}.
    \]
    Now, we have
    \[
I\Bigl(\Omega_n\cap\pi^{-1}\bigl(\pi(X)\bigr)\Bigr)
=\sat\biggl(I(\Omega_n)+I\Bigl(\pi^{-1}\bigl(\pi(X)\bigr)\Bigr)\biggr)=I\bigl(\Omega_n\bigr)+I\Bigl(\pi^{-1}\bigl(\pi(X)\bigr)\Bigr),
    \]
    where the second equality follows by the fact that $\pi^{-1}\bigl(\pi(X)\bigr)$ is Cohen-Macaulay. Finally, it is enough to note that \[
    \omega_n\in (h)^{-1},\qquad I\Bigl(\pi^{-1}\bigl(\pi(X)\bigr)\Bigr)\subset I_X\subset (h)^{-1}.\qedhere
    \]
\end{proof}
As a direct consequence, we also get an upper bound depending on the number of isotropic linear forms appearing in a Waring decomposition.
\begin{cor}
    Let $\omega_n\in\calD_{n,2}$ be a non-degenerate quadratic form and let $h\in\calH_{n,d}^{\omega_n}$. If there exists a minimal decomposition of $h$ containing $k$ isotropic linear forms, then
    \[
\rk h\leq \irk_{\omega_n} h\leq 2\rk h-k.
\]
\end{cor}
\begin{proof}
    Let $\rk h=r$ and let $\ell_1,\dots,\ell_r\in\calR_{n,1}$ be linear forms such that
    \[
    h=\ell_1^d+\cdots+\ell_r^d
    \]
    and $\omega_n\circ\ell_i^2=0$ for every $i=1,\dots,k$. Then, defining $h'\coloneqq h-\ell_{1}^d-\cdots-\ell_k^d$, we have $h'\in\calH_{n,d}^{\omega_n}$ and $\rk h'\leq r-k$. In particular, by \autoref{teo:two_times_rank}, we have
    \[
    \irk_{\omega_n} h'\leq 2(r-k),
    \]
and hence,
\[
\irk_{\omega_n}h\leq 2r-k.\qedhere
\]
\end{proof}

\section{Secant varieties of isotropic Veronese varieties}\label{sec:d=2}
In this section, we study the dimensions of $\sigma_r(\Isot_{n,d})$ for any $r,n,d\in\bbN$. We compute them by distinguishing three cases: $d=2$, $d=3$ and $d\geq 4$. The case $d=2$ is solved by computing the tangent spaces of $\sigma_r(\Isot_{n,2})$. The case $d=3$ is the more challenging one and is solved generalising the technique introduced in \cite{BO08}. Finally, the case $d\geq 4$ is solved by induction using the differential Horace method.

We recall here the Terracini lemma, which provides a useful way to compute the tangent space of a secant variety.
\begin{lem}[Terracini Lemma \cite{Ter12}]
\label{lem: Terracini}
Let $Y\subset\bbP^n$ be a smooth variety, $Z_1,\dots,Z_r\in Y$ general points in $Y$ and $P$ a general point in $\left\langle Z_1,\dots,Z_r\right\rangle$. Then,
$$T_P\bigl(\sigma_r(Y)\bigr)=T_{Z_1}Y+\dots+T_{Z_r}Y.$$
\end{lem}
\subsection{The case \texorpdfstring{$d=2$}{d=2}}
We start by focusing on the case of quadratic harmonic forms. Note that, as expected, in analogy with Veronese varieties, in this case all the proper secant varieties to $\Isot_{n,2}$ are defective and thus the generic isotropic rank is higher than the expected value.
\begin{prop}\label{prop:dim_isotropic_secants_quadrics}
For every $n,r\in\bbN$, 
$$\dim\sigma_r(\Isot_{n,2})=\begin{cases}
rn-\dbinom{r-1}{2}-1,\quad &\text{if $1\leq r\leq n$,}\\[2ex]
\dim\calH_{n,2}-1,\quad &\text{if $r\geq n+1$}.\end{cases}$$
\end{prop}
\begin{proof}
    Given any isotropic linear form $\ell_1\in\Isot_{n,2}$, we have by \autoref{prop:tangent_space_isotropic_powers} that 
    \[
    T_{\ell_1^2}(\hat{\Isot}_{n,2})=\ell_1\ell_1^{\perp}.
    \]
    For any $r\in\bbN$ such that $2\leq r\leq n+1$, we can consider $r$ linearly independent isotropic linear forms $\ell_1,\dots,\ell_r$, such that,
    \begin{equation}\label{formula:conditions_harmonic_quadrics}
    \omega_n\circ\ell_i\ell_j\neq 0,
    \end{equation}
    for every $i,j$ such that $1\leq i<j\leq r$, that is, $\ell_i\not\in\ell_j^\perp$ for every $i\neq j$. Then, by condition \eqref{formula:conditions_harmonic_quadrics}, we can suppose that 
\[
\bigcap_{j=1}^r\ell_j^{\perp}=\langle\rmm_0,\dots,\rmm_{n-r}\rangle.
\]
Moreover, for every $j=1,\dots,r$, and $1\leq i_1<i_2\leq r$ such that $i_1,i_2\neq j$, the quadratic form
\[
\bigl(\omega_n\circ(\ell_j\ell_{i_2})\bigr)\ell_j\ell_{i_1}-\bigl(\omega_n\circ(\ell_j\ell_{i_1})\bigr)\ell_j\ell_{i_2}
\]
is $\omega_n$-harmonic. Therefore, for every $j=1,\dots,r$, we observe that 
\begin{align*}
T_{\ell_j^2}(\hat{\Isot}_{n,2})&=\langle\ell_j^2,\ell_j\rmm_0,\dots,\ell_j\rmm_{n-r}\rangle+\Bigl\langle\bigl(\omega_n\circ(\ell_j\ell_{i_2})\bigr)\ell_j\ell_{i_1}-\bigl(\omega_n\circ(\ell_j\ell_{i_1})\bigr)\ell_j\ell_{i_2}\Bigr\rangle_{\substack{1\leq i_1<i_2\leq r\\i_1,i_2\neq j}}.
\end{align*}
Indeed, for $r=2$ the second block is $\langle 0\rangle$, and, for $r\geq 3$,
fixing $i_1\neq j$, the $r-2$ linear forms
\[
\bigl(\omega_n\circ(\ell_j\ell_{i_2})\bigr)\ell_{i_1}-\bigl(\omega_n\circ(\ell_j\ell_{i_1})\bigr)\ell_{i_2},
\]
with $i_2\neq i_1,j$ are linearly independent and 
\[
\langle\ell_j,\rmm_0,\dots,\rmm_{n-r}\rangle\cap\Bigl\langle\bigl(\omega_n\circ(\ell_j\ell_{i_2})\bigr)\ell_{i_1}-\bigl(\omega_n\circ(\ell_j\ell_{i_1})\bigr)\ell_{i_2}\Bigr\rangle_{\substack{1\leq i_2\leq r\\i_2\neq i_1,j}}=\{0\}.
\]
Then, considering a generic linear combination $h$ of $\ell_1^2,\dots,\ell_r^2$, by \Cref{lem: Terracini},
we have
\begin{equation}\label{formula:secant_varieties_harmonic_quadrics}
T_{h}\sigma_r(\hat{\Isot}_{n,2})=\biggl(\bigoplus_{j=1}^r\langle\ell_j^2,\ell_j\rmm_0,\dots\ell_j\rmm_{n-r}\rangle\biggr)\oplus \biggl\langle \ell_1\ell_2-\frac{\omega_n\circ(\ell_1\ell_{2})}{\omega_n\circ(\ell_{i_1}\ell_{i_2})}\ell_{i_1}\ell_{i_2}\biggr\rangle_{(i_1,i_2)\neq(1,2)}.
\end{equation}
Since $\ell_1,\dots,\ell_r$ are linearly independent, we consider a linear change of variables and write the space of quadratic forms as \[
\calR_{n,2}=\bbC[\ell_1,\dots,\ell_r,\rmm_0,\dots,\rmm_{n-r}]_2.\]
Therefore, by formula \eqref{formula:secant_varieties_harmonic_quadrics}, we have
\[
\dim\bigl(T_{h}\sigma_r(\hat{\Isot}_{n,2})\bigr)=r(n-r+2)+\frac{r(r-1)}{2}-1=\frac{r(2n-r+3)}{2}-1=rn-\dbinom{r-1}{2}.
\]
In particular, in the case where $r=n+1$, we obtain
\[
\dim\bigl(T_{h}\sigma_n(\hat{\Isot}_{n,2})\bigr)=\dim \calH_{n,2}=\binom{n+2}{2}-1.\qedhere
\]
\end{proof}
\subsection{The translation of the problem}\label{sec:d=3}
We now want to study the dimension of $\sigma_r(\Isot_{n,d})$ for $d\geq 3$. To do that we use a very classical approach, i.e.~we reduce the problem of computing $\dim\sigma_r(\Isot_{n,d})$ to an interpolation problem of $0$-dimensional schemes, and we solve it by using the differential Horace method. For further information on postulation problems, we refer the reader to \cite{BM04}. Since the latter works by induction on $n$ and $d$, we first have to check the base cases $d=2$, $d=3$, and $n=3$. The case $d=2$ follows by \autoref{prop:dim_isotropic_secants_quadrics} and the case $n=3$ follows from \cites{Laf02,LP13}. For the base case $d=3$, we use the Brambilla-Ottaviani technique developed in \cite{BO08}. 
We recall here some definitions and results on Hilbert functions and good postulation of subschemes. In the rest of the paper, $Q_n=\{q_n=0\}$ always denotes a non-degenerate quadric in $\bbP^n$.

Given a variety $Y$ and a subscheme $X\subset Y$ we denote by $\calI_{X,Y}$ the ideal sheaf of $X$ in $Y$. If $X\subset Y$ is a 0-dimensional scheme, then we denote by $\ell(X)$ its degree, and we call it the {\it length of $X$}; note that $\ell(X)$ does not depend on the embedding of $X$. If $Y\subset \bbP^n$ we denote by $\calO_Y(d)$ the sheaf $\calO_{\bbP^n|Y}(d)$ and by $\calI_{X,Y}(d)$ the sheaf $\calI_{X,Y}\otimes\calO_{Y}(d)$. From now on, by divisor we mean an effective Cartier divisor.
In this section, the apolar action which we consider is the harmonic one and, for this reason, we do not specify it every time.
\begin{defn}
Let $S$ be a finitely generated graded commutative $\bbC$-algebra and let 
$$S=\bigoplus_{d\in\bbN}S_d$$
its graded decomposition, with $S_0=\bbC$. \textit{The Hilbert function of $S$} is the function $H_S\colon \bbN\to\bbN$ defined by $H_S(d)=\dim_\bbC(S_d)$.
\end{defn}
\begin{defn}\label{defn:Hilbertfn}
Let $Y\subset\bbP^n$ be a variety and let $X\subset Y$ be a subscheme of $Y$. The {\it Hilbert function of $X$ in $Y$} is the function
\[
\begin{tikzcd}[row sep=0pt,column sep=1pc]
 H_Y(X,d)\colon\bbN\arrow[r] & \bbN\hphantom{.}\\
  {\hphantom{H_Y(X,d)\colon{}}} d \ar[r,mapsto] & h^0\bigl(\calO_Y(d)\bigr)-h^0\bigl(\calI_{X,Y}(d)\bigr).
\end{tikzcd}
\]
If $X$ is 0-dimensional, we have $h^0\bigl(\calI_{X,Y}(d)\bigr)\geq h^0\bigl(\calO_Y(d)\bigr)-\ell(X)$ and thus
$$h^0\bigl(\calI_{X,Y}(d)\bigr)\geq\max\bigl\{0,h^0\bigl(\calO_Y(d)\bigr)-\ell(X)\bigr\},\quad H_Y(X,d)\leq\min\bigl\{h^0\bigl(\calO_Y(d)\bigr),\ell(X)\bigr\}.$$
\end{defn} 
\begin{defn}
Let $X$ as in \autoref{defn:Hilbertfn} and let $X$ be 0-dimensional. Then {\it $X$ has good postulation in degree $d$} (or {\it $X$ has the expected Hilbert function in degree $d$}) if 
$$h^0\bigl(\calI_{X,Y}(d)\bigr)=\max\bigl\{0,h^0\bigl(\calO_Y(d)\bigr)-\ell(X)\bigr\}$$  
and  {\it $X$ has good postulation} (or {\it $X$ has the expected Hilbert function}) if it has good postulation in degree $d$ for any $d\in\bbN$.
\end{defn}
Note that saying that $X\subset Y$ has good postulation in degree $d$ is equivalent to say that $H_Y(X,d)=\min\bigl\{h^0\bigl(\calO_Y(d)\bigr),\ell(X)\bigr\}$.
\begin{rem}
We are mainly interested in the Hilbert functions of subschemes $X\subset Q_n$. In this case, we have $H^0\bigl(\calO_{Q_n}(d)\bigr)=\calH_{n,d}$ and, if $I\subseteq \calH_n$ is the ideal of $X$ in $Q_n$, then $H^0\bigl(\calI_{X,Y}(d)\bigr)=I_d$. As a consequence, we get 
\[
H_{Q_n}(X,d)=\dim\calH_{n,d}-\dim I_d=\dim (I^\perp)_d=\dim I^{-1}_d. \qedhere
\]
\end{rem}
To reduce our problem to an interpolation one, we start by giving the following definition.
\begin{defn}\label{defn:double-point}
Let $P\in Q_n\subset\bbP^n$ and $\wp\subset\calH_n$ the ideal defining $P$ in the coordinate ring of $Q_n$. The \emph{double point of $Q_n$ supported at $P$} is the 0-dimensional subscheme $$2_{Q_n}P:=\Proj(\calH_n/\wp^2)\subseteq Q_n.$$
\end{defn}
In \autoref{defn:double-point} we specified that $2_{Q_n}P$ is a double point of $Q_n$ in order to emphasise that it is not what is usually named just as double point, i.e. a double point of $\bbP^n$. However, in the rest of the paper we mainly deal with double points on quadrics, and thus, unless otherwise specified and when no confusion can arise, we write \lq\lq double point\rq\rq\, instead of \lq\lq double point of $Q_n$\rq\rq.
\begin{rem}\label{rem:lunghezza}
It is easy to show that, if we denote by $2P\subseteq\bbP^n$ the double point of $\bbP^n$ supported at $P\in Q_n\subset\bbP^n$, then $2_{Q_n}P=2P\cap \bbT_p Q_n$, where $\bbT_pQ_n$ is the tangent hyperplane to $Q_n$ at $P$. As a consequence,  $2_{Q_n}P$ is isomorphic to a double point of $\bbP^{n-1}$ and, in particular, its length is $\ell(2_{Q_n}P)=n$.
\end{rem}
The next proposition describes the tangent spaces of $\Isot_{n,d}$ in terms of perps of ideal of double points.
\begin{prop}\label{prop:tangent-inverse}
Let $P\in Q_n$ and $\wp\subseteq\calH_n$ the ideal defining $P$ in the coordinate ring of $Q_n$. Then, for every $d\geq 2$, $$T_{\nu_{n,d}(P)}\hat{\Isot}_{n,d}\simeq(\wp^2_d)^{\perp}.$$
\end{prop}
\begin{proof}
Without loss of generality, we may assume that $Q_n$ is the zero locus of $q_n=x_0^2+\dots+x_n^2$ and $P=[1,i,0,\dots,0]\in Q_n$. If we denote by $\ell:=x_0+ix_1$ the isotropic form corresponding to $P$, then, by \autoref{prop:tangent_space_isotropic_powers}, we have $$T_{\nu_{n,d}(P)}\hat{\Isot}_{n,d}=\ell^{d-1}\langle\ell,x_2,x_3,\dots,x_n\rangle.$$
Since $P=[1,i,0,\dots,0]$, then $\wp=(\ell, x_2,x_3,\dots,x_n)$ and thus $$(\alpha_0+i\alpha_1)^{d-1}\langle\alpha_0+i\alpha_1,\alpha_2,\dots,\alpha_n\rangle\subseteq(\wp^2_d)^\perp.$$
By \autoref{rem:lunghezza}, we know that $\dim\wp^2_d=\dim\calH_{n,d}-n$ and thus $\dim(\wp^2_d)^\perp=n$. Hence, by dimensional reasons, we get \[(\alpha_0+i\alpha_1)^{d-1}\langle\alpha_0+i\alpha_1,\alpha_2,\dots,\alpha_n\rangle=(\wp^2_d)^\perp.\qedhere
\]
\end{proof}
Using \Cref{prop:tangent-inverse} and properties of inverse systems, we can compute $\dim\sigma_r(\Isot_{n,d})$ as the union $r$ general double points on $Q_n$.
\begin{prop}\label{prop:interpol-problem}
Let $P_1,\dots,P_r\in Q_n$ be points in general position and let $\wp_1,\dots,\wp_n\subseteq\calH_n$ their defining ideals. Then
$$\dim\sigma_r(\Isot_{n,d})=\dim(\wp_1^2\cap\dots\cap\wp_r^2)^{-1}_d-1=H_{Q_n}(2_{Q_n}P_1\cup\dots\cup2_{Q_n}P_r,d)-1.$$
\end{prop}
\begin{proof}
We set $Z_i:=\nu_{n,d}(P_i)$. Since $P_1,\dots,P_r$ are in general position on $Q_n$, so are $Z_1,\dots,Z_r$ on $\Isot_{n,d}$. By \Cref{lem: Terracini}, if $Z$ is a general point in $\langle Z_1,\dots,Z_r\rangle$, then the tangent space to $\sigma_r(\hat{\Isot}_{n,d})$ at $Z$ is
$$T_Z\bigl(\sigma_r(\hat{\Isot}_{n,d})\bigr)=T_{Z_1}\hat{\Isot}_{n,d}+\dots+T_{Z_r}\hat{\Isot}_{n,d}.$$
By \autoref{prop:tangent-inverse} and by \autoref{lem:invprop} we get
\begin{align*}
T_{Z_1}\hat{\Isot}_{n,d}+\dots+T_{Z_r}\hat{\Isot}_{n,d}&\simeq\bigl((\wp_1^2)_d\bigr)^\perp+\dots+\bigl((\wp_r^2)_d\bigr)^\perp=(\wp_1^2)^{-1}_d+\dots+(\wp_r^2)^{-1}_d\\
&=\bigl((\wp_1^2)^{-1}+\dots+(\wp_r^2)^{-1}\bigr)_d=(\wp_1^2\cap\dots\cap\wp_r^2)^{-1}_d
\end{align*}
and this proves the first equality. The second one follows immediately by noting that
\[
\dim(\wp_1^2\cap\dots\cap\wp_r^2)^{-1}_d=\dim\calH_{n,d}-\dim(\wp_1^2\cap\dots\cap\wp_r^2)_d=H_{Q_n}(2_{Q_n}P_1\cup\dots\cup2_{Q_n}P_r,d).\qedhere\]
\end{proof}
As an immediate consequence of \autoref{prop:interpol-problem} we obtain the following corollary.
\begin{cor}\label{cor:equivalence}
Fix $n,d,r\in\bbN_{\geq1}$. Then $$\dim \sigma_r(\Isot_{n,d})=\expdim \sigma_r(\Isot_{n,d})$$ if and only if a general union of $r$ double points of $Q_n$ has good postulation in degree $d$.
\end{cor}
\begin{proof}
Let $X\subset Q_n$ be a general union of $r$ double points in $Q_n$. By \autoref{prop:interpol-problem}, we have
\begin{align*}
\dim\sigma_r(\Isot_{n,d})& =H_{Q_n}(X,d)-1\leq\min\bigl\{h^0\bigl(\calO_{Q_n}(d)\bigr),\ell(X)\bigr\}-1\\
& =\min\{\dim\calH_{n,d},rn\}-1=\expdim\sigma_r(\Isot_{n,d}).
\end{align*}
Hence, $\dim\sigma_r(\Isot_{n,d})=\expdim\sigma_r(\Isot_{n,d})$ if and only if $H_{Q_n}(X,d)=\min\bigl\{h^0\bigl(\calO_{Q_n}(d)\bigr),\ell(X)\bigr\}$.
\end{proof}
\subsection{The differential Horace method}\label{horace}
\autoref{prop:interpol-problem} allows us to translate the original problem of calculating $\dim\sigma_r(\Isot_{n,d})$  into an interpolation problem of double points on a quadric. More precisely, we are reduced to calculating the number of conditions imposed on the linear system $|\calO_{Q_n}(d)|$ by a general union of $r$ double points of $Q_n$, and, to do that, we use the so-called Horace method. It was used in a “rudimentary” form by Terracini in \cite{Ter1516}, and later introduced in its modern formulation by Hirschowtiz in \cite{Hir85}. In the following years, the method was refined in various forms and allowed J. Alexander and A. Hirschowitz to prove their widely celebrated theorem on the postulation of double points; see \cites{Ale88,AH92a,AH92b,AH95}. The reader interested in a more in-depth historical treatment of the subject can referer to \cite{BO08}*{Section 7}.

The idea of the Horace method, in our specific case, works as follows: given $X\subset Q_n$ a general union of double points, we want to compute $h^0\bigl(\calI_{X,Q_n}(d)\bigr)$. First of all, we fix a divisor $D\subset Q_n$, and we split the scheme $X$ into two schemes: the {\it trace of $X$ on $D$}, i.e.~the part of $X$ contained in $D$ considered as a subscheme of $D$, and the {\it residue of $X$ with respect to $D$}, i.e.~the part of $X$ not contained in $D$. At this point, we have to study the postulation of the trace in $D$ with respect to $\calO_{D}(d)$ and the postulation of the residue in $Q_n$ with respect to $\calO_{Q_n}(d)(-D)$. If $D\in|\calO_{Q_n}(1)|$ and $P\in D$, then $D\cong Q_{n-1}$, the trace of $2_{Q_n}P$ on $D$ is $2_{Q_{n-1}}P$, and the residue of $2_{Q_n}P$ with respect to $D$ is $P$. As a consequence, since the trace is a union of double points in $Q_{n-1}$ and the residue is a union of double points and simple points in $Q_n$, we can use induction on $n$ to study the postulation of the trace, and induction on $d$ to study the postulation of the residue. Clearly, by semicontinuity of cohomology, it suffices to prove the good postulation for a specialisation of $X$.

For the Horace method to work, it is necessary that $X$ can be specialised in such a way that the conditions imposed by $X$ on $|\calO_{Q_n}(d)|$, are divided between trace and residue “without waste”. Unfortunately, due to arithmetic obstructions, this is not always possible. However, we can use a more refined version of the Horace method, the differential Horace method, introduced by J. Alexander and A. Hirschowtiz in \cites{AH92a,AH00}, which allows $X$ to be split into trace and residue more precisely. In our case, the aforementioned arithmetic obstructions appear, and we therefore have to use the differential version of the method.

By \autoref{prop:dim_isotropic_secants_quadrics} and \autoref{prop:interpol-problem}, we know that, for $d=2$ and $r\leq n$, a general union of $r$ double points on $Q_n$ has bad postulation. This exceptional behaviour makes it particularly difficult to apply the Horace method, as described above, to prove that for $d=3$ a generic union of double points on $Q_n$ has good postulation. The analogous problem, which arises when studying the postulation of a generic union of double points in $\bbP^n$ in degree 3, was the most challenging point in proving the Alexander-Hirschowitz theorem, and in fact the case of cubics was the last to be proved. At a later time, the proof of the whole theorem was simplified by several authors and, in particular, M. C. Brambilla and G. Ottaviani provided in \cite{BO08} a very smart technique to study the case of cubics in $\bbP^n$. Their technique is based on a simple but very clever observation: arithmetic obstructions can be avoided by specialising the double points not on divisors but on a linear subspace of appropriate codimension. In \Cref{d=3bis}, we generalise the Brambilla-Ottaviani technique and we use it to solve our interpolation problem in degree 3; for a more detailed treatise of the topic we refer to \cite{BO08}.

We introduce now all the tools necessary to use the differential Horace method; for more details on the Horace differential method, the reader can refer to \cite{AH00}.

\begin{defn}
Let $Y$ be a variety, $X$ a subscheme of $Y$ and $H\subset Y$ an effective Cartier divisor. The schematic intersection $X\cap H\subset H$, defined by the ideal subsheaf $$\calI_{X,H}=(\calI_X+\calI_H)/\calI_H$$ of $\calO_H$, is called {\it the trace of $X$ on $H$} and is denoted by $\Tr_H(X)$. The subscheme of $Y$ defined by the ideal sheaf $\calI_X:\calI_H$ is called {\it the residue of $X$ with respect to $H$} and it is denoted by $\Res_H(X)$. The canonical exact sequence
$$\begin{tikzcd}
	0 \arrow[r] & \calI_{\Res_H(X),Y}(-H) \arrow[r] & \calI_{X,Y} \arrow[r] & \calI_{\Tr_H(X),H} \arrow[r] & 0
\end{tikzcd}$$
is called {\it the Castelnuovo sequence (or residual sequence) of $X$ with respect to $H$}.
\end{defn}
\begin{rem}\label{rem:trresdouble}
Consider $P\in Q_n\subset \bbP^n$ and $H\in|\calO_{Q_n}(1)|$ a general hyperplane section of $Q_n$ such that $P\in H$. Note that, by general assumption, $H\cong Q_{n-1}$. It is easy to see that
$$\Tr_H(2_{Q_n}P)=2_{H}P\subset H,\quad \Res_H(2_{Q_n}P)=P\in Q_n.$$
In particular we have $\ell\bigl(\Tr_H(2_{Q_n}P)\bigr)=n-1$ and $\ell\bigl(\Res_H(2_{Q_n}P)\bigr)=1$. We stress that the general assumption on $H$ is necessary: indeed, if $H$ is the tangent hyperplane to $Q_n$ at $P$, the equalities above are no longer true, because $2_{Q_{n}}P$ is contained in this tangent hyperplane.
\end{rem}
\autoref{rem:trresdouble} highlights the arithmetical obstructions that can appear when using the Horace method for double points. In fact, whenever we specialise a double point on a hyperplane section $H$ we are specialising $n-1$ conditions on $H$, but it may happen that the number of conditions we need to specialise is not divisible by $n-1$. In such a case we are forced to specialise extra conditions, thus wasting some of them and, in the so-called {\it cas rangé}, i.e.~those cases in which the length of the scheme is equal to the dimension of the linear system being considered, this causes the method to fail.

In order to avoid these obstructions, we introduce now the differential version of the Horace method.
\begin{defn}\rm
In the algebra of formal series $\bbC[[\bfx,y]]$, where $\bfx=(x_1,\dots,x_{n-1})$,
a \textit{vertically graded ideal with respect to $y$} is an ideal of the form:
$$I=I_0\oplus I_1y\oplus\dots\oplus I_{m-1}y^{m-1}\oplus(y^m)$$
where, for $i=0,\dots,m-1$, $I_i\subseteq\bbC[[\bfx]]$ is an ideal.
\end{defn}
\begin{defn}\label{vergrad}\rm
Let $Y$ be a smooth variety of dimension $n$ and $H$ be a smooth irreducible divisor of $Y$. We say that $X\subset Y$ is a \textit{vertically graded subscheme} of $Y$ with base $H$ and support $P\in H$ if $X$ is a 0-dimensional scheme supported at $P$ and there exists a regular system of parameters $(\bfx,y)$ at $P$ such that $y=0$ is a local equation for $H$ and the ideal of $X$ in $\widehat{\calO}_{Y,P}\cong\bbC[[\bfx,y]]$ is vertically graded with respect to $y$.
\end{defn}
In a sense, vertically graded schemes can be thought of as 0-dimensional schemes (supported on a point) that can be “sliced” into sections which are parallel to a divisor containing the support of the scheme.
\begin{defn}\rm
Let $Y$ be a smooth variety and $X\subset Y$ be a vertically graded subscheme of $Y$ with base $H$ and $p\geq 0$ a fixed integer. The \textit{$p$-th differential residue of $X$ with respect to $H$} is the  subscheme of $Y$ denoted by $\Res^p_H(X)$ and defined by the ideal sheaf
$$\calI_{Res^p_H(X),Y}\coloneqq\calI_{X,Y}+(\calI_{X,Y}\colon\calI_{H,Y}^{p+1})\calI_{H,Y}^p.$$
The \textit{$p$-th differential trace of $X$ on $H$} is the closed subscheme of $H$ denoted by $\Tr^p_H(X)$ and defined by the ideal sheaf
$$\calI_{\Tr_H^p(X),H}\coloneqq(\calI_{X,Y}\colon\calI_{H,Y}^p)\otimes\calO_H.$$
\end{defn}
Note that $\Res^p_H(X)$ is obtained by removing from $X$ the $(p+1)$-th \lq\lq slice\rq\rq\,of $X$, while $\Tr^p_H(X)$ is exactly the $(p+1)$-th slice. Moreover, for $p=0$ we obtain the standard trace and residue.
\begin{rem}
Consider $P\in Q_n\subset \bbP^n$ and $H\in|\calO_{Q_n}(1)|$ a general hyperplane section of $Q_n$ such that $P\in H$. Note that, by general assumption, $H\cong Q_{n-1}$. It is easy to see that
$$\Tr^1_H(2_{Q_n}P)=P\in H,\quad \Res^1_H(2_{Q_n}P)=2_HP\subset  Q_n.$$
In particular we have $\ell\bigl(\Tr^1_H(2_{Q_n}P)\bigr)=1$ and $\ell\bigl(\Res^1_H(2_{Q_n}P)\bigr)=n-1$. In this way we can solve the arithmetic obstructions. Note that, as in \autoref{rem:trresdouble}, the general assumption on $H$ is necessary.
\end{rem}

We set now the following notation: let $Y$ be a variety $X_1,\dots,X_r\subset Y$ be vertically graded subschemes with base $H$, $X=X_1\cup\dots\cup X_r$ and $\bfp=(p_1,\dots,p_r)\in\bbN^r$. We set
$$\Tr_H^\bfp(X)\coloneqq\Tr_H^{p_1}(X_1)+\dots+\Tr_H^{p_r}(X_r),\qquad \Res_H^\bfp(X)\coloneqq\Res_H^{p_1}(X_1)+\dots+\Res_H^{p_r}(X_r)$$

We are finally ready to state the Horace differential lemma.
\begin{lem}[Lemme d'Horace différentielle, \cite{AH00}*{Proposition 9.1}]\label{HoraceDiff}
Let $Y$ be a variety, $H$ be a smooth irreducible divisor on $Y$ and $\calL$ be a line bundle on $Y$. Consider a 0-dimensional subscheme  $X\subset Y$ and $Z_1,\dots,Z_r,Z'_1,\dots,Z'_r$ 0-dimensional irreducible subschemes of $Y$ such that $Z_i\cong Z'_i$ for $i=1,\dots,r$, $Z_i'$ has support on $H$ and is vertically graded with base $H$, and the supports of $Z=Z_1+\dots+Z_r$ and $Z'=Z_1'+\dots +Z_r'$ are generic in their respective Hilbert schemes. Let $\bfp=(p_1,\dots,p_r)\in\bbN^r$. If
\begin{enumerate}[leftmargin=*]
    \item $h^0(\calI_{\Tr_H(X)\cup\Tr^{\bfp}_H(Z'),H}\otimes\calL_{|H})=0$ and
    \item $h^0\bigl(\calI_{\Res_H(X)\cup\Res_H^{\bfp}(Z'),Y}\calL(-H)\bigr)=0$,
\end{enumerate}
then
$$h^0\bigl(\calI_{X\cup Z,Y}(d)\bigr)=0.$$
\end{lem}
We conclude this section with the following remark, which allows us to simplify many proofs.
\begin{rem}\label{rem:rangè}
Fix $n,d,r\in\bbN$ with $n\geq 1$ and let $X\subset Q_n$ a 0-dimensional subscheme. Note that:
\begin{itemize}[leftmargin=*]
    \item if $\ell(X)\leq h^0\bigl(\calO_{Q_n}(d)\bigr)$ and $h^0\bigl(\calI_{X,Q_n}(d)\bigr)=h^0\bigl(\calO_{Q_n}(d)\bigr)-\ell(X)$, then for any $X^\prime\subseteq X$ we have $h^0\bigl(\calI_{X^\prime,Q_n}(d)\bigr)=h^0\bigl(\calO_{Q_n}(d)\bigr)-\ell(X^\prime)$;
    \item if $h^0\bigl(\calO_{X,Q_n}(d)\bigr)=0$, then for any 0-dimensional scheme $X^\prime$ with $X^{\prime}\supseteq X$ we have \linebreak $h^0\bigl(\calI_{X^\prime,Q_n}(d)\bigr)=0$.
\end{itemize}
As a consequence, keeping in mind that the length of a double point in $Q_n$ is $n$, to prove that a general union of $r$ double points in $Q_n$ has good postulation in degree $d$ it is enough to prove one the two following conditions:
\begin{itemize}[leftmargin=*]
\item $X\subset Q_n$ has good postulation in degree $d$, where $X$ is a general union of $\pint{h^0\bigl(\calO_{Q_n}(d)\bigr)/n}$ double points on $Q_n$ and a scheme of length $h^0\bigl(\calO_{Q_n}(d)\bigr)-n\pint{h^0\bigl(\calO_{Q_n}(d)\bigr)/n}$ contained in a double point of $Q_n$;
\item a general union of $r$ double points on $Q_n$ has good postulation in degree $d$ for $$r\in\left\{\left\lfloor{\frac{h^0\bigl(\calO_{Q_n}(d)\bigr)}{n}}\right\rfloor,\left\lceil{\frac{h^0\bigl(\calO_{Q_n}(d)\bigr)}{n}}\right\rceil\right\}.$$
\end{itemize} 
\end{rem}
\subsection{The solution of the interpolation problem in degree 3}\label{d=3bis}
The purpose of this section is to prove the following theorem.
\begin{teo}\label{thm:post-3}
Let $n,r\in\bbN$ with $n,r\geq 1$. If $X$ is a general union of double points in $Q_n$ then $X$ has good postulation in degree 3.
\end{teo}
As we said, since the Horace method does not work in the case of cubics, we use the Brambilla-Ottaviani technique to solve it, and then we use cubics as a base case of our induction.

We recall that, if $Y$ is a variety, $X\subseteq Y$ is a subscheme and $H$ is a divisor, then, the sequence
\begin{equation}\label{exc}
\begin{tikzcd}
	0 \arrow[r] & \calI_{X\cup H,Y} \arrow[r] & \calI_{X,Y} \arrow[r] & \calI_{X\cap H,H} \arrow[r] & 0
\end{tikzcd}
\end{equation}
is exact.

For the rest of this section, given $n\in\bbN$, we denote
$$k_n\coloneqq\biggl\lfloor{\frac{h^0\bigl(\calO_{Q_n}(3)\bigr)}{n}}\biggr\rfloor, \qquad \delta_n\coloneqq h^0\bigl(\calO_{Q_n}(3)\bigr)-nk_n.$$
Clearly, $k_n$ and $\delta_n$ are the unique natural numbers such that $h^0\bigl(\calO_{Q_n}(3)\bigr)=nk_n+\delta_n$ and $0\leq\delta_n\leq n-1$.
\begin{lem}\label{lem:num-lem}
Fix $n\in\bbN$ and let $p,r\in\bbN$ the unique natural numbers such that $n=6p+r$ with $0\leq r\leq 5$. Then
$$
k_n=\begin{cases}
6p^2+6p,\quad  &\text{if $r=0$,}\\
6p^2+8p+2,\quad &\text{if $r=1$,}\\
6p^2+10p+3,\quad &\text{if $r=2$,}\\
6p^2+12p+5,\quad &\text{if $r=3$,}\\
6p^2+14p+7,\quad &\text{if $r=4$,}\\
6p^2+16p+10,\quad &\text{if $r=5$,}\\
\end{cases}\qquad
\delta_n=\begin{cases}
5p,\quad  &\text{if $r=0$,}\\
0,\quad &\text{if $r=1$,}\\
3p+1,\quad &\text{if $r=2$,}\\
2p+1,\quad &\text{if $r=3$,}\\
3p+2,\quad &\text{if $r=4$,}\\
0,\quad &\text{if $r=5$.}\\
\end{cases}$$
Moreover $k_{n}-k_{n-6}=2n$ and $k_n-2k_{n-6}+k_{n-12}=12$.
\end{lem}
\begin{proof}
We give the proof for $r=0$, the reader can check the other completely analogous cases. We have 
$$h^0\bigl(\calO_{Q_n}(3)\bigr)=h^0\bigl(\calO_{\bbP^n}(3)\bigr)-h^0\bigl(\calO_{\bbP^n}(1)\bigr)=\binom{n+3}{3}-\binom{n+1}{1}=\frac{n(n+1)(n+5)}{6}$$
and, since $n=6p$, we get
$$h^0\bigl(\calO_{Q_n}(3)\bigr)=p(6p+1)(6p+5)=p(36p^2+36p+5)=6p(6p^2+6p)+5p=n(6p^2+6p)+5p.$$
Hence, $k_n=6p^2+6p$ and $\delta_n=5p$. Moreover, we have
\begin{gather*}
k_{n}-k_{n-6}=6p^2+6p-\bigl(6(p-1)^2+6(p-1)\bigr)=12p=2n,\\
k_n-2k_{n-6}+k_{n-12}=6p^2+6p-2\bigl(6(p-1)^2+6(p-1)\bigr)+6(p-2)^2+6(p-2)=12.\qedhere
\end{gather*}
\end{proof}
\autoref{lem:num-lem} is the core of the inductive argument: indeed, since $\delta_{n-6}=\delta_n$, we can avoid the arithmetic obstructions by using induction from $n-6$ to $n$.

We prove now the three lemmata which give us \autoref{thm:post-3}.
\begin{lem}\label{lem:cubic1}
Let $n\in\bbN$ with $n\geq 16$ and let $L,M,N$ be three general codimension 6 sections of $Q_n\subset\bbP^n$. Then $h^0\bigl(\calI_{L\cup M\cup N, Q_n}(2)\bigr)=0$.
\end{lem}
\begin{proof}
We proceed by induction on $n$. The statement for $n=16$ can be checked with the help of a computer. Now, suppose that the statement is true for $n$. Let $H$ be a general hyperplane section of $Q_n$ and note that, since by general assumption $H$ cuts $L\cup M\cup N$ in a proper closed subscheme of $L\cup M\cup N$, we have $\Res_H(L\cup M\cup N)=L\cup M\cup N$. Thus, the Castelnuovo sequence of $L\cup M\cup N$ gives
$$\begin{tikzcd}
	0 \arrow[r] & \calI_{L\cup M\cup N,Q_n}(1) \arrow[r] & \calI_{L\cup M\cup N,Q_n}(2) \arrow[r] & \calI_{\Tr_H(L\cup M\cup N),H}(2) \arrow[r] & 0
\end{tikzcd}$$
and, since by general assumption $H=Q_{n-1}$, we get by induction that 
$$h^0\bigl(\calI_{\Tr_H(L\cup M\cup N),H}(2)\bigr)=0.$$
As a consequence, we have $h^0\bigl(\calI_{L\cup M\cup N,Q_n}(2)\bigr)=h^0\bigl(\calI_{L\cup M\cup N,Q_n}(1)\bigr)$. Using again the Castel\-nuovo sequence with respect to $H$ we get
$$\begin{tikzcd}
	0 \arrow[r] & \calI_{L\cup M\cup N,Q_n}(0) \arrow[r] & \calI_{L\cup M\cup N,Q_n}(1) \arrow[r] & \calI_{\Tr_H(L\cup M\cup N),H}(1) \arrow[r] & 0
\end{tikzcd}.$$
Clearly $h^0\bigl(\calI_{L\cup M\cup N,Q_n}(0)\bigr)=0$ and, since $h^0\bigl(\calI_{\Tr_H(L\cup M\cup N),H}(2)\bigr)=0$, then $$h^0\bigl(\calI_{\Tr_H(L\cup M\cup N),H}(1)\bigr)=0.$$ Hence, 
$h^0\bigl(\calI_{L\cup M\cup N,Q_n}(1)\bigr)=0$
and this ends the proof.
\end{proof}
\begin{lem}\label{lem:cubic2}
Let $n\in\bbN$ with $n\geq 16$ and let $L,M,N\subset Q_n$ be three general codimension 6 linear sections of $Q_n$. For $i=1,\dots,12$ let $l_i$ (resp. $m_i$, $n_i$) be points in general position in $L$ (resp. $M$, $N$) and let 
$$X=\bigcup_{i=1}^{12}(2_{Q_n}l_i\cup2_{Q_n}m_i\cup2_{Q_n}n_i).$$
Then $h^0\bigl(\calI_{X\cup L\cup M\cup N, Q_n}(3)\bigr)=0$.
\end{lem}
\begin{proof}
We proceed by induction on $n$. The statement for $n=16$ can be checked with the help of a computer. Now, fix $n\in\bbN$ with $n\geq 17$ and suppose that the statement is true for $n-1$. Let $H$ be a general hyperplane section of $Q_n$ and suppose that $X$ is such that 
$$l_i\in L\cap H,\, m_i\in M\cap H,\, n_i\in N\cap H,\,i=1,\dots,12$$
with $l_i$, resp. $m_i$ and $n_i$, in general position in $L\cap H$, resp. $M\cap H,$ and $N\cap H$.
By semicontinuity of cohomology, it is enough to prove the statement for such a specialised $X$. Since \linebreak $\Res_{H}(X\cup L\cup M\cup N)= L\cup M\cup N$, the Castelnuovo sequence of $X$ with respect to $H$ gives
$$\begin{tikzcd}
	0 \arrow[r] & \calI_{L\cup M\cup N,Q_n}(2) \arrow[r] & \calI_{X\cup L\cup M\cup N,Q_n}(3) \arrow[r] & \calI_{\Tr_H(X\cup L\cup M\cup N),H}(3) \arrow[r] & 0
\end{tikzcd}.$$
Note that, by general assumption, we have $H=Q_{n-1}$ and thus, by induction, we have 
$$h^0\bigl(\calI_{\Tr_H(X\cup L\cup M\cup N),H}(3)\bigr)=0,$$
and, by \autoref{lem:cubic1}, we have 
$$h^0\bigl(\calI_{L\cup M\cup N, Q_n}(2)\bigr)=0.$$
Hence, $h^0\bigl(\calI_{X\cup L\cup M\cup N, Q_n}(3)\bigr)=0$.
\end{proof}
\begin{lem}\label{lem:cubic3}
Let $n\in\bbN$ with $n\geq 12$ and let $L,M\subset Q_n$ be two general codimension 6 linear sections of $Q_n$. Consider $l_i\in L$ points in general position on $L$ and $m_i\in M$ points in general position on $M$ for $i=1,\dots,2n-12$, and $q_j\in Q_n$ points in general position on $Q_n$ with $j=1,\dots,12$. Let $X\subset Q_n$ the 0-dimensional subscheme defined as
$$X=\bigcup_{i=1}^{2n-12}(2_{Q_n}l_i\cup2_{Q_n}m_i)\cup\bigcup_{j=1}^{12}2_{Q_n}q_j.$$
Then, $h^0\bigl(\calI_{X\cup L\cup M, Q_n}(3)\bigr)=0$.
\end{lem}
\begin{proof}
We proceed by induction from $n-6$ to $n$. The statement for $n=12,\dots,17$ can be checked with the help of a computer. Now, fix $n\in\bbN$ with $n\geq 18$ and suppose by inductive hypotesis the statement true for $n-6$. Let $N\subset Q_n$ be a general codimension 6 linear section of $Q_n$. Write $X=X_1\cup X_2$ with 
$$X_1=\bigcup_{i=1}^{2n-24}(2_{Q_n}l_i\cup2_{Q_n}m_i),\qquad X_2=\bigcup_{i=2n-23}^{2n-12}(2_{Q_n}l_i\cup2_{Q_n}m_i)\cup\bigcup_{j=1}^{12} 2_{Q_n}q_i$$
and suppose that:
\begin{itemize}[leftmargin=*]
    \item $l_i\in L\cap N$ for $i=1,\dots, 2n-24$ and they are in general position in $L\cap N$;
    \item $m_i\in M\cap N$ for $i=1,\dots, 2n-24$ and they are in general position in $M\cap N$.  
    \item $l_i\in L$ are in general position in $L$ for $i=2n-23,\dots,2n-12$;
    \item $m_i\in M$ are in general position in $M$ for $i=2n-23,\dots,2n-12$;
    \item $q_j\in N$ are in general position in $N$ for $j=1,\dots,12$;  
    \end{itemize}
By semicontinuity of cohomology, it is enough to prove the statement for such a special $X$. \linebreak Since $X_2$ is made of $12$ double points of $Q_n$ in general position on $L$, $12$ double points of $Q_n$ in general position on $M$ and $12$ double points of $Q_n$ in general position on $N$, by \autoref{lem:cubic2} we get $h^0\bigl(\calI_{X_2\cup L\cup M\cup N,Q_n}(3)\bigr)=0$, and thus $h^0\bigl(\calI_{X\cup L\cup M\cup N,Q_n}(3)\bigr)=0$. Moreover, $N=Q_{n-6}$ and $X\cap N$ is made of $2(n-6)-12$ double points of $N$ in general position on $L\cap N$, $2(n-6)-12$ double points of $N$ in general position on $M\cap N$ and 12 double points of $N$ in general position on $N$. Thus, by inductive assumption, we have $h^0\bigl(\calI_{(X\cup L\cup M)\cap N,N}(3)\bigr)=0$. Hence, the exact sequence \eqref{exc} applied to $X\cup L\cup M$ and $N$ gives
\begin{equation*}
\begin{tikzcd}
	0 \arrow[r] & \calI_{X\cup L\cup M\cup N,Q_n}(3) \arrow[r] & \calI_{X\cup L\cup M,Q_n}(3) \arrow[r] & \calI_{(X\cup L\cup M)\cap N,N}(3) \arrow[r] & 0
\end{tikzcd}
\end{equation*}
and we get $h^0\bigl(\calI_{X\cup L\cup M,Q_n}(3)\bigr)=0$.
\end{proof}
\begin{lem}\label{lem:cubic4}
Let $n\in\bbN$ with $n\geq 6$ and let $L\subset Q_n$ be a general codimension $6$ linear section of $Q_n$. Consider $l_i\in L$ points in general position on $L$  for $i=1,\dots,k_{n-6}$, $q_j\in Q_n$ points in general position on $Q_n$ for $j=1,\dots,2n$, and $\eta\subset Q_n$ a 0-dimensional subscheme supported at $p\in L$ and such that $\eta\subset2_{Q_n}p$, $\ell(\eta)=\delta_n$ and $\ell(\eta\cap L)=\delta_n-\delta_{n-6}$. Let $X\subset Q_n$ the subscheme defined as
$$X=\bigcup_{i=1}^{k_{n-6}}2_{Q_n}l_i\cup\bigcup_{j=1}^{2n}2_{Q_n}q_i\cup\eta.$$
Then, $h^0\bigl(\calI_{X\cup L,Q_n}(3)\bigr)=0$.
\end{lem}
\begin{proof}
We proceed by induction from $n-6$ to $n$. The statement for $n=6,\dots,11$ can be checked with the help of a computer. Now fix $n\in\bbN$ with $n\geq 12$ and suppose by inductive hypotesis the statement true for $n-6$. Let $M\subset Q_n$ be a general codimension 6 linear section of $Q_n$. Write $X=X_1\cup X_2$ with 
$$X_1=\bigcup_{i=1}^{k_{n-12}}2_{Q_n}l_i\cup\bigcup_{j=1}^{2n-12}q_j\cup\eta,\qquad X_2=\bigcup_{i=k_{n-12}+1}^{k_{n-6}}2_{Q_n}l_i\cup \bigcup_{j=2n-13}^{2n}2_{Q_n}q_j $$
and suppose that:
\begin{itemize}[leftmargin=*]
\item $l_i\in L\cap M$ for $ i=1,\dots,k_{n-12}$ and they are in general position on $L\cap N$;
\item $q_j\in M$ for $j=1,\dots,2n-12$ and they are in general position on $M$;
\item $l_i\in L$ for $i=k_{n-12}+1,\dots,k_{n-6}$ and they are in general position on $L$;
\item $q_j\in Q_n$ for $j=2n-13,\dots,2n$ and they are in general position on $Q_n$;
\item $\eta$ supported at $p\in L\cap M$ such that $\ell(\eta\cap M)=\delta_{n-6}$ and $\ell(\eta\cap L\cap M)=\delta_{n-6}-\delta_{n-12}$.
\end{itemize}
By semicontinuity of cohomology, it is enough to prove the statement for such a special $X$. By \autoref{lem:num-lem} we have $k_{n-6}-k_{n-12}=2n-12$. Thus, $X_2$ is made of $2n-12$ double points of $Q_n$ in general position on $L$ and $12$ double points of $Q_n$ in general position on $Q_n$. Hence, by \autoref{lem:cubic3} we get $h^0\bigl(\calI_{X_2\cup L\cup M,Q_n}(3)\bigr)=0$, and thus $h^0\bigl(\calI_{X\cup L\cup M,Q_n}(3)\bigr)=0$. \linebreak Moreover, $M=Q_{n-6}$ and $X\cap M$ is made of $k_{n-12}$ double points of $M$ in general position on $L\cap M$, $2n-12=2(n-6)$ double points of $M$ in general position on $M$ and the scheme $\eta\cap M$ such that $\ell(\eta\cap M)=\delta_{n-6}$ and $\ell(\eta\cap L\cap M)=\delta_{n-6}-\delta_{n-12}$. Thus, by inductive assumption, we have $h^0\bigl(\calI_{(X\cup L)\cap M,M}(3)\bigr)=0$. Hence, the exact sequence \eqref{exc} applied to $X\cup L$ and $M$ gives
\begin{equation*}
\begin{tikzcd}
	0 \arrow[r] & \calI_{X\cup L\cup M,Q_n}(3) \arrow[r] & \calI_{X\cup L,Q_n}(3) \arrow[r] & \calI_{(X\cup L)\cap M,M}(3) \arrow[r] & 0
\end{tikzcd}
\end{equation*}
and we get $h^0\bigl(\calI_{X\cup L,Q_n}(3)\bigr)=0$.
\end{proof}
We are finally ready to prove \autoref{thm:post-3}.
\begin{proof}[Proof of \autoref{thm:post-3}]
We proceed by induction from $n-6$ to $n$. The statement for $n=1$ is trivial and for $n=2,\dots,6$ can be checked with the help of a computer. Now fix $n\in\bbN$ with $n\geq 7$ and suppose by inductive hypothesis the statement true for $n-6$. We want to prove that, if $X\subset Q_n$ is a general union of double points, then $X$ has good postulation. By \autoref{rem:rangè} we can suppose that $X$ is the union of $k_n$ general double points in general position and a scheme $\eta$ supported at $p$ such that $\ell(\eta)=\delta_{n}$ and $\eta\subset2_{Q_n}p$, and we have to prove that $h^0\bigl(\calI_{X,Q_n}(3)\bigr)=0$. Let $q_1,\dots,q_{k_n}$ be the supports of the double points of $X$ and $L\subset Q_n$ be a general codimension 6 linear section of $Q_n$. Suppose that:
\begin{itemize}[leftmargin=*]
\item $q_i\in L$ for $i=1,\dots,k_{n-6}$ and they are in general position on $L$;
\item $q_i\in Q_n$ for $i=k_{n-6}+1,\dots,k_n$ and they are in general position in $Q$;
\item $p\in L$ and $\ell(\eta\cap L)=\delta_{n-6}.$
\end{itemize}
By semicontinuity of cohomology it is enough to prove the statement for such a special $X$. Since in $X$ we have $k_{n-6}$ double points in general position on $L$ and $k_n-k_{n-6}=2n$ (see \autoref{lem:num-lem}) double points in general position on $Q$, by \autoref{lem:cubic4}, we have $h^0\bigl(\calI_{X\cup L,Q_n}(3)\bigr)=0$. Morever, $L=Q_{n-1}$ and $X\cap L$ is made of $k_{n-6}$ double points in general position on $L$ and $\eta\cap L\subseteq 2_Lp$ with $\ell(\eta\cap L)=\delta_{n-6}$. Thus, by inductive assumption, we have $h^0\bigl(\calI_{X\cap L,L}(3)\bigr)=0$. Hence, the exact sequence \eqref{exc} applied to $X$ and $L$ gives
\begin{equation*}
\begin{tikzcd}
	0 \arrow[r] & \calI_{X\cup L,Q_n}(3) \arrow[r] & \calI_{X,Q_n}(3) \arrow[r] & \calI_{X\cap L,L}(3) \arrow[r] & 0
\end{tikzcd}
\end{equation*}
and we get $h^0\bigl(\calI_{X,Q_n}(3)\bigr)=0$.
\end{proof}
\begin{teo}
Let $n\in\bbN$ with $n\geq 3$. Then $\dim\sigma_r(\Isot_n^3)=\expdim\sigma_r(\Isot_n^3)$.
\end{teo}
\begin{proof}
The statement follows by \autoref{cor:equivalence} and \autoref{thm:post-3}.
\end{proof}
\subsection{The solution of the interpolation problem for \texorpdfstring{$d\geq4$}{dgeq4}}
Before giving the proof of \Cref{theo:AH_harmonics}, we need the following numerical lemma. 
\begin{lem}\label{lem:num2}
Let $d,n\in\bbN$ with $d,n\geq 4$, and let $r,u,\varepsilon\in\bbN$ such that $r\leq\bigl\lceil{{h^0\bigl(\calO_{Q_n}(d)\bigr)}/{n}}\bigr\rceil$, $0\leq\varepsilon<n-1$ and $rn-(n-1)u-\varepsilon=h^0\bigl(\calO_{Q_n}(d-1)\bigr)$. Then,
\begin{enumerate}[leftmargin=*]
\item $(n-1)\varepsilon+u\leq h^0\bigl(\calO_{Q_{n-1}}(d-1)\bigr)$;
\item $h^0\bigl(\calO_{Q_n}(d-2)\bigr)\leq (r-u-\varepsilon)n$;
\item $r-u-\varepsilon\geq 0$
\item $u\geq\varepsilon$ if $r\geq\bigl\lfloor{{h^0\bigl(\calO_{Q_n}(d)\bigr)}/{n}}\bigr\rfloor$;
\item $r-u-\varepsilon\geq n+1$ if $d=4$ and $n\geq 9$.
\end{enumerate}
\end{lem}
\begin{proof}
We set 
$$f_{n,d}\coloneqq h^0\bigl(\calO_{Q_n}(d)\bigr)=\binomial{n+d}{n}-\binomial{n+d-2}{n}.$$
(1): by definition of $u$ and $\eps$ we have
\begin{gather*}
u=\frac{1}{n-1}(rn-f_{n,d-1}-\eps)\leq\frac{1}{n-1}(f_{n,d}+n-f_{n,d-1})=\frac{1}{n-1}(f_{n-1,d}+n)
\end{gather*}
and thus
$$(n-1)\eps+u\leq(n-1)(n-2)+\frac{n}{n-1}+\frac{f_{n-1,d}}{n-1}.$$
A sufficient condition for (1) to hold is that
\begin{equation}\label{disq}
(n-1)(n-2)+\frac{n}{n-1}\leq \underbrace{f_{n-1,d-1}-\frac{f_{n-1,d}}{n-1}}_{g_{n,d}}.
\end{equation}
We have
$$g_{n,d}-g_{n,d-1}=f_{n-2,d-1}-\frac{f_{n-2,d}}{n-1}=\binomial{n+d-5}{n-2}\frac{(n-2)(d(2n-4)+n^2-7n+12)}{d(d-2)(n-1)}$$
and the right-hand side of the equation is non-negative for any $n\geq 4$ and $d\geq 4$. Hence, $g_{n,d}$ is increasing in $d$ for any $d\geq 4$ and $n\geq 4$. As a consequence, it is enough to check \eqref{disq} for $d=4$ and, for this value of $d$, \eqref{disq} is equivalent to
$$3n^4+69n^2-22n^3-122n-48\geq0$$
which is always verified for $n\geq 4$.\\
(2): we have
$$(r-u-\eps)n=rn-un-\eps n=f_{n,d-1}-u-(n-1)\eps$$
and thus
$$f_{n,d-2}\leq(r-u-\eps)n\ \Leftrightarrow\ \underbrace{f_{n,d-2}-f_{n,d-1}}_{-f_{n-1,d-1}}\leq -u-(n-1)\eps\ \Leftrightarrow\ f_{n-1,d-1}\geq u+(n-1)\eps$$
and this is true by part (1).\\
(3): it is an immediate consequence of (2).\\
(4): it suffices to show that $u\geq n-2$. By definition, we have
$u=\left\lfloor\frac{rn-f_{n,d-1}}{n-1}\right\rfloor$, so
it is enough to show that
$$\frac{rn-f_{n,d-1}}{n-1}-1\geq n-2,$$
which is equivalent to
$$rn\geq f_{n,d-1}+(n-1)^2.$$
Since $r\geq \left\lfloor\frac{f_{n,d}}{n}\right\rfloor$, we have $rn\geq f_{n,d}-n$, and thus it suffices to show that $$f_{n,d}-n\geq f_{n,d-1}+(n-1)^2$$ 
and this is equivalent to
$$\underbrace{f_{n,d}-f_{n,d-1}}_{f_{n-1,d}}\geq (n-1)^2+n$$
As in case (1), since $f_{n-1,d}$ is increasing in $d$ for any $n\geq 4$, it is enough to check that \[f_{n-1,4}\geq(n-1)^2+n,
\]and this is equivalent to
$$n^4+6n^3-25n^2+18n-24\geq0$$
which holds for any $n\geq 4$.\\
(5): We have
\begin{align*}
r-u-\eps&=\frac{1}{n-1}\bigl(f_{n,3}-r-(n-2)\eps\bigr)\geq\frac{1}{n-1}\left(f_{n,3}-\frac{f_{n,4}}{n}-1-(n-2)^2\right)\\
&=\frac{3n^3-10n^2+93n-134}{24(n-1)}
\end{align*}
and 
$$\frac{3n^3-10n^2+93n-134}{24(n-1)}\geq n+1$$
for any $n\geq9 $.
\end{proof}
Now we denote by $C_{n,d}(r)$ the following statement:
$$C_{n,d}(r)\coloneqq\text{ A general union of $r$ double points of $Q_n$ has good postulation in degree $d$}.$$
We are now ready to prove the following theorem.
\begin{teo}\label{teo:bigind}
Let $n,d\in\bbN_{\geq 4}$ and let $r,u,\varepsilon\in\bbN$ such that $$\left\lfloor{\frac{h^0\bigl(\calO_{Q_n}(d)\bigr)}{n}}\right\rfloor\leq r\leq\left\lceil{\frac{h^0\bigl(\calO_{Q_n}(d)\bigr)}{n}}\right\rceil,$$ $0\leq\varepsilon<n-1$ and $rn-(n-1)u-\varepsilon=h^0\bigl(\calO_{Q_n}(d-1)\bigr)$. If $C_{n-1,d}(u)$, $C_{n,d-1}(r-u)$, and $C_{n,d-2}(r-u-\varepsilon)$ hold, then $C_{n,d}(r)$ holds.
\end{teo}
\begin{proof}
To prove the statement, it is enough to construct a union of $r$ double points on $Q_n$ having good postulation in degree $d$. We fix a general hyperplane section $Q_{n-1}\in|\calO_{Q_n}(1)|$ and we consider the following schemes
\begin{gather*}
X_1=2_{Q_n}P_1\cup\dots 2_{Q_n}P_u, \quad X_2=2_{Q_n}P_{u+1}\cup\dots\cup2_{Q_n}P_{r-\varepsilon},\quad X=X_1\cup X_2\\
Z=2_{Q_n}P_{r-\varepsilon+1}\cup\dots\cup 2_{Q_n}P_r, \quad Z'=2_{Q_n}P'_{r-\varepsilon+1}\cup\dots\cup 2_{Q_n}P'_r,\qquad\text{ if $\varepsilon\geq 1$,}\\
Z=\emptyset, \quad Z'=\emptyset,\qquad\text{ if $\varepsilon=0$,}
\end{gather*}
where $P_1,\dots,P_r,P'_{r-\varepsilon+1},\dots,P'_r$ are points of $Q_n$ such that
\begin{itemize}[leftmargin=*]
\item $P_1,\dots,P_u\in Q_{n-1}$;
\item $P_{u+1},\dots,P_r$ are in general position on $Q_n$;
\item $P'_{r-\varepsilon+1},\dots,P'_r\in Q_{n-1}$;
\item $P_1,\dots,P_u,P'_{r-\varepsilon+1},\dots,P'_r\in Q_{n-1}$ are in general position on $Q_{n-1}$.
\end{itemize}
Note that $X_1,X_2,Z,Z'$ are well-defined by \autoref{lem:num2}, (3). We have
$$\Tr_{Q_{n-1}}(X)=2_{Q_{n-1}}P_1\cup\dots\cup2_{Q_{n-1}}P_u,\quad \Res_{Q_{n-1}}(X)=P_1\cup\dots\cup P_u\cup X_2$$
and, if we take $\bfp=(1,\dots,1)\in\bbN^{\varepsilon},$
we have
\begin{gather*}
\Tr_{Q_{n-1}}^{\bfp}Z'=P'_{r-\varepsilon+1}\cup\dots\cup P'_r,\quad \Res_{Q_{n-1}}^{\bfp}Z'=2_{Q_{n-1}}P'_{r-\varepsilon+1}\cup 2_{Q_{n-1}}P'_r,\quad\text{if $\varepsilon\geq 1$}\\
\Tr_{Q_{n-1}}^{\bfp}Z'=\emptyset,\quad \Res_{Q_{n-1}}^{\bfp}Z'=\emptyset,\quad\text{if $\varepsilon=0$.}
\end{gather*}
Recall that the (differential) traces are subschemes of $Q_{n-1}$ while the (differential) residues are subschemes of $Q_n$. In particular $\Res_{Q_{n-1}}^{\bfp}Z'$ is a union of $\varepsilon$ double points of $Q_{n-1}$ which are embedded in $Q_n$ and whose supports are general in $Q_{n-1}$. Note that $X_2\cup \Res_{Q_{n-1}}^{\bfp}Z'$ is contained in a general union $W$ of $r-u$ double points of $Q_n$. By $C_{n,d-1}(r-u)$ we have
$$h^0\bigl(\calI_{W,Q_{n}}(d-1)\bigr)=\max\{0,h^0\bigl(\calO_{Q_n}(d-1)\bigr)-(r-u)n\}=\max\{0,u-\varepsilon\}=u-\varepsilon$$
where the first equality follows by the definition of $u$ and $\varepsilon$ and the second one follows by \autoref{lem:num2} (4). As a consequence, by \autoref{rem:rangè}, we get
$$h^0\bigl(\calI_{X_2\cup \Res_{Q_{n-1}}^{\bfp}Z' }(d-1)\bigr)=h^0\bigl(\calI_{W,Q_{n}}(d-1)\bigr)+\varepsilon=u.$$
Moreover, by $C_{n,d-2}(r-u-\varepsilon)$ we have 
$$h^0\bigl(\calI_{X_2,Q_{n}}(d-2)\bigr)=\max\{0,h^0\bigl(\calO_{{Q_n}}(d-2)\bigr)-(r-u-\varepsilon)n\}=0,$$
where the last equality follows by \autoref{lem:num2} (2).
Since $h^0\bigl(\calI_{X_2,Q_n}(d-2)\bigr)=0$, the base locus of $|\calI_{X_2\cup \Res_{Q_{n-1}}^{\bfp}Z' }(d-1)|$ cannot contain $Q_{n-1}$. As a consequence, since $P_1\cup\dots\cup P_u$ is a general union of simple points on $Q_{n-1}$, we have 
$$h^0\bigl(\calI_{\Res_{Q_{n-1}}(X)\cup\Res_{Q_{n-1}}^{\bfp}(Z')}(d-1)\bigr)=u-u=0.$$
Hence, by \autoref{HoraceDiff}, we have
$$h^0\bigl(\calI_{X\cup Z}(d)\bigr)=h^0\bigl(\calI_{\Tr_{Q_{n-1}}(X)\cup\Tr_{Q_{n-1}}^{\bfp}(Z'),Q_{n-1}}(d)\bigr)$$
and, by $C_{n-1,d}(u)$, we have
\begin{align*}
h^0\bigl(\calI_{\Tr_{Q_{n-1}}(X)\cup\Tr_{Q_{n-1}}^{\bfp}(Z'),Q_{n-1}}(d)\bigr)&=\max\{0,h^0\bigl(\calO_{Q_{n-1}}(d)\bigr)-u(n-1)-\varepsilon\}\\
&=\max\{0,h^0\bigl(\calO_{Q_{n}}(d)\bigr)-rn\}.\qedhere
\end{align*}
\end{proof}
We are ready to give the proof of \Cref{theo:AH_harmonics}, which, by \Cref{cor:equivalence} can be formulated in the following equivalent statement.
\begin{teo}\label{teo:post_fin}
Let $n,r,d\in \bbN_{\geq 1}$ and let $X\subset Q_n\subset\bbP^n$ be a general union of $r$ double points of $Q_n$. If $d\geq 3$, then
$$h^0\bigl(\calI_{X,Q_n}(d)\bigr)=\max\bigl\{0,h^0\bigl(\calO_{Q_n}(d)\bigr)-rn\bigr\},$$
while, if $d=2$, then
$$h^0\bigl(\calI_{X,Q_n}(2)\bigr)=\begin{cases}h^0\bigl(\calO_{Q_n}(2)\bigr)-rn+\dbinom{r-1}{2},&\text{if $1\leq r\leq n$,}\\[2ex]
0,&\text{if $r\geq n+1$.}\end{cases}$$
\begin{proof}
The case $d=1$ is trivial, the case $d=2$ is an immediate consequence of \autoref{prop:interpol-problem} and \autoref{prop:dim_isotropic_secants_quadrics} and the case $d=3$ is \autoref{thm:post-3}. The cases $n=1,2$ are trivial and the case $n=3$ follows by the fact that $Q_3\cong \bbP^1\times\bbP^1$, $\calO_{Q_3}\otimes\calO_{\bbP^3}(d)\cong\calO_{\bbP^1\times\bbP^1}(d,d)$  and, by \cite{LP13}*{Theorem 3.1} and \cite{Laf02}*{Table I}, a general union of double points on $\bbP^1\times\bbP^1$ has always good postulation with respect to $\calO_{\bbP^1\times\bbP^1}(d,d)$ for any $d\geq 4$. Thus, we can suppose, from now on, that $d\geq 4$ and $n\geq 4$. We use the same notation of \autoref{teo:bigind} and, by \autoref{rem:rangè}, we can suppose
$$ \left\lfloor{\frac{h^0\bigl(\calO_{Q_n}(d)\bigr)}{n}}\right\rfloor\leq r\leq\left\lceil{\frac{h^0\bigl(\calO_{Q_n}(d)\bigr)}{n}}\right\rceil.$$
If $d=4$ and $4\leq n\leq 6$, the result can be checked by using a computer. For $d=4$ and $n\geq 9$, suppose  by induction that $C_{n-1,4}(u)$ holds. For $n=7,8$, an easy check shows that $r-u-\eps\geq n+1$ and, by \autoref{lem:num2}, $r-u-\eps\geq n+1$ for any $n\geq 9$. Thus, under our assumption, $r-u-\eps\geq n+1$ for any $n\geq 7$.  We have that $C_{n,3}(r-u) $ holds because there are no exceptional cases for $d=3$ and $C_{n,2}(r-u-\eps)$ holds because $r-u-\eps\geq n+1$ . Hence, by \autoref{teo:bigind}, also $C_{n,4}(r)$ holds. The result for $n\geq4$ and $d\geq 5$ follows by induction on $n$ and $d$ by using \autoref{teo:bigind} on the base cases $(n,d)=(3,d),(n,3),(n,4)$.
\end{proof}
\end{teo}

\section{Binary and ternary harmonic forms}\label{sec:ternary_forms}
Decompositions of harmonic binary forms are trivial. Indeed, for any non-degenerate quadratic form $\omega_1\in\calD_{1,2}$, we have
\[
\dim\calH_{1,2}=2.
\]
Fixing, for the sake of simplicity, the quadric $\omega_1=\alpha_0^2+\alpha_1^2$, there are only two $\omega_1$-isotropic points, which are the forms
\[
u\coloneqq -\frac{x_0+\rmi x_1}{2},\quad v\coloneqq \frac{x_0-\rmi x_1}{2},
\]
and we have $\calH_{1,2}=\langle u^2,v^2\rangle$. Hence, \Cref{prop:harmonic_binary_forms} immediately follows and, in particular, we have $\irk(au^2+bv^2)=2$ for any $a,b\in\bbC\setminus\{0\}$.

The case of harmonic ternary quadratic forms is strictly connected with binary forms in $\calR_{2,d}$. In particular, determining the isotropic rank of a harmonic ternary form is exactly the same as determining the Waring rank of a binary form. A relevant fact concerning the spaces of harmonic polynomials is that, for every $n\in\bbN$, the space $\calH_{n,d}$ is an irreducible $\SO_{n+1}(\bbC)$-module (see e.g.~\cite{GW98}*{Theorem 5.2.4}). In particular, for the case $n=2$, we have that the Lie algebra $\mfso_3(\bbC)$ is canonically isomorphic to $\mfsl_2(\bbC)$. By the uniqueness of the irreducible representations of $\mfsl_2(\bbC)$ (see e.g.~\cite{FH91}*{section 11.1}), we have $\calH_{2,d}\cong\calR_{2,2d}\cong\bbC[s,t]_{2d}$. 
Note that, from a more geometric perspective, this isomorphism is a consequence of the fact that $\Isot_{2}$ is a smooth conic of $\bbP^2$ and thus it is isomorphic to $\bbP^1$. Thus, we have
$$\Isot_{2,d}=\nu_{2,d}(\Isot_2)\cong\nu_{2,d}\bigl(\nu_{1,2}(\bbP^1)\bigr)=\nu_{2,2d}(\bbP^1).$$
We want now to make this isomorphism explicit from an algebraic point of view.
Assuming $\omega_2=\alpha_0^2+\alpha_1^2+\alpha_2^2$,
we set
\[
u\coloneqq -\frac{x_0+\rmi x_1}{2},\quad v\coloneqq \frac{x_0-\rmi x_1}{2},\quad z\coloneqq x_2,
\]
so that \[
q_2=x_0^2+x_1^2+x_2^2=z^2-4uv.
\]
For the next formulas we use the divided powers, introducing the notation
\[
u^{[k]}\coloneqq\frac{1}{k!}u^k,\quad v^{[k]}\coloneqq\frac{1}{k!}v^k,\quad z^{[k]}\coloneqq\frac{1}{k!}z^k,
\]
for every $k\in\bbN$.
The canonical isomorphism is obtained by assigning the element
\[
h_{d,k}\coloneqq {\binom{2d}{k+d}}^{-1}\sum_{j=0}^{\pint*{\frac{d-\abs*{k}}{2}}}u^{\pq*{\frac{\abs*{k}+k}{2}+j}}v^{\pq*{\frac{\abs*{k}-k}{2}+j}}z^{\pq*{d-\abs*{k}-2j}}
\]
to the monomial $s^{d-k}t^{d+k}\in S^{2d}V$, for every
$k=-d,\dots,d$, see e.g.~\cite{Fla24}*{Proposition 5.25} for a direct proof.
In particular, we have the identifications
\begin{equation}\label{formula:identification_ternary_forms}
s^2\sim v,\qquad t^2\sim u,\qquad st\sim \frac{1}{2}z.
\end{equation}
In other words, the isomorphism $\calH_{2,d}\simeq\calR_{1,2d}$ is defined on monomials by restricting to harmonic polynomials the following surjective map: 
\[
\begin{tikzcd}[row sep=0pt,column sep=1pc]
 \bfbeta_d\colon \calR_{2,d}\arrow{r} & \calR_{1,2d}\hphantom{,} \\
  {\hphantom{\bfbeta_d\colon{}}} u^{a}v^{b}z^{d-a-b} \arrow[mapsto]{r} & {2^{d-a-b}}{d!}s^{d+a-b}t^{d-a+b},
\end{tikzcd}\qquad 0\leq a,b,a+b\leq d.
\]
In particular, observe that 
\[
\bfbeta_{2}(q_2)=\bfbeta_2(z^2-4uv)=4x^2y^2-4x^2y^2=0.
\]
Then, by dimensional reasons, we have $\bfbeta_d(\calH_{2,d})\simeq\calR_{1,2d}$ for every $d\in\bbN$.
The following trivial fact allows to get explicitly the decomposition as sum of isotropic linear forms from decompositions of binary forms.
\begin{prop}
    Let $\ell\in\calR_{2,1}$. Then $\ell$ is isotropic if and only if $\bfbeta_2(\ell)=\rmm^2$ for some $\rmm\in\calR_{1,1}$.
\end{prop}
\begin{proof}
    The linear form $a_0u+a_1v+a_2z$ is isotropic if and only if $a_2^2=4a_0a_1$. We have
    \begin{align*}
    \bfbeta_2(a_0u+a_1v+a_2z)&=a_0x^2+a_1y^2+\dfrac{1}{2}a_2xy=a_0x^2+a_1y^2\pm 2(a_0a_1)^{1/2}xy\\[1ex]
    &=\bigl(a_0^{1/2}x\pm a_1^{1/2}y\bigr)^2.\qedhere
    \end{align*}
\end{proof}

Using the isomorphism of harmonic ternary forms and binary forms, we can take advantage of the Sylvester algorithm \cite{Syl51}, for the case of even degree binary forms, to determine the isotropic rank of every harmonic ternary form (see also \cites{CGLM08,CS11} for a more modern point of view).
Observe that
\[
\pdv{}{u}=-\pdv{}{x_0}+\rmi\pdv{}{x_1},\qquad \pdv{}{v}=\pdv{}{x_0}+\rmi\pdv{}{x_1},\qquad \pdv{}{z}=\pdv{}{x_2},
\]
and also
\[
\Delta=\pdv[2]{}{x_0}+\pdv[2]{}{x_2}+\pdv[2]{}{x_2}=\pdv[2]{}{z}-\pdv[2]{}{u}{v}.
\]
Moreover, if $k>0$, a simple calculation shows that
\[
\pdv{h_{d,k}}{u}=\frac{(d+k)(d+k-1)}{2d(2d-1)}h_{d-1,k-1},\qquad \pdv{h_{d,k}}{v}=\frac{(d+k)(d+k-1)}{2d(2d-1)}h_{d-1,k+1}.
\]
This implies that also apolarity is preserved. In particular, we have
\[
\pdv{}{u}\sim\frac{1}{2d(2d-1)}\pdv[2]{}{y},\quad \pdv{}{v}\sim\frac{1}{2d(2d-1)}\pdv[2]{}{x},\quad \pdv{}{z}\sim\frac{1}{2d(2d-1)}\pdv[2]{}{x}{y}.
\]
Therefore, using the aforementioned isomorphism and \cite{BCC+18}*{Algorithm 2}, we can determine the isotropic rank of any harmonic ternary form. Moreover, by \cite{CS11}*{Theorem 2} we obtain \Cref{teo:harmonic_Comas_Seiguer}.
\begin{exam}\label{exam:rank_vz}
    Let us consider the harmonic polynomial \[
    vz=(x_0-\rmi x_1)x_2,
    \]
    which is equivalent, as a binary form, to $\bfbeta_2(vz)=2xy^3$. A minimal decomposition is given by
    \begin{align*}
    2xy^3&=\frac{1}{8}(x+y)^4+\frac{\rmi}{8}(x+\rmi y)^4-\frac{1}{8}(x-y)^4-\frac{\rmi}{8}(x-\rmi y)^4\\[1ex]
    &=\frac{1}{8}(x^2+y^2+2xy)^2+\frac{\rmi}{8}(x^2-y^2+2\rmi xy)^2-\frac{1}{8}(x^2+y^2-2xy)^2-\frac{\rmi}{8}(x^2-y^2-2\rmi xy)^2
    \end{align*}
    Then, using the identification \eqref{formula:identification_ternary_forms}, we obtain
\[
    vz=\frac{1}{8}(u+v+z)^2+\frac{\rmi}{8}(u-v+\rmi z)^2-\frac{1}{8}(u+v-z)^2-\frac{\rmi}{8}(u-v-\rmi z)^2.\qedhere
\]
\end{exam}

\section{Harmonic quadratic forms}
We denote by $\Oa_{\omega_n}(\bbC)$ the orthogonal group with respect to the quadratic form $\omega_n$. In the case where $\omega_n=x_0^2+\cdots+x_n^2$ we write $\Oa_{n+1}(\bbC)$.
For the sake of simplicity, throughout this section we assume that $\omega_n=\alpha_0^2+\cdots+\alpha_n^2$, so that by \textit{harmonic} we mean $\omega_n$-harmonic. Also, once fixed the dual bases $\{x_0,\dots,x_n\}$
 and $\{\alpha_0,\dots,\alpha_n\}$, we consider any quadratic form as a matrix in the classical way by setting, for any quadratic form 
 \[
 q=\sum_{i=0}^na_{ii}x_i^2+2\sum_{i<j}a_{ij}x_ix_j,
 \] 
 the corresponding symmetric matrix 
 \[
M_q=\begin{pNiceMatrix}
a_{00}&\Cdots&a_{n1}\\
\Vdots&\Ddots&\Vdots\\
a_{n1}&\Cdots&a_{nn}
\end{pNiceMatrix}.
 \]
 
 If there is no risk of confusion, in this section we treat the elements of $\calR_{n,2}$ and $\calD_{n,2}$ as $(n+1)$-square symmetric matrices and the elements of $\calR_{n,1}$ as vectors of $\bbC^{n+1}$, whose coordinates are given by the coefficients of the corresponding linear form. In particular, by a \textit{harmonic matrix}, we mean a matrix whose corresponding quadratic form is harmonic. Note that a matrix is harmonic if and only if it is a trace-zero matrix. 
 
 Observe that, in the language of matrices, a Waring decomposition of a matrix $A\in\calR_{n,2}$ is a (not necessarily square) matrix $L\in\bbC^{n,r}$, for some $r\in\bbN$ such that $A=L\transpose L$. In particular, 
the columns of $L$ consist of the coefficients of the linear forms of such a decomposition with respect to the basis $\{x_0,\dots,x_n\}$. Also, any square of a linear form corresponds to a symmetric matrix $v\transpose v$, for some $v\in\bbC^{n+1}$. We denote by $L^{(i)}$ the $i$-th column of $L$.
\begin{rem}\label{rem:isotropic_rank_equal}
Since the orthogonal group $\Oa_{n+1}(\bbC)$ preserves isotropy of points, the isotropic rank is invariant under the action of $\Oa_{n+1}(\bbC)$. In particular, two orthogonally similar symmetric matrices have the same isotropic rank.
\end{rem}
 
 The next lemma translates the property of a full rank harmonic matrix to have the isotropic rank equal to the Waring rank in the languages of matrices. According to \cite{Gen07}*{p.~42}, we say that a square matrix is a \textit{hollow matrix} if all the elements of the  diagonal are zero. Recall that by \textit{non-degenerate} matrix, we mean a full-rank matrix. Moreover, we say that a square matrix $A$ is a \textit{square root} of $B$ if $A^2=B$, and we write $A=\sqrt{B}$.
\begin{lem}\label{lem:full_isot_rank_iff_hollow_matrix}
    Let $H\in\calH_{n,2}$ be a non-degenerate matrix. Then $\irk(H)=\rk(H)=n+1$ if and only if there exists a matrix $Q\in\Oa_{n+1}(\bbC)$ such that $\transpose{Q} H Q$ is a hollow matrix.
\end{lem}
\begin{proof}
Let us assume that there exists an isotropic decomposition $L\in\bbC^{n+1,n+1}$ of $H$, that is, $H=L\transpose{L}$ such that $\transpose{L}^{(i)}L^{(i)}=0$ for every $i=0,\dots,n$. In particular, $\transpose{L}L$ is a hollow matrix. Since $A$ is non-degenerate, then also $L$ must be non-degenerate. Therefore, by \cite{Gan98b}*{Theorem 3}, there exist a complex symmetric matrix $S\in\calR_{n,2}$ and an orthogonal matrix $Q\in\Oa_{n+1}(\bbC)$ such that \[L=QS,\qquad S=\sqrt{\transpose{L}L}.\]
Therefore, we obtain the chain of equalities
\[
H=L\transpose{L}=QS\transpose(QS)=QS\transpose S\transpose Q=QS^2\transpose Q=Q\transpose{L}L\transpose Q, 
\]
that is, since $Q\in\Oa_{n+1}(\bbC)$, $\transpose{Q}H Q=\transpose{L}L$, which is a hollow matrix. Conversely, let $K$ be a hollow matrix such that $K=\transpose Q H Q$ for some $Q\in\Oa_{n+1}(\bbC)$. Since $H$ is symmetric, then also $K$ is symmetric, and hence, by \autoref{rem:isotropic_rank_equal}, we have
\[
\irk(K)=\irk(\transpose Q H Q)=\irk(H).
\]
Since it is always possible to determine a square root matrix of a complex symmetric matrix (see, e.g., \cite{Gan98a}*{Chapter V, \S 1}), there exists a square symmetric matrix $A\in\calR_{n,2}$ such that 
\[
K=A^2=A\transpose{A}.
\]
Since $K$ is a hollow matrix, we must have $A^{(i)}\transpose{A}^{(i)}=0$ for every $i=0,\dots,n$, that is, $A$ is an isotropic decomposition of $K$, and hence $\irk(K)=\irk(H)=n+1$.
\end{proof}
\begin{lem}\label{lem:isoquaisola}
Let $n\geq 1$, $H\in\calH_{n,2}$ a non-degenerate matrix. Then, there exist $v_0,\dots,v_n\in \bbC^{n+1}$ such that $\transpose{v}_iHv_i=0$ for any $i=0,\dots,n$, $\transpose{v}_iv_j=0$ for any $0\leq i<j\leq n$, and $\transpose{v}_iv_i\neq 0$ for any $i=0,\dots,n$.
\end{lem}
\begin{proof}
Let $Q_n\subset\bbP^n$  be the variety associated to the quadratic form $q_n\in\calR_{n,2}$, the dual quadratic form of $\omega_n$, 
    and let $P_n\subset\bbP^n$ be the variety associated to the quadratic form $h$, corresponding to the matrix $H$. Proving the statement is equivalent to prove that there exist $[v_0],\dots,[v_n]\in P_n\setminus(P_n\cap Q_n)$ such that $v_i$ and $v_j$ are $\omega_n$-orthogonal for any $0\leq i<j\leq n$. We prove this equivalent statement by induction. Let $\eta\in\calD_{n,2}$ the dual quadric of $h$. If $n=1$, then the space of $\eta$-harmonic polynomials is generated by $2$ square powers of $\eta$-isotropic linear forms $\ell_0$ and $\ell_1$ such that $h=\ell_0\ell_1$. Therefore, since $q_1$ is $\eta$-harmonic and non-degenerate, we can write $q_1=a_0\ell_0^2+a_1\ell_1^2$, for some $a_0,a_1\in\bbC$ with $a_0\neq0$ and $a_1\neq 0$. Without loss of generality, we can assume $a_0=a_1=1$. Denoting by $v_i\in\bbC^2$ the vector corresponding to $\ell_i$ for $i=1,2$, we have $P_1=\set{[v_1],[v_2]}$. Note that $[v_i]\in Q_1$ if and only if $\ell_i\mid q_1$, but this is impossible because $a_0$ and $a_1$ are both not zero. Moreover, if we set $(\beta_0,\beta_1)$ as the dual basis of $(\ell_0,\ell_1)$, then we have $\omega_1=\beta_0^2+\beta_1^2$. In particular, we have
    \[
    \omega_1\circ(\ell_0\ell_1)=(\beta_0^2+\beta_1^2)\circ(\ell_0\ell_1)=0,
    \]
    and thus $\ell_0$ and $\ell_1$ are $\omega_1$-orthogonal.
    Hence, $v_1$ and $v_2$ give two vectors as in the statement.
    Now we suppose that the statement is true for $n$ and we prove it for $n+1$. Let $[v_0]\in P_n$ such that the hyperplane $\bbP(v_0^\perp)$ is not tangent neither to $Q_n$ nor to $P_n$, where $v_0^\perp\subset\bbC^{n+1}$ is the orthogonal subspace of $v_0$ with respect to the inner product induced by $\omega_n$. We can always choose such a $v_0$ because the set of $[v]\in\bbP^n$ such that $\bbP(v^\perp)$ is tangent to $Q_n$ is exactly $Q_n$. Since $Q_n\cap \bbP(v_0^\perp)$ and $P_n\cap\bbP(v_0^\perp)$ are two non-degenerate quadrics in $\bbP(v_0^\perp)\cong\bbP^{n-1}$, we conclude by induction.
\end{proof}
By the previous two lemmata we obtain the first result about isotropic rank of harmonic quadratic forms, concerning non-degenerate matrices.
\begin{teo}\label{teo:isot_rank_nondegenerate_form}
    Let $H\in\calH_{n,2}$ be a non-degenerate matrix. Then $\irk H=\rk H=n+1$.
\end{teo}

\begin{proof}
    By \autoref{lem:full_isot_rank_iff_hollow_matrix}, it is sufficient to prove that $H$ is orthogonally similar to a hollow matrix, and this is true by \Cref{lem:isoquaisola}.
\end{proof}
\autoref{teo:isot_rank_nondegenerate_form} implicitly states that the isotropic rank of a general harmonic quadratic form is equal to its Waring rank as also expected by \Cref{theo:AH_harmonics}. The following analysis of the isotropic rank of degenerate harmonic quadrics makes the condition of generality explicit.

Now we focus on degenerate forms, namely, we determine the isotropic rank of any harmonic forms $h\in\calH_{n,2}$ such that $\rk h<n+1$.  In order to do that, we need the following definition for a specific class of matrices.
\begin{defn}
    A matrix $A=(a_{ij})_{0\leq i,j\leq n}\in\bbC^{n+1,n+1}$ is an \textit{$r$-frame matrix}, for some $r\in\bbN$, if $a_{ij}=0$ whenever $i\geq r$ or $j\geq r$.
\end{defn}
For degenerate matrices, we can write a result similar to \autoref{lem:full_isot_rank_iff_hollow_matrix}, involving hollow matrices.
\begin{prop}\label{prop:hollow_r_matrix}
 Let $H\in\calH_{n,2}$ such that $\rk H=r$. If $H$ is orthogonally similar to an $r$-frame hollow matrix, then $\irk H=r$.   
\end{prop}
\begin{proof}
Let $H$ be orthogonally similar to an $r$-frame hollow matrix, that is, there exists a matrix $Q\in\Oa_{n+1}(\bbC)$ such that $\transpose{Q}HQ$ is a $r$-frame hollow matrix. Then we can proceed exactly as in the proof of \autoref{lem:full_isot_rank_iff_hollow_matrix}. In particular, there is a symmetric hollow matrix $H'\in\bbC^{r,r}$ such that
\[
\transpose{Q}HQ=\begin{pNiceMatrix}[margin, cell-space-limits=4pt,vlines=2]
        H'&0\\
        \hline
        0&0
    \end{pNiceMatrix}.
\]
Therefore, there exists a symmetric matrix $A'\in\bbC^{r,r}$ such that $H'=A'^2=A'\transpose{A}'$. If we define the matrix
\[
A\coloneqq\begin{pNiceMatrix}[margin, cell-space-limits=4pt]
        A'\\
        \hline
        0
    \end{pNiceMatrix}\in\bbC^{n+1,r},
\]
then we have
\[
\transpose{Q}HQ=\begin{pNiceMatrix}[margin, cell-space-limits=4pt,vlines=2]
        H'&0\\
        \hline
        0&0
    \end{pNiceMatrix}=\begin{pNiceMatrix}[margin, cell-space-limits=4pt,vlines=2]
        A'\transpose{A}'&0\\
        \hline
        0&0
    \end{pNiceMatrix}=\begin{pNiceMatrix}[margin, cell-space-limits=4pt]
        A'\\
        \hline
        0
    \end{pNiceMatrix}\begin{pNiceMatrix}[cell-space-limits=4pt,vlines=2]
        \transpose{A}'&0
    \end{pNiceMatrix}=A\transpose{A}.
\]
Therefore, $A$ represents a decomposition of size $r$ made of isotropic points. We conclude that
\[
\irk(H)=\irk(\transpose{Q}HQ)=r.\qedhere
\]
\end{proof}
Observe that the reverse implication of \autoref{prop:hollow_r_matrix} does not hold. Indeed, if we consider the harmonic matrix
\[
H=\begin{pNiceMatrix}
    \rmi&1\\
    1&-\rmi
\end{pNiceMatrix},
\]
then we have $H=\rmi v\transpose{v}$, with $v=(1,-\rmi)$, so that $\irk H=1$, but $H$ cannot be orthogonally similar to a matrix
\[
H'=\begin{pNiceMatrix}
    a&0\\
    0&0
\end{pNiceMatrix}
\]
for any $a\in\bbC\setminus\{0\}$, because we should have $\tr H'=0$.

Although every symmetric real matrix is orthogonally similar to a diagonal matrix, this does not hold for symmetric complex matrices in general. However, it is possible to determine a canonical form for any complex matrix, up to orthogonal transformation. This classification is provided by F.~R.~Gantmacher in \cite{Gan98b}*{Chapter XI, \S 3} and can be summarised in the following theorem.
\begin{teo}[\cite{Gan98b}*{Corollary 2}]\label{teo:canonical_form}
    Any complex symmetric matrix $S\in\bbC^{n+1,n+1}$ is orthogonally similar to a blocks diagonal matrix
    \[
    \tilde{S}=\begin{pNiceMatrix}[margin,columns-width = 14pt,cell-space-limits=8pt]
\Block[borders={top,left,right,bottom}]{1-1}{}S_{s_1}(\lambda_1)&0& & &0\\
0&\Block[borders={top,left,right,bottom}]{1-1}{}S_{s_2}(\lambda_2)& 0& 0&\\[2ex]
 &0&&0& \\[2ex]
 &0&0&\Block[borders={top,left,right,bottom}]{1-1}{}S_{s_{r-1}}(\lambda_{r-1})&0\\
0& & &0&\Block[borders={top,left,right,bottom}]{1-1}{}S_r(\lambda_r)
 \CodeAfter 
       \line[shorten=14pt]{1-2}{1-5} 
       \line[shorten=14pt]{2-3}{2-4}
       \line[shorten=9pt]{4-2}{4-3} 
       \line[shorten=14pt]{5-1}{5-4}
       \line[shorten=15pt]{2-2}{4-4}
       \line[shorten=4pt]{2-1}{5-1} 
       \line[shorten=0pt]{3-2}{4-2}
       \line[shorten=0pt]{2-4}{3-4} 
       \line[shorten=4pt]{1-5}{4-5}
        \end{pNiceMatrix},
        \]
        where $\lambda_1,\dots,\lambda_r\in\bbC$ are eigenvalues of $S$ and, for every $j=1,\dots,r$, the block $S_{s_j}(\lambda_j)\in\bbC^{s_j,s_j}$ is the matrix
        \begin{equation}\label{formula:block_canonical_form}
        S_{s_j}(\lambda_j)=\lambda_jI_{s_j}+\frac{1}{2}J_{s_j}+\frac{\rmi}{2}K_{s_j},
        \end{equation}
        where $I_{s_j}$ is the $s_j\times s_j$ identity matrix and
\[
    J_{s_j}=
\begin{pNiceMatrix}[margin,columns-width = 3pt,cell-space-limits=5pt]
0&1&0& &0\\
1& & & & \\
0& & & &0\\
 & & & &1\\
0& &0&1&0\\
\CodeAfter
       \line[shorten=7pt]{1-3}{1-5} 
       \line[shorten=3pt]{1-5}{3-5}
       \line[shorten=9pt]{1-3}{3-5}
       \line[shorten=9pt]{1-2}{4-5}      
       \line[shorten=9pt]{1-1}{5-5}
       \line[shorten=9pt]{2-1}{5-4}
       \line[shorten=9pt]{3-1}{5-3}
       \line[shorten=3pt]{3-1}{5-1} 
       \line[shorten=7pt]{5-1}{5-3}
\end{pNiceMatrix}\in\bbC^{s_j,s_j},\qquad 
K_{s_j}=\begin{pNiceMatrix}[margin,columns-width = 3pt,cell-space-limits=5pt]
0& &0&1&0\\
 & & & &-1\\
0& & & &0\\
1& & & & \\
0&-1&0& &0\\
\CodeAfter
       \line[shorten=7pt]{1-3}{1-1} 
       \line[shorten=3pt]{1-1}{3-1}
       \line[shorten=9pt]{1-3}{3-1}
       \line[shorten=9pt]{1-4}{4-1}      
       \line[shorten=9pt]{1-5}{5-1}
       \line[shorten=9pt]{2-5}{5-2}
       \line[shorten=9pt]{3-5}{5-3}
       \line[shorten=3pt]{3-5}{5-5} 
       \line[shorten=7pt]{5-5}{5-3}
\end{pNiceMatrix}\in\bbC^{s_j,s_j}.
        \]
Moreover, $\tilde{S}$ is unique up to the order of the blocks.        
\end{teo}
The proof of \autoref{teo:canonical_form} is based on \cite{Gan98b}*{Theorem 5}, where it is proved that each block of the form \eqref{formula:block_canonical_form} is orthogonally similar to a block of the canonical Jordan form of the matrix. In particular, 
$\det S_{k}(\lambda)=(-1)^k\lambda^{k}$ for $\lambda\in\bbC$ and $k\in\bbN$ and 
$\rk S_{k}(0)=k-1$ for every $k\in\bbN$.  We refer to the blocks of the normal form of a symmetric matrix $S$ as the \textit{normal blocks} of $S$. Since by \autoref{rem:isotropic_rank_equal} we can consider matrices up to orthogonal transformations, we can work only by considering normal forms of matrices. We define a \textit{normal sequence} of a symmetric matrix $S$, whose normal blocks are $S_{s_1}(\lambda_1),\dots, S_{s_r}(\lambda_r)$, as the sequence
$\smash{(\lambda_1^{(s_1)},\dots,\lambda_r^{(s_r)})}$. Since the normal sequences of $S$ differ only by the order of the blocks, in the following we abuse of notation and we write {\it the normal sequence of $S$}.
For any $h\in\calH_{n,2}$, the {\it normal sequence of $h\in\calH_{n,2}$} is the normal sequence of its associated matrix.

By \autoref{prop:hollow_r_matrix}, we immediately get the following result, concerning frame matrices decomposed as direct sum of non-degenerate normal blocks.
\begin{cor}\label{cor:block_non_degen}
    Let $h\in\calH_{n,2}$ be a harmonic quadric with $\rk h=k$ and normal sequence $\smash{(\lambda_1^{(s_1)},\dots,\lambda_r^{(s_r)})}$. If $s_j=1$ for any $j$ such that $\lambda_j=0$, then $\irk h=k$.
\end{cor}
By \autoref{cor:block_non_degen}, in order to get the full classification of isotropic ranks of harmonic matrices, we just need to consider the normal sequences of harmonic matrices having at least one block $S_k(0)$, with $k\geq 2$. In the following, we analyse the normal sequences of all of these possible cases. As we have already seen in \autoref{exam:rank_vz}, in $\calH_{2,2}$ we have
\begin{equation*}
\irk\bigl((x_0-\rmi x_1)x_2\bigr)=4.
\end{equation*}
In particular, since the matrix 
\[
S_3(0)=\dfrac{1}{2}\begin{pNiceMatrix}
    0&1+\rmi&0\\
    1+\rmi&0&1-\rmi\\
    0&1-\rmi&0
\end{pNiceMatrix}
\]
corresponds to the harmonic quadratic form $(1+\rmi)(x_0-\rmi x_2)x_1$,
we immediately get
\begin{equation}\label{cor:rank_S3(0)}
\irk S_3(0)=4.
\end{equation}
The matrix $S_3(0)$ is indeed the unique normal block whose isotropic rank is different from its Waring rank. We start by considering the odd-size matrices.
\begin{prop}\label{prop:odd_blocks}
For every $k\in\bbN$, $\irk \bigl(S_{2k+5}(0)\bigr)=\rk \bigl(S_{2k+5}(0)\bigr)=2(k+2)$.
\end{prop}
\begin{proof}
Observe that, in general,
    the matrix \[
    \setcounter{MaxMatrixCols}{20}
    S_{2k+5}(0)=\dfrac{1}{2}
\begin{pNiceMatrix}[margin,columns-width = auto,cell-space-limits=6pt]
0   &1    &0&    &      &      &      &     &0&\rmi&0\\
1   &     & &    &      &      &      &     & &    &-\rmi\\
0   &     & &    &      &      &      &     & &    &0\\
    &     & &0   &1     &0     &\rmi  &0    & &    & \\
    &     & &1   &0     &1+\rmi&0     &-\rmi& &    & \\
    &     & &0   &1+\rmi&0     &1-\rmi&0    & &    & \\
    &     & &\rmi&0     &1-\rmi&0     &1    & &    & \\
    &     & &0   &-\rmi &0     &1     &0    & &    & \\
0   &     & &    &      &      &      &     & &    &0\\
\rmi&     & &    &      &      &      &     & &    &1 \\
0   &-\rmi&0&    &      &      &      &     &0&1   &0\\
\CodeAfter
       \line[shorten=10pt]{1-1}{4-4}
       \line[shorten=10pt]{1-2}{4-5}
       \line[shorten=10pt]{1-3}{4-6}
       \line[shorten=10pt]{2-1}{5-4}
       \line[shorten=10pt]{3-1}{6-4}
       \line[shorten=10pt]{1-11}{4-8}
       \line[shorten=10pt]{1-10}{4-7}
       \line[shorten=10pt]{1-9}{4-6}
       \line[shorten=10pt]{2-11}{5-8}
       \line[shorten=10pt]{3-11}{6-8}
       \line[shorten=10pt]{11-1}{8-4}
       \line[shorten=10pt]{10-1}{7-4}
       \line[shorten=10pt]{9-1}{6-4}
       \line[shorten=10pt]{11-2}{8-5}
       \line[shorten=10pt]{11-3}{8-6}
       \line[shorten=10pt]{11-11}{8-8}
       \line[shorten=10pt]{11-10}{8-7}
       \line[shorten=10pt]{11-9}{8-6}
       \line[shorten=10pt]{10-11}{7-8}
       \line[shorten=10pt]{9-11}{6-8}
       \line[shorten=14pt]{1-3}{1-9}
       \line[shorten=10pt]{3-1}{9-1}
       \line[shorten=14pt]{11-3}{11-9}
       \line[shorten=10pt]{3-11}{9-11}
\end{pNiceMatrix}
    \] 
    corresponds to the quadratic form
    \begin{align*}
    h&=x_{k+1}\bigl(x_k+(1+\rmi)x_{k+2}-\rmi x_{k+4})+\rmi x_{k+3}\bigl(x_k-(1+\rmi)x_{k+2}-\rmi x_{k+4}\bigr)\\
    &\hphantom{{}={}}+\sum_{j=1}^k (x_{k-j}-\rmi x_{k+j+4})(x_{k-j+1}+\rmi x_{k+j+3}).
    \end{align*}
    The first $2$ summands represent to the middle $(5\times 5)$-block of $S_{2k+5}(0)$ and each of the successive $k$ summands correspond to the $k$ external frames of the matrix. 
    We can write
     \begin{align*}
    h&=\dfrac{1+\rmi}{8}\Bigl(\bigl(x_k-\rmi x_{k+4}+(1+\rmi)(x_{k+2}-\rmi x_{k+1})\bigr)^2-\bigl(x_k-\rmi x_{k+4}+(1+\rmi)(x_{k+2}+\rmi x_{k+1})\bigr)^2\Bigr)\\[1ex]
    &\hphantom{{}={}}+\dfrac{1-\rmi}{8}\Bigl(\bigl(x_k-\rmi x_{k+4}-(1+\rmi)(x_{k+2}-\rmi x_{k+3})\bigr)^2-\bigl(x_k-\rmi x_{k+4}-(1+\rmi)(x_{k+2}+\rmi x_{k+3})\bigr)^2\Bigr)\\
    &\hphantom{{}={}}+\dfrac{1}{4}\sum_{j=1}^k\bigl((x_{k-j}-\rmi x_{k+j+4}+x_{k-j+1}+\rmi x_{k+j+3})^2-(x_{k-j}-\rmi x_{k+j+4}-x_{k-j+1}-\rmi x_{k+j+3})^2\bigr),
    \end{align*}
    which is an isotropic decomposition of size $2(k+2)$. Since $\rk S_{2k+5}(0)=2(k+2)$, we conclude by \Cref{teo:two_times_rank}.
\end{proof}
We can do a similar procedure for the even-size normal blocks.
\begin{prop}\label{prop:blocchipari}
    For every $k\in\bbN$, $\irk \bigl(S_{2(k+1)}(0)\bigr)=\rk \bigl(S_{2(k+1)}(0)\bigr)=2k+1$.
\end{prop}
\begin{proof}
    Observe that, for every $k\in\bbN$,
    the matrix \[
    \setcounter{MaxMatrixCols}{20}
    S_{2(k+1)}(0)=\dfrac{1}{2}
\begin{pNiceMatrix}[margin,columns-width = auto,cell-space-limits=6pt]
0   &1    &0&    &      &      &      &     & &0&\rmi&0\\
1   &     & &    &      &      &      &     & & &    &-\rmi\\
0   &     & &    &      &      &      &     & & &    &0\\
    &     & &    &      &0     &0     &     & & &    & \\
    &     & &    &0     &1     &\rmi  &0    & & &    & \\
    &     & &0   &1     &\rmi  &1     &-\rmi&0& &    & \\
    &     & &0   &\rmi  &1     &-\rmi &1    &0& &    & \\
    &     & &    &0     &-\rmi &1     &0    & & &    & \\
    &     & &    &      &0     &0     &     & & &    & \\
0   &     & &    &      &      &      &     & & &    &0\\
\rmi&     & &    &      &      &      &     & & &    &1 \\
0   &-\rmi&0&    &      &      &      &     & &0&1   &0\\
\CodeAfter
       \line[shorten=10pt]{1-1}{5-5}
       \line[shorten=10pt]{1-2}{5-6}
       \line[shorten=10pt]{2-1}{6-5}
       \line[shorten=10pt]{1-3}{4-6}
       \line[shorten=10pt]{3-1}{6-4}
       \line[shorten=10pt]{1-12}{5-8}
       \line[shorten=10pt]{1-11}{5-7}
       \line[shorten=10pt]{2-12}{6-8}
       \line[shorten=10pt]{1-10}{4-7}
       \line[shorten=10pt]{3-12}{6-9}
      \line[shorten=10pt]{12-1}{8-5}
       \line[shorten=10pt]{11-1}{7-5}
       \line[shorten=10pt]{12-2}{8-6}
       \line[shorten=10pt]{10-1}{7-4}
       \line[shorten=10pt]{12-3}{9-6}
       \line[shorten=10pt]{12-12}{8-8}
       \line[shorten=10pt]{11-12}{7-8}
       \line[shorten=10pt]{10-12}{7-9}
       \line[shorten=10pt]{12-11}{8-7}
       \line[shorten=10pt]{12-10}{9-7}
       \line[shorten=14pt]{1-3}{1-10}
       \line[shorten=10pt]{3-1}{10-1}
       \line[shorten=14pt]{12-3}{12-10}
       \line[shorten=10pt]{3-12}{10-12}
\end{pNiceMatrix}
    \]     
    corresponds to the quadratic form
    \[
    h=\dfrac{\rmi}{2} (x_k-\rmi x_{k+1})^2+\sum_{j=0}^{k-1}(x_{k-j-1}-\rmi x_{k+j+2})(x_{k-j}+\rmi x_{k+j+1}),
    \]
    where the first summand corresponds to the middle $(2\times 2)$-block of the matrix and each of the successive $k$ summands correspond to the $k$ external frames of the matrix. So we can do as in the proof of \autoref{prop:odd_blocks} and write
    \begin{align*}
    h=\dfrac{\rmi}{2} (x_k-\rmi x_{k+1})^2+\dfrac{1}{4}\sum_{j=0}^{k-1}\bigl(&(x_{k-j-1}-\rmi x_{k+j+2}+x_{k-j}+\rmi x_{k+j+1})^2\\
    &-(x_{k-j-1}-\rmi x_{k+j+2}-x_{k-j}-\rmi x_{k+j+1})^2\bigr),
    \end{align*}
    which gives a decomposition of size $2k+1$. Since $\rk S_{2(k+1)}(0)=2k+1$, we conclude by \Cref{teo:two_times_rank}.
\end{proof}
Now, we consider the case where the normal form consists of a direct sum of multiple copies of $S_3(0)$.
\begin{prop}\label{prop:S3(0)_blocks}
    For every $k\geq 2$, $\irk\bigl(S_3(0)^{\oplus k}\bigr)=\rk\bigl(S_3(0)^{\oplus k}\bigr)=2k$.
\end{prop}
\begin{proof}
    The matrix 
    \[
    \setcounter{MaxMatrixCols}{20}
    S_3(0)^{\oplus k}=
\begin{pNiceMatrix}[margin,columns-width = auto,cell-space-limits=6pt]
\Block[borders={right,bottom}]{3-3}{}
0     &1+\rmi&0     &0&   &      &      &0     \\
1+\rmi&0     &1-\rmi&0&   &      &0     &      \\
0     &1+\rmi&0     &0&   &0     &      &      \\
0     &0     &0     &\Block[borders={top,left,right,bottom}]{2-2}{} &   &      &      &      \\
      &      &      & &   &0     &0     &0     \\
      &      &0     & &0  &\Block[borders={top,left}]{3-3}{}0     &1+\rmi&0     \\
      &0     &      & &0  &1+\rmi&0     &1-\rmi\\
0     &      &      & &0  &0     &1-\rmi&0     \\
\CodeAfter
       \line[shorten=27pt]{3-3}{6-6}
       \line[shorten=14pt]{1-4}{1-8}
       \line[shorten=14pt]{2-4}{2-7}
       \line[shorten=14pt]{3-4}{3-6}
       \line[shorten=5pt]{5-8}{1-8}
       \line[shorten=5pt]{5-7}{2-7}
       \line[shorten=5pt]{5-6}{3-6}
       \line[shorten=5pt]{4-1}{8-1}
       \line[shorten=5pt]{4-2}{7-2}
       \line[shorten=5pt]{4-3}{6-3}
       \line[shorten=14pt]{8-5}{8-1}
       \line[shorten=14pt]{7-5}{7-2}
       \line[shorten=14pt]{6-5}{6-3}
\end{pNiceMatrix}
    \] 
    corresponds in $\calH_{3k-1,2}$ to the quadratic form
    \[
    h=(1+\rmi)\sum_{j=0}^{k-1}(x_{3j}-\rmi x_{3j+2})x_{3j+1}.
    \]
    To prove the statement, it is sufficient to prove it for the cases where $k=2$ and $k=3$.  
    In the first case, we have
    \[
    h=(1+\rmi)\bigl((x_0-\rmi x_2)x_1+(x_3-\rmi x_5)x_4\bigr),
    \]
    which can be decomposed as
    \begin{align*}
    h&
    =\dfrac{1+\rmi}{8}\bigl((x_0+x_1-\rmi x_2-\rmi x_3+\rmi x_4-x_5)^2-(x_0-x_1-\rmi x_2-\rmi x_3-\rmi x_4-x_5)^2\bigr)\\
    &\hphantom{{}={}}+\dfrac{1-\rmi}{8}\bigl( (x_0+\rmi x_1 -\rmi x_2 + \rmi x_3+x_4+ x_5)^2- (x_0-\rmi x_1 -\rmi x_2 + \rmi x_3-x_4+ x_5)^2\bigr),
    \end{align*}
    which is an isotropic decomposition.
    Therefore, by \Cref{teo:two_times_rank},
    $\irk h=4$.
    For the case where $k=3$, instead, we have
    \[
    h=(1+\rmi)\bigl((x_0-\rmi x_2)x_1+(x_3-\rmi x_5)x_4+(x_6-\rmi x_8)x_7\bigr),
    \]
    which can be decomposed as
    \begin{align*}
h
&=\rmi(x_0+x_1-\rmi x_2-x_3+\rmi x_4+\rmi x_5-\rmi x_6-x_8)^2\\
&\hphantom{{}={}}-\rmi(x_0-x_1-\rmi x_2-x_3-\rmi x_4+\rmi x_5-\rmi x_6-x_8)^2\\
&\hphantom{{}={}}+(x_0-\rmi x_2+\rmi x_3+x_4+ x_5-\rmi x_6+\rmi x_7- x_8)^2\\
&\hphantom{{}={}}-(x_0-\rmi x_2+\rmi x_3-x_4+ x_5-\rmi x_6-\rmi x_7- x_8)^2\\
&\hphantom{{}={}}-\rmi(x_0+\rmi x_1-\rmi x_2+\rmi x_3+ x_5- x_6+x_7+\rmi x_8)^2\\
&\hphantom{{}={}}+\rmi(x_0-\rmi x_1-\rmi x_2+\rmi x_3+ x_5- x_6-x_7+\rmi x_8)^2.
    \end{align*}
    Hence, by \Cref{teo:two_times_rank}, $\irk h=6$.
\end{proof}
\autoref{prop:S3(0)_blocks} allows to do some considerations on the additivity of the isotropic rank. One of the main topics about the Waring rank is the so-called symmetric Strassen conjecture, which states that, given two polynomials $f\in\bbC[x_0,\dots,x_n]$ and $g\in\bbC[y_0,\dots,y_m]$ of the same degree $d$, we have
\[
\rk(f+g)=\rk f+\rk g.
\]
Symmetric Strassen conjecture is a specific case of Strassen conjecture, which is stated in the more general setting of tensors (see \cite{Str73}). Y.~Shitov recently proved in \cite{Shi19} that Strassen conjecture is false in general (see also \cite{BFPSS25} for an alternative and more specific version for the field complex numbers). Also for the Waring rank, the conjecture has been proved to be false, see \cite{Shi24}. However, there are some \textit{small} cases for which the statement holds, see \cite{CCC15}. In particular, we have the following case for rank one polynomials or, equivalently, polynomials in one variable.
\begin{prop}[\cite{CCC15}*{Proposition 3.1}]\label{prop:Strassen_rank_one}
    Let $f\in\bbC[x_0,\dots,x_n]$ and $g\in\bbC[y_0,\dots,y_m]$ be homogeneous polynomials of degree $d$. If either $n=0$ or $m=0$, then 
    \[
\rk(f+g)=\rk f+\rk g.
\]
\end{prop}
Thanks to \autoref{cor:rank_S3(0)} and \autoref{prop:S3(0)_blocks} together, we can immediately prove that also the harmonic version of the Strassen conjecture does not hold, providing the explicit counterexample
\begin{equation}
    \irk\bigl(S_3(0)\oplus S_3(0)\bigr)=4<2\irk\bigl(S_3(0)\bigr)=8.
\end{equation}
We want to show that in the analysis of the remaining cases we can assume that our matrices do not have zero blocks, that is, their normal sequences do not contain elements $0^{(1)}$. This is a direct consequence of the following statement, which is a weaker version for isotropic rank of \cite{CCO17}*{Proposition 2.3} for Waring decompositions. 
\begin{lem}\label{lem:variables_extra}
    Let $h\in\bbC[x_0,\dots,x_k]\subseteq \bbC[x_0,\dots,x_n]$ be a harmonic form of degree $d$ for some $k\leq n$ and let $r\coloneqq \irk h\leq \rk h+2$. Let
    \[
    h=\sum_{i=1}^r\ell_i^d=\sum_{j=1}^s\ell_j'^d,
    \]
    for some $\ell_1,\dots,\ell_r\in\bbC[x_0,\dots,x_k]$ such that $\omega_k\circ\ell_i^2=0$ for every $i=1,\dots,r$, and some $\ell_1',\dots,\ell_s'\in\bbC[x_0,\dots,x_n]$ such that $\omega_n\circ\ell_j'^2=0$ for every $j=1,\dots,s$. If there exists $j$ with $k<j\leq n$, such that $\alpha_j\circ\ell_i'\neq 0$ for some $1\leq i\leq s$, then $r\leq s$.
\end{lem}
\begin{proof}
    Let us assume that $s\leq r-1$ and, without loss of generality, that $\alpha_{k+1}\circ\ell_1'\neq 0$. Then, we have
    \[
    \rk(h-\ell_1'^d)\leq\irk(h-\ell_1'^d)\leq s-1,
    \]
    but since \autoref{prop:Strassen_rank_one} implies that
    \[
    \rk(h-\ell_1'^d)=\rk(h)+\rk(\ell_1'^d)\geq\irk(h)-2+1=r-1\geq s>s-1,
    \]
    we have a contradiction.
\end{proof}
Since the isotropic rank of a harmonic quadratic form is at most the sum of the isotropic ranks of the blocks of its normal form, we can use \autoref{teo:isot_rank_nondegenerate_form}, \autoref{cor:rank_S3(0)}, \autoref{prop:odd_blocks}, \autoref{prop:blocchipari} and \autoref{prop:S3(0)_blocks} together to conclude that 
\[
\irk(h)\leq \rk(h)+2,
\]
for any $h\in\calH_{n,2}$.
Therefore, using \autoref{lem:variables_extra}, we can always reduce to the cases where in the normal sequence $\smash{(\lambda_1^{(s_1)},\dots,\lambda_r^{(s_r)})}$ of a quadratic form, we do not have blocks of type $0^{(1)}$. We can gather the results obtained so far in the following proposition.
\begin{prop}\label{prop: partial class}
Let $h\in\calH_{n,2}$ have normal sequence $(\lambda_1^{(s_1)},\dots,\lambda_r^{(s_r)})$. If either $s_j\neq 3$ whenever $\lambda_j=0$, or there are $1\leq j_1<j_2\leq r$  such that $\lambda_{j_1}^{(s_{j_1})}=\lambda_{j_2}^{(s_{j_2})}=0^{(3)}$, then $\irk(h)=\rk(h)$.
\end{prop}
It remains to analyse the cases where the normal form of a matrix has just one degenerate $3\times 3$ block. First, we consider the case where there exists another odd degenerate block.
\begin{prop}\label{prop: partial class1}
    Let $h\in\calH_{n,2}$ have normal sequence $(\lambda_1^{(s_1)},\dots,\lambda_r^{(s_r)})$, such that there exists a unique $1\leq j\leq r$ such that $\lambda_j^{(s_j)}=0^{(3)}$. If there exists $1\leq t\leq r$ such that $\lambda_t=0$ and $s_t=2k+5$ for some $k\in\bbN$, then $\irk h=\rk h$. 
\end{prop}
\begin{proof}
By \Cref{prop: partial class} and the subadditivity of the isotropic rank, it is sufficient to prove the statement for the matrix $S_{2k+5}(0)\oplus S_3(0)$, corresponding to the form
\begin{align*}
    h&=x_{k+1}\bigl(x_k+(1+\rmi)x_{k+2}-\rmi x_{k+4})+\rmi x_{k+3}\bigl(x_k-(1+\rmi)x_{k+2}-\rmi x_{k+4}\bigr)\\
    &\hphantom{{}={}}+\sum_{j=1}^k\bigl((x_{k-j}-\rmi x_{k+j+4})(x_{k-j+1}+\rmi x_{k+j+3})\bigr)+(1+\rmi)x_{2k+6}(x_{2k+5}-\rmi x_{2k+7})\\
    &=(x_k-\rmi x_{k+4})(x_{k+1}+\rmi x_{k+3})+(1+\rmi)x_{k+2}(x_{k+1}-\rmi x_{k+3})\\
    &\hphantom{{}={}}+\sum_{j=1}^k\bigl((x_{k-j}-\rmi x_{k+j+4})(x_{k-j+1}+\rmi x_{k+j+3})\bigr)+(1+\rmi)x_{2k+6}(x_{2k+5}-\rmi x_{2k+7})\\
    &=(1+\rmi)\bigl((x_{k+1}-\rmi x_{k+3})x_{k+2}+(x_{2k+5}-\rmi x_{2k+7})x_{2k+6}\bigr)\\
    &\hphantom{{}={}}+\sum_{j=0}^k\bigl((x_{k-j}-\rmi x_{k+j+4})(x_{k-j+1}+\rmi x_{k+j+3})\bigr).
    \end{align*}
    In a similar way as in the proof of \autoref{prop:S3(0)_blocks}, we write
    \begin{align*}
    h&=\dfrac{1+\rmi}{8}\bigl((x_{k+1}+x_{k+2}-\rmi x_{k+3}-\rmi x_{2k+5}+\rmi x_{2k+6}-x_{2k+7})^2\\[1ex]
    &\hphantom{{}={}}\qquad\ -(x_{k+1}-x_{k+2}-\rmi x_{k+3}-\rmi x_{2k+5}-\rmi x_{2k+6}-x_{2k+7})^2\bigr)\\[1ex]
    &\hphantom{{}={}}+\dfrac{1-\rmi}{8}\bigl( (x_{k+1}+\rmi x_{k+2} -\rmi x_{k+3} + \rmi x_{2k+5}+x_{2k+6}+ x_{2k+7})^2\\[1ex]
    &\hphantom{{}={}}\qquad\quad\ - (x_{k+1}-\rmi x_{k+2} -\rmi x_{k+3} + \rmi x_{2k+5}-x_{2k+6}+x_{2k+7})^2\bigr)\\
    &\hphantom{{}={}}+\dfrac{1}{4}\sum_{j=0}^k\bigl((x_{k-j}-\rmi x_{k+j+4}+x_{k-j+1}+\rmi x_{k+j+3})^2-(x_{k-j}-\rmi x_{k+j+4}-x_{k-j+1}-\rmi x_{k+j+3})^2\bigr),
    \end{align*}
    which is a decomposition of size $2(k+3)$, proving the statement by \Cref{teo:two_times_rank}.
\end{proof}
The same argument holds if an even-size degenerate block appears in the normal form, but only if it is at least a $4\times 4$ block. However, an explicit decomposition is not trivial.
\begin{prop}\label{prop: partial class 2}
    Let $h\in\calH_{n,2}$ have normal sequence $(\lambda_1^{(s_1)},\dots,\lambda_r^{(s_r)})$, such that there exists a unique $1\leq j\leq r$ such that $\lambda_j^{(s_j)}=0^{(3)}$. If there exists $1\leq t\leq r$ such that $\lambda_t=0$ and $s_t=2(k+1)$ for some $k\geq 1$, then $\irk h=\rk h$. 
\end{prop}
\begin{proof}
    By \Cref{prop: partial class} and the subadditivity of the isotropic rank, it is sufficient to prove the statement for the matrix $S_{2(k+1)}(0)\oplus S_3(0)$, corresponding to the form
    \begin{align*}
    h&=\rmi (x_k-\rmi x_{k+1})^2+\sum_{j=0}^{k-1}\bigl((x_{k-j-1}-\rmi x_{k+j+2})(x_{k-j}+\rmi x_{k+j+1})\bigr)+(1+\rmi)(x_{2k+2}-\rmi x_{2k+4})x_{2k+3}\\
    &=\rmi (x_k-\rmi x_{k+1})^2+(1+\rmi)(x_{2k+2}-\rmi x_{2k+4})x_{2k+3}+(x_{k-1}-\rmi x_{k+2})(x_{k}+\rmi x_{k+1})\\
    &\hphantom{{}={}}+\frac{1}{4}\sum_{j=1}^{k-1}\bigl((x_{k-j-1}-\rmi x_{k+j+2}+x_{k-j}+\rmi x_{k+j+1})^2-(x_{k-j-1}-\rmi x_{k+j+2}-x_{k-j}-\rmi x_{k+j+1})^2\bigr).
    \end{align*}
    If we determine a decomposition of size $5$ of the form 
    \[
    h'=\rmi (x_k-\rmi x_{k+1})^2+2(x_{k-1}-\rmi x_{k+2})(x_{k}+\rmi x_{k+1})+2(1+\rmi)(x_{2k+2}-\rmi x_{2k+4})x_{2k+3},
    \]
    given by the first three summands, then the statement follows by \Cref{teo:two_times_rank}. We can write
    \begin{align*}
    h'
    &=\dfrac{1}{8}\bigl(4(1+\rmi)(x_{2k+3}+\rmi x_{k})+2(x_{2k+2}-\rmi x_{2k+4})\bigr)^2-4\rmi(x_{2k+3}+\rmi x_{k})^2\\
    &\hphantom{{}={}}-\dfrac{\rmi}{4}\bigl((1-\rmi)(x_{2k+2}-\rmi x_{2k+4})+2\rmi (x_k-\rmi x_{k+1})\bigr)^2\\
    &\hphantom{{}={}}+\frac{1}{4}\bigl((1-\rmi)(x_{2k+2}-\rmi x_{2k+4})+(x_{k-1}-\rmi x_{k+2})+(x_{k}+\rmi x_{k+1})\bigr)^2\\
    &\hphantom{{}={}}+\frac{1}{4}\bigl((1-\rmi)(x_{2k+2}-\rmi x_{2k+4})+(x_{k-1}-\rmi x_{k+2})-(x_{k}+\rmi x_{k+1})\bigr)^2,
    \end{align*}
    which gives the required decomposition and conclude the statement.
    \end{proof}
    The last case where isotropic rank and Waring rank coincide is obtained for matrices whose normal form contains at least one non-degenerate block.
\begin{prop}\label{prop: partial class 3}
    Let $h\in\calH_{n,2}$ have normal sequence $(\lambda_1^{(s_1)},\dots,\lambda_r^{(s_r)})$, such that there exists a unique $1\leq j\leq r$ such that $\lambda_j^{(s_j)}=0^{(3)}$. If there exists $1\leq t\leq r$ such that $\lambda_t\neq 0$, then $\irk h=\rk h$. 
\end{prop}
\begin{proof}
    Observe that since $h$ is harmonic, the sum of the dimensions of non-degenerate blocks must be at least $2$, since the trace must be equal to $0$. By the subadditivity of the isotropic rank, we just need to prove that, for any non-degenerate harmonic quadratic form $q$ of rank at least $2$, we have
    \[\irk\bigl(q+(1+\rmi)(x_0-\rmi x_2)x_1\bigr)=\rk\bigl(q+(1+\rmi)(x_0-\rmi x_2)x_1\bigr).\]
    In particular, we just need to prove that, given two isotropic linear forms $\ell_1,\ell_2$ such that $\omega_n\circ\ell_1\ell_2\neq 0$, we have
    \[
    \irk\bigl((1+\rmi)(x_0-\rmi x_2)x_1+\ell_1^2+\ell_2^2\bigr)=4.
    \]
    Since $\omega_n\circ \ell_1\ell_2\neq 0$, it follows that $\omega_n\circ(\ell_1+\ell_2)^2\neq 0$, that is, $\ell_1+\ell_2$ is not isotropic. Therefore, there exists some $a\in\bbC$ such that $(1+\rmi)x_1-a(\ell_1+\ell_2)$ is isotropic. 
    Then, we can write
    \begin{align*}
    (1+\rmi)(x_0-\rmi x_2)x_1+\ell_1^2+\ell_2^2
        &=\dfrac{1}{4}\bigl(a(x_0-\rmi x_2)+2\ell_1\bigr)^2+\dfrac{1}{4}\bigl(a(x_0-\rmi x_2)+2\ell_2\bigr)^2\\
        &\hphantom{{}={}}+\dfrac{1}{4}\biggl((1+\rmi)x_1-a\ell_1-a\ell_2+\dfrac{2-a^2}{2}(x_0-\rmi x_2)\biggr)^2\\
        &\hphantom{{}={}}-\dfrac{1}{4}\biggl((1+\rmi)x_1-a\ell_1-a\ell_2-\dfrac{2+a^2}{2}(x_0-\rmi x_2)\biggr)^2,
    \end{align*}
    which gives an isotropic decomposition of size $4$.
\end{proof}
We analyse now the unique remaining case.
\begin{prop}\label{prop: partial class 4}
    If $h\in\calH_{n,2}$ has normal sequence $(0^{(3)},0^{(2)},\dots,0^{(2)})$, then
    $\irk h=\rk h+2$.
\end{prop}
\begin{proof}
   For any $k\in\bbN$, we have to determine the isotropic rank of the matrix
   \[
    \setcounter{MaxMatrixCols}{20}
    H=S_3(0)\oplus S_2(0)^{\oplus k}=\dfrac{1}{2}
\begin{pNiceMatrix}[margin,columns-width = auto,cell-space-limits=6pt]
\Block[borders={right,bottom}]{3-3}{}
0     &1+\rmi&0     &0&   &      &      &0     \\
1+\rmi&0     &1-\rmi&0&   &      &0     &      \\
0     &1+\rmi&0     &0&   &0     &      &      \\
0     &0     &0     &\Block[borders={top,left,right,bottom}]{2-2}{}\rmi &1 &      &      &      \\
      &      &      &1 &-\rmi   &0     &     &     \\
      &      &0     & &0  &\Block[borders={top,left,right,bottom}]{1-1}{}     &0&0     \\
      &0     &      & &  &0&\Block[borders={top,left}]{2-2}{}\rmi     &1\\
0     &      &      & &  &0     &1&-\rmi     \\
\CodeAfter
       \line[shorten=14pt]{1-4}{1-8}
       \line[shorten=14pt]{2-4}{2-7}
       \line[shorten=14pt]{3-4}{3-6}
       \line[shorten=5pt]{6-8}{1-8}
       \line[shorten=5pt]{6-7}{2-7}
       \line[shorten=5pt]{5-6}{3-6}
       \line[shorten=5pt]{4-1}{8-1}
       \line[shorten=5pt]{4-2}{7-2}
       \line[shorten=5pt]{4-3}{6-3}
       \line[shorten=14pt]{8-6}{8-1}
       \line[shorten=14pt]{7-6}{7-2}
       \line[shorten=14pt]{6-5}{6-3}
       \line[shorten=23pt]{5-5}{7-7}
\end{pNiceMatrix}.
    \] 
    First observe that $\rk H=k+2$ 
    and a minimal Waring decomposition of $h$ is given by 
    \[
    h=\dfrac{1+\rmi}{4}(x_0+x_1-\rmi x_2)^2-\dfrac{1+\rmi}{4}(x_0-x_1-\rmi x_2)^2+\rmi\sum_{j=1}^{k}(x_{2j+1}-\rmi x_{2j+2})^2.
    \]
    We first assume that also $\irk H=k+2$. 
    In this case we can determine a decomposition
    \[
    S_3(0)\oplus S_2(0)^{\oplus k}=v_1\transpose{v}_1+\cdots+v_{k+2}\transpose{v}_{k+2},
    \]
    where $v_1,\dots,v_{k+2}\in\bbC^{2k+3}$ are isotropic. Let $w_0$ be the isotropic vector corresponding to the linear form $x_0-\rmi x_2$, let $w_j$ be the isotropic vector corresponding to $x_{2j+1}-\rmi x_{2j+2}$ for every $j=1,\dots,k$, and let $e_1$ be the vector corresponding to $x_1$. Then, since $\{v_1,\dots,v_{k+2}\}$ and $\{e_1,w_0,\dots,w_{k}\}$ are two sets of linearly independent vectors whose squares span the same quadratic form, we have
    \[
    \langle v_1,\dots,v_{k+2}\rangle=\langle e_1,w_0,\dots,w_{k}\rangle
    \]
    Therefore, for every $i=1,\dots,k+2$, there exists $\beta_1,\alpha_0,\dots,\alpha_k\in\bbC$ such that
    \[
    v_i=\beta_1e_1+\alpha_0w_0+\cdots+\alpha_kw_k.
    \]
    However, since $v_i$ is isotropic for every $i=1,\dots,k+2$, we have
    \[
\transpose{v}_iv_i=\beta_1^2\transpose{e}_1e_1+\alpha_0\transpose{w}_0w_0+\cdots+\alpha_k\transpose{w}_kw_k=\beta_1^2=0.
    \]
    In particular
    \[
    \langle v_1,\dots,v_{k+2}\rangle\subseteq\langle w_0,\dots,w_{k}\rangle,
    \]
    but this means that $\rk H\leq k+1$, which is a contradiction. Let us assume now that $\irk H=k+3$.
    In this case we can determine a decomposition
    \[
    S_3(0)\oplus S_2(0)^{\oplus k}=v_1\transpose{v}_1+\cdots+v_{k+3}\transpose{v}_{k+3},
    \]
    where $v_1,\dots,v_{k+3}\in\bbC^{2k+3}$ are isotropic.
     By construction, we must have
    \[
    Hw_j=(v_1\transpose{v}_1+\cdots+v_{k+3}\transpose{v}_{k+3})w_j=(\transpose{v}_1w_j)v_1+\cdots+(\transpose{v}_{k+3}w_j)v_{k+3}=0
    \]
    for every $j=0,\dots,k$. Fix $j=1$. Since $\dim\langle w_0,\dots,w_k\rangle^{\perp}=k+2$, there exist $1\leq i\leq k+3$ such that $\transpose{v}_iw_1\neq 0$. That is, there exist $a_1,\dots,a_{k+3}\in\bbC$, not all zero elements, such that
    \[
    a_1v_1+\cdots+a_{k+3}v_{k+3}=0.
    \]
    This means that $v_1,\dots,v_{k+3}$ are not linearly independent. This implies that \[
    \dim\langle v_1,\dots,v_{k+3}\rangle\leq k+2,
    \]
    and hence we must have
    \[
     \langle v_1,\dots,v_{k+3}\rangle\subseteq \langle e_1,w_0,\dots,w_k\rangle.
    \]
Proceeding like in the previous case, we obtain
    \[
    \langle v_1,\dots,v_{k+3}\rangle\subseteq\langle w_0,\dots,w_{k}\rangle,
    \]
    but this means that $\rk H\leq k+1$, which is a contradiction. Therefore, by the subadditivity of the isotropic rank we obtain $\irk H=k+4$.
\end{proof}
We are finally ready to provide the complete classification for the isotropic rank of harmonic quadratic forms.
\begin{teo}\label{teo: quadrics}
Let $h\in\calH_{n,2}$ a harmonic quadratic form with normal sequence. Then
$$\irk h=\begin{cases}
    \rk h+2, &\text{if $(\lambda_1^{(s_1)},\dots,\lambda_r^{(s_r)})=(0^{(3)},0^{(2)},\dots,0^{(2)})$,}\\
    \rk h,& \text{otherwise.}
\end{cases}$$
\end{teo}
\begin{proof}
It follows by \Cref{prop: partial class}, \Cref{prop: partial class1}, \Cref{prop: partial class 2}, \Cref{prop: partial class 3} and \Cref{prop: partial class 4}.
\end{proof}
The previous theorem can be stated in a more concise and equivalent formulation. Indeed, the normal sequences $(\lambda_1^{(s_1)},\dots,\lambda_r^{(s_r)})=(0^{(3)},0^{(2)},\dots,0^{(2)})$ correspond exactly to symmetric nilpotent matrices whose square power is a rank-one matrix, which, in particular, have order of nilpotency equal to $3$. Therefore, we obtain the formulation of \Cref{teo: quadrics}.
\section{Harmonic monomials}\label{sec:monomials}
In the section, we consider the case of harmonic monomials. In particular, by monomial, we mean a product of powers of at most $n+1$ linearly independent linear forms.
Differently from the case of the Waring rank, if we want to preserve a quadratic form $\omega_{n,2}$ we can modify the monomials only using $\omega_n$-orthogonal transformations. Therefore, once the quadratic form $\omega_n$ is fixed, we cannot always assume the linear forms to be variables. In order to do that, we have to modify the quadratic form $\omega_n$ through a linear transformation. 
The next proposition characterise the linear forms which can appear in a harmonic monomial. 
\begin{prop}\label{prop:orthogonal_iff_monomial}
    Let $h=\ell_{1}\cdots\ell_{d}\in\calR_{n,d}$ be a monomial. Then $h\in\calH_{n,d}^{\omega_n}$ if and only if $\ell_j$ and $\ell_k$ are $\omega_n$-orthogonal whenever $1\leq j<k\leq d$.
\end{prop}
\begin{proof}
First, observe that, if 
\[
\omega_n=\xi_0^2+\cdots+\xi_n^2,
\]
for some linear forms $\xi_0,\dots,\xi_n\in\calD_{n,1}$, then
\[
\omega_n\circ(fg)=(\omega_n\circ f)g+f(\omega_n\circ g)+2\sum_{i=0}^n(\xi_i\circ f)(\xi_i\circ g),
\]
for every $f,g\in\calR_n$. In particular, if $\ell_j\ell_k\in\calR_{n,1}$, we have
\[
\omega_n\circ(\ell_j\ell_k)=2\sum_{i=0}^n(\xi_i\circ \ell_j)(\xi_i\circ \ell_k)
\]
and thus, by Leibniz's rule, given any $\ell_1,\dots,\ell_d\in\calR_{n,1}$, we have
\begin{equation}\label{formula:quadric_Leibniz}
\omega_n\circ(\ell_1\cdots\ell_d)=2\sum_{1\leq j< k\leq d}\biggl(\sum_{i=0}^n(\xi_i\circ \ell_j)(\xi_i\circ \ell_k)\biggr)\prod_{\substack{1\leq t\leq d\\ t\neq j,k}}\ell_t=\sum_{1\leq j< k\leq d}\omega_n\circ(\ell_j\ell_k)\prod_{\substack{1\leq t\leq d\\ t\neq j,k}}\ell_t.
\end{equation}
The statement follows directly from formula \eqref{formula:quadric_Leibniz}, since the set of monomials
   \[
   \Set{\prod_{\substack{1\leq t\leq d\\ t\neq j,k}}\ell_t}_{1\leq j< k\leq d}
   \]
is a set of linearly independent forms in $\calR_{n,d}$.
\end{proof}
As a consequence of \autoref{prop:orthogonal_iff_monomial}, we can explicitly describe how a harmonic monomial is made in terms of isotropic linear forms.
\begin{cor}\label{cor:structure_monomials}
    Let $\rmm=\ell_0^{a_0}\cdots\ell_r^{a_r}\in\calR_{n,d}$.Then $h\in\calH_{n,d}^{\omega_n}$ if and only if $\ell_i$ and $\ell_j$ are $\omega_n$-orthogonal whenever $0\leq i< j\leq r$ and $\ell_k$ is $\omega_n$-isotropic for every $k$ such that $0\leq k\leq r$ and $a_k>1$.
\end{cor}
Recall that a form $f\in\calR_{n,d}$ is called \textit{concise} if there are no linear forms in $(f)^{-1}$.
\begin{cor}
    For any $n\in\bbN$ and for any non-degenerate quadratic form $\omega_n$, there exists a unique concise $\omega_n$-harmonic monomial , up to $\omega_n$-orthogonal transformations.
\end{cor}
\begin{rem}
Observe that, in general, a product of linear forms can be harmonic even if these are not pairwise $\omega_n$-orthogonal. For instance, if $\omega_1=\alpha_0^2+\alpha_1^2$, the element
\[
x_0^3x_1-x_0x_1^3=x_0x_1(x_0-x_1)(x_0+x_1)
\]
is harmonic, but it is not a monomial in $\bbC[x_0,x_1]$ and it is not a product of $\omega_n$-orthogonal linear forms.
The following two propositions give the isotropic rank of any harmonic monomial.
\end{rem}
\begin{prop}\label{prop:monomi1}
    Let $\rmm=\ell_0^{a_0}\cdots\ell_r^{a_r}\in\calH_{n,d}^{\omega_n}$ be a $\omega_n$-harmonic monomial of degree $d$, with $\ell_i\neq\ell_j$ for $i\neq j$ and $1\leq a_0\leq\cdots\leq a_r\leq d$. 
If 
\[
\abs{\Set{\ell_i|0\leq i\leq r,\,\omega_n\circ\ell_i^2\neq 0}}\neq 1,
\]
then
\[
\irk_{\omega_n}\rmm=\rk\rmm=\prod_{i=1}^r(a_i+1).
\]
\end{prop}
\begin{proof}
    Let us start by the case where all the linear forms $\ell_0,\dots,\ell_r$ are isotropic. Since, by \autoref{cor:structure_monomials}, all the linear forms must also be pairwise $\omega_n$-orthogonal, then we must have\[
    r+1\leq \biggl\lfloor \frac{n+1}{2}\biggr\rfloor.
    \]
    In particular, by the transitivity of $\Oa_{\omega_n}(\bbC)$, we can assume that
    \[
\omega_n=\alpha_0\alpha_{r+1}+\cdots+\alpha_{r}\alpha_{2r+1}+\alpha_{2(r+1)}^2+\cdots+\alpha_{n}^2
    \]
    and the forms $\ell_0,\dots,\ell_r$ to be $x_0,\dots,x_{r}$, so that
    \[
    \rmm=x_0^{a_0}\cdots x_{r}^{a_{r}}.
    \]
The inverse system of $\rmm$, with respect to classical apolar action, is given by
    \[
    (\rmm)^{-1}=(\alpha_{r+1},\dots,\alpha_{n},\alpha_0^{a_0+1},\dots,\alpha_{r}^{a_{r}+1}).
    \]
    If we consider the ideal
    \[
    I_X=(\alpha_{r+1},\dots,\alpha_{n},\alpha_0^{a_1+1}-\alpha_1^{a_1+1},\dots,\alpha_0^{a_{r}+1}-\alpha_{r}^{a_{r}+1})\subset(\rmm)^{-1},
    \]
    then clearly $\omega_n\in I_X$ and $I_X$ is a complete intersection ideal with
    \[
    \ell\bigl(V(I_X)\bigr)=\rk\rmm=\prod_{i=1}^r(a_i+1).
    \]
    \Cref{rem: apofacile} implies $\irk_{\omega_n}\rmm\leq\ell\bigl(V(I_X)\bigr)$ and, since by \cite{CCG12} $\rk\rmm=\ell\bigl(V(I_X)\bigr)$, we can conclude using \Cref{teo:two_times_rank}.
    Let us now consider the case where there are at least two non-$\omega_n$-isotropic linear forms. By \autoref{cor:structure_monomials}, there exists $0<k\leq r$ such that $a_0=\cdots=a_k=1$, the linear forms $\ell_0,\dots,\ell_k$ are not $\omega_n$-isotropic, and $\ell_{k+1},\dots,\ell_r$ are $\omega_n$-isotropic. In this case, we may assume that
    \[
\omega_n=\alpha_0^2+\cdots+\alpha_{k}^2+\alpha_{k+1}\alpha_{r+1}+\cdots+\alpha_{r}\alpha_{2r-k}+\alpha_{2r-k+1}^2+\cdots+\alpha_n^2
    \]
    and also
    \[
    \rmm=x_0\cdots x_{k}x_{k+1}^{a_{k+1}}\cdots x_{r}^{a_{r}}.
    \]
The inverse system of $\rmm$ with respect to classical apolarity is
\[
    (\rmm)^{-1}=(\alpha_{r+1},\dots,\alpha_{n},\alpha_0^2,\dots,\alpha_{k}^2,\alpha_{k+1}^{a_{k+1}+1},\dots,\alpha_{r}^{a_{r}+1}).
    \]
    Then, if we consider the ideal 
    \[
    I_X=(\alpha_{r+1},\dots,\alpha_{n},\alpha_0^2+k\alpha_1^2,\dots,\alpha_0^2+k\alpha_{k}^2,\alpha_0^{a_{k+1}+1}-\alpha_{k+1}^{a_{k+1}+1},\dots,\alpha_0^{a_{r}+1}-\alpha_{r}^{a_{r}+1})\subset (\rmm)^{-1},
    \]
    since $k\neq 0$, we have
    \[
    \omega_n=\frac{1}{k}\sum_{i=1}^{k}\bigl(\alpha_0^2+k\alpha_i^2\bigr)+\alpha_{k+1}\alpha_{r+1}+\cdots+\alpha_{r}\alpha_{2r-k}+\alpha_{2r-k+1}^2+\cdots+\alpha_n^2\in I_X.
    \]
    Proceeding as in the previous case, we conclude that $\irk_{\omega_n}(\rmm)=\rk\rmm$.
    \end{proof}
\begin{prop}\label{prop:monomi2}
Let $\rmm=\ell_0^{a_0}\cdots\ell_r^{a_r}\in\calH_{n,d}^{\omega_n}$ be a $\omega_n$-harmonic monomial of degree $d$, with $\ell_i\neq\ell_j$ for any $i\neq j$ and $1\leq a_0\leq\cdots\leq a_r\leq d$. If $a_0=1$, $\ell_0$ is not $\omega_n$-isotropic and $\ell_i$ is $\omega_n$-isotropic for any $1\leq i\leq r$, then
\[
\irk_{\omega_n}\rmm=2\rk\rmm=2\prod_{i=1}^r(a_i+1).
\]
\end{prop}    
\begin{proof}    
By the transitivity of $\Oa_{\omega_n}(\bbC)$ we may assume that
    \[
    \omega_n=\alpha_0^2+\alpha_{1}\alpha_{r+1}+\cdots+\alpha_{r}\alpha_{2r}+\alpha_{2r+1}^2+\cdots+\alpha_n^2
    \]
    and also
    \[
    \rmm=x_0x_1^{a_1}\cdots x_{r}^{a_{r}}.
    \]
The inverse system of $\rmm$ with respect to classical apolarity is    
\begin{align*}
(\rmm)^{-1}&=(\alpha_{r+1},\dots,\alpha_{n},\alpha_0^2,\alpha_1^{a_1+1},\dots,\alpha_{r}^{a_{r}+1})=(\alpha_{r+1},\dots,\alpha_{n},\omega_n,\alpha_1^{a_1+1},\dots,\alpha_{r}^{a_{r}+1})\\
&=(\omega_n)+(\alpha_{r+1},\dots,\alpha_{n},\alpha_1^{a_1+1},\dots,\alpha_{r}^{a_{r}+1}).
    \end{align*}
    If we set the monomial $\rmm'=x_1^{a_1}\cdots x_{r}^{a_{r}}$, then the inverse system of $\rmm'$ in $\bbC[\alpha_1,\dots,\alpha_n]$ is
    \[
    (\rmm')^{-1}=(\alpha_{r+1},\dots,\alpha_{n},\alpha_1^{a_1+1},\dots,\alpha_{r}^{a_{r}+1})\subseteq \bbC[\alpha_1,\dots,\alpha_n]\subseteq \bbC[\alpha_0,\dots,\alpha_n].
    \]
    In the rest of the proof we consider $(\rmm')^{-1}$ as an ideal of $\bbC[\alpha_0,\dots,\alpha_n]$.
    In particular, we have \[
    (\rmm)^{-1}=(\omega_n)+(\rmm')^{-1}\subset\bbC[\alpha_0,\dots,\alpha_n].
    \]
    Let $I_X\subset(\rmm)^{-1}$ be the ideal associated to a reduced scheme of points $X\coloneqq{V(I_X)}\subset\Omega_n$, where $\Omega_n=V(\omega_n)$. Then, we can write
    \[
    I_X=(\omega_n)+I_X'\subset (\omega_n)+(\rmm')^{-1},
    \]
    for some ideal $I_X'\subset (\rmm')^{-1}$. Now, we have
    \[
    X=V(I_X)=V\bigl((\omega_n)+I_X'\bigr)=V(\omega_n)\cap V(I_X').
    \]
    In particular, we have
    \[
    0=\dim V(I_X)\geq \dim V(I_X')-1, 
    \]
    so that $\dim V(I_X')\leq 1$. Moreover, since the generators of $I_X'$ do not involve $\alpha_0$, for any point $Q\in V(I_X')$ we have 
    \[
    \langle p,[1,0,\dots,0]\rangle\subseteq V(I_X').
    \]
    Hence, we conclude that $\dim V(I_X')=1$ and $V(I_X')$ is a finite union of lines passing through the vertex $P=[1,0,\dots,0]$. Now, if $l,l'\subseteq V(I_X')$ are two distinct lines, we have $l\cap l'=P\notin\Omega_n$, and thus
    \[
    \bigl(l\cap \Omega_n\bigr)\cap\bigl(l'\cap \Omega_n\bigr)=\varnothing.
    \]
Finally, any line of $l\subseteq V(I_X')$ must intersect $\Omega_n$ in two distinct points, as $I_X=(\omega_n)+I_X'$ is the ideal of a reduced scheme of points. In particular, $\abs{V(I_X)}$ is the double of the number of lines of $V(I_X')$ or, equivalently,
    $$\abs{V(I_X)}=2\abs{V(I_X')\cap V(\alpha_0)}.$$
    In order to complete the proof we have to find the minimal number of lines of such a $V(I_X')$. To do that, we need the following two claims.\\
    \textbf{Claim 1}: Let $l\subset V(I_X')$ a line. Then $(\alpha_{r+1},\dots,\alpha_{n})\not\subset I_l$.
    \begin{proof}[Proof of Claim 1]\let\qed\relax
    Suppose by contradiction that $(\alpha_{r+1},\dots,\alpha_{n})\subset I_l$. Then, parametric equations of $l$ are
    $$l:\begin{cases}
    \alpha_0=u,\\
    \alpha_1=c_1v,\\
    \vdots\\
    \alpha_{r}=c_{r}v,\\
    \alpha_{r+1}=0,\\
    \vdots\\
    \alpha_n=0,
    \end{cases}\qquad [u,v]\in\bbP^1,$$ and thus $l\cap\Omega_n$ is given by $\omega_n(u,c_1v,\dots,c_{r}v,0,\dots,0)=0$, i.e.~$u^2=0$. Hence, $l$ is tangent to $\Omega_n$, a contradiction.
    \end{proof}
    \textbf{Claim 2}: {\it Let $e\in\bbN$ be the minimal number of lines of a 1-dimensional cone $Y$ of vertex $P$ such that $I_Y\subset(\rmm')^{-1}$ and any line of $Y$ is not tangent to $\Omega_n$. Then, there exists $Z\subset\bbP^n$ such that $I_Z\subset(\rmm')^{-1}$, $Z$ is a 1-dimensional cone of vertex $P$ made of $e$ lines which are not tangent to $\Omega_n$ and $(\alpha_{r+1},\dots,\alpha_{n-1})\subset I_Z$.}
    \begin{proof}[Proof of Claim 2]\let\qed\relax
    We set $a=a_1+\dots+a_{r}$. Let $Y$ such in the statement of Claim 2 and let 
    $$\rmm'=\sum_{i=1}^e\ell_i^a$$
    the corresponding decomposition of $\rmm'$, with $$\ell_i=q_{i,1}x_1+\dots+q_{i,n}x_n\in\bbC[x_1,\dots,x_n]\subset\bbC[x_0,\dots,x_n].$$ Since $\rmm'$ does not involve $x_{r+1},\dots,x_{n}$, we have
    $$\rmm'=\sum_{i=1}^e\ell_i^a(0,x_1,\dots,x_{r},c_{i,r+1}x_n,\dots,c_{i,n-1}x_{n},x_n),$$
    for any $c_{i,j}\in\bbC$, $1\leq i\leq e$ and $r+1\leq j\leq n-1$.
    Let 
    $$Q_i=[0,q_{i,1},\dots,q_{i,r},0,\dots,0,c_{i,r+1}q_{i,r+1}+\dots+c_{i,n-1}q_{i,n-1}+q_{i,n}]\in\bbP^n$$
    the point corresponding to $\ell_i(0,x_1,\dots,x_{r},c_{i,r+1}x_n,\dots,c_{i,n-1}x_{n},x_n)$ and $Z\subset\bbP^n$ the cone of vertex $P$ and base $\{Q_1,\dots,Q_e\}$. Clearly, we have $I_Z\subset(\rmm')^{-1}$ and $(\alpha_{r+1},\dots,\alpha_{n-1})\subset I_Z$. Moreover, $Z$ is made of $e$ distinct lines. It remains to prove that, for suitable $c_{i,j}$, each of these lines is not tangent to $\Omega_n$. Parametric equations of the line $l_i\coloneqq\langle PQ_i\rangle$ are
    $$l_i\colon\begin{cases}
    \alpha_0=u,\\
    \alpha_1=q_{i1}v,\\
    \vdots\\
    \alpha_{r}=q_{i,r}v,\\
    \alpha_{r+1}=0,\\
    \vdots\\
    \alpha_{n-1}=0,\\
    \alpha_n=(c_{i,r+1}q_{i,r+1}+\dots+c_{i,n-1}q_{i,n-1}+q_{i,n})v,
    \end{cases}\quad [u,v]\in\bbP^1,$$ and thus $l_i\cap\Omega_n$ is given by 
    $$u^2+(c_{i,r+1}q_{i,r+1}+\dots+c_{i,n-1}q_{i,n-1}+q_{i,n})^2v^2=0.$$
    In particular, $l_i$ is tangent to $\Omega_n$ if and only if $$c_{i,r+1}q_{i,r+1}+\dots+c_{i,n-1}q_{i,n-1}+q_{in}=0.$$ 
    By Claim 1, there exists at least a non-zero element in $\{q_{i,r+1},\dots,q_{i,n}\}$ and, as a consequence, it is possible to choose $c_{i,r+1},\dots,c_{i,n-1}$ such that $l_i$ is not tangent to $\Omega_n$.
    \end{proof}
    Since the number of lines of $V(I_X')$ equals the number of points of $V(I_X')\cap V(\alpha_0)$, we abuse of notation and we consider now $I_X'$ as an ideal of $\calD_{n-1}=\bbC[x_1,\dots,x_n]$ and we write $V(I_X')$ instead of $V(I_X')\cap V(\alpha_0)$. By Claim 2, we can suppose that $(\alpha_{r+1},\dots,\alpha_{n-1})\subset I_X'$ and thus, by Claim 1, we have that the class of $\alpha_n$ is not a zero-divisor in $\calD_n/I_X'$. As a consequence, the sequence
    $$\begin{tikzcd}
	0 \arrow[r] & \calD_{n-1}/I_X'(-1) \arrow[r, "\cdot\alpha_n"] & \calD_{n-1}/I_X' \arrow[r] & \calD_{n-1}/\bigl(I_X'+(\alpha_n)\bigr) \arrow[r] & 0
\end{tikzcd}$$
is exact, and, for $t\gg 0$, we have
$$\abs{V(I_X')}=H_{\calD_{n-1}/I_X'}(t)=\sum_{i=0}^tH_{\calD_{n-1}/(I_X'+(\alpha_n))}(i).$$
Now, recall that $I_X'\subset(\rmm')^{-1}$ and note that $(\rmm')^{-1}+(\alpha_n)=(\rmm')^{-1}$. Thus, $$I_X'+(\alpha_n)\subset(\rmm')^{-1}\subset\calD_{n-1}$$ and, for $t\gg0$, we have
$$\sum_{i=0}^tH_{\calD_{n-1}/(I_X'+(\alpha_n))}(i)\geq\sum_{i=0}^t H_{\Ann(\rmm')}(i)=\prod_{i=1}^{r-1}(a_i+1)=\rk\rmm.$$
Hence, $\irk_{\omega_n}\rmm\geq 2\rk\rmm$ and, by \autoref{teo:two_times_rank}, we get the statement.
\end{proof}
\Cref{prop:monomi1} and \Cref{prop:monomi2} together imply \Cref{teo:harmonics_monomials}.
\begin{exam}
    Let us consider the case of the monomial $\rmm=x_0x_1^{d-1}$ and the quadratic form $\omega_2=\alpha_0^2+\alpha_1\alpha_2$. The linear form $x_0$ is not $\omega_2$-isotropic and the linear form $x_1$ is $\omega_2$-isotropic. Hence, by \Cref{prop:monomi2} we have $\irk\rmm=2d$. Since, by \cite{CCG12} we know that $\rk\rmm=d$, we have found an infinite family of forms for which the upper bound of \Cref{teo:two_times_rank} is sharp.
\end{exam}



\section*{Acknowledgements}
The authors would like to thank Cristiano Bocci, Jaros{\l}aw Buczyński, Enrico Carlini, Fulvio Gesmundo, Alessandro Gimigliano, Monica Idà, Joachim Jelisiejew, Giorgio Ottaviani, Pierpaola Santarsiero, and Ettore Teixeira Turatti for very helpful comments and Paolo Lella for helping them with Macaulay2 code. Part of this work was written while the authors were assistant professors at the University of Warsaw. Part of this work was written when the second author was a research fellow at the {University of Florence}. The first author has been partially founded by the Italian Ministry of University and Research in the framework of the Call for Proposals for scrolling of final rankings of the PRIN 2022 call - Protocol no. 2022NBN7TL. The second author has been partially supported by the scientific project \textit{Multilinear Algebraic Geometry} of the program \textit{Progetti di ricerca di Rilevante Interesse Nazionale} (PRIN), Grant Assignment Decree No.~973, adopted on 06/30/2023 by the Italian Ministry of University and Research (MUR). The authors have been partially supported by the project \textit{Thematic Research Programmes}, Action I.1.5 of the program \textit{Excellence Initiative -- Research University} (IDUB) of the Polish Ministry of Science and Higher Education.

\bibliographystyle{amsalpha}
\bibliography{Bibliography_BCF.bib}
\pagebreak 
\appendix
\section{Macaulay2 code for base cases of inductions }
In this appendix we provide an example of Macaulay2 code to check the base cases of inductions of \autoref{lem:cubic1}, \autoref{lem:cubic2}, \autoref{lem:cubic3}, \autoref{lem:cubic4} and \autoref{teo:post_fin}.
\medskip

\noindent {\bf\small Code for the base case $n=16$ in \autoref{lem:cubic1}}:
\begin{lstlisting}[basicstyle=\ttfamily\scriptsize\linespread{0.8}\selectfont]
Macaulay2, version 1.24.11
R=ZZ/101[x_0..x_16];
L=ideal(x_0..x_5);
M=ideal(x_6..x_11);
N=ideal(x_12..x_16,random(1,R));
Q=ideal(random(2,R));
I=saturate(intersect(L,M,N)+Q);
hilbertFunction(2,R)-hilbertFunction(2,I)
1
\end{lstlisting}
\medskip

\noindent {\bf\small Code for the base case $n=16$ in \autoref{lem:cubic2}}:
\begin{lstlisting}[basicstyle=\ttfamily\scriptsize\linespread{0.8}\selectfont]
KK=ZZ/3;
R=KK[x_0..x_16]; -- the coordinate ring
Q=ideal(random(2,R)); -- a general quadric
L'=ideal(x_0..x_5); -- a codimension 6 linear subspace
M'=ideal(x_6..x_11); -- a codimension 6 linear subspace
N'=ideal(x_12..x_16,random(1,R)); -- a codimension 6 linear subspace
L=saturate(L'+Q); -- the section of Q defined by L'
M=saturate(M'+Q); -- the section of Q defined by M'
N=saturate(N'+Q); -- the section of Q defined by N'

gbL = gens gb L; -- a Groebner basis of L
gbM = gens gb M; -- a Groebner basis of M
gbN = gens gb N; -- a Groebner basis of N

-- a list of ideals of 6 couples of double points of Q on L embedded in P^n
PL = for i from 0 to 5 list (saturate(Q+(L+ideal(random(1,R),random(1,R),
random(1,R),random(1,R),random(1,R),random(1,R),
random(1,R),random(1,R),random(1,R)))^2));

-- a Groebner basis of the ideals of PL up to degree 3
gbPL = for I in PL list gens gb (I,DegreeLimit=>{3}); 

-- a list of ideals of 6 couples of double points of Q on M embedded in P^n
PM = for i from 0 to 5 list (saturate(Q+(M+ideal(random(1,R),random(1,R),
random(1,R),random(1,R),random(1,R),random(1,R),
random(1,R),random(1,R),random(1,R)))^2));

-- a Groebner basis of the ideals of PM up to degree 3
gbPM = for I in PM list gens gb (I,DegreeLimit=>{3}); 

-- a list of ideals of 6 couples of double points of Q on L embedded in P^n
PN = for i from 0 to 5 list (saturate(Q+(N+ideal(random(1,R),random(1,R),
random(1,R),random(1,R),random(1,R),random(1,R),
random(1,R),random(1,R),random(1,R)))^2)); 

-- a Groebner basis of the ideals of PN up to degree 3
gbPN = for I in PN list gens gb (I,DegreeLimit=>{3}); 

-- construction of a generic cubic
dimCubics = hilbertFunction(3,R);
A = KK[Variables=>dimCubics];
AR = A[gens R];
B3 = rsort flatten entries basis(3,AR);
genericCubic = sum for i from 0 to dimCubics-1 list A_i*B3#i;


-- imposing the containment of the generic cubic in the ideals
gbList = {gbL, gbM, gbN} | gbPL | gbPM | gbPN;
equations = {};
for G in gbList do (
    gbG = forceGB sub(G,AR);
    r = genericCubic%gbG;
    equations = equations |  (for t in terms r list leadCoefficient t); 
);

E = gens gb ideal equations;
hF = dimCubics - numColumns E -- number of cubics of P^n containing the subscheme
17
\end{lstlisting}
\noindent {\bf\small Code for the base cases $n=12,\dots,17$ in \autoref{lem:cubic3}}:
\begin{lstlisting}[basicstyle=\ttfamily\scriptsize\linespread{0.8}\selectfont]
KK=ZZ/3;
n=12; -- for n=13,...,17 change n here
R=KK[x_0..x_n]; -- the coordinate ring
Q=ideal(random(2,R)); -- a general quadric
L'=ideal(x_0..x_5); -- a codimension 6 linear subspace
M'=ideal(x_6..x_11); -- a codimension 6 linear subspace
L=saturate(L'+Q); -- the section of Q defined by L'
M=saturate(M'+Q); -- the section of Q defined by M'

gbL = gens gb L; -- a Groebner basis of L
gbM = gens gb M; -- a Groebner basis of M


-- a list of ideals of n-6 couples of double points of Q on L embedded in P^n
PL ={};
for i from 0 to n-7 do(
    I=L;
    for j from 0 to n-8 do(I=I+random(1,R));
    I=saturate(Q+I^2);
    PL=append(PL,I);
    );

-- a Groebner basis of the ideals of PL up to degree 3
gbPL = for I in PL list gens gb (I,DegreeLimit=>{3}); 

-- a list of ideals of n-6 couples of double points of Q on L embedded in P^n
PM ={};
for i from 0 to n-7 do(
    I=M;
    for j from 0 to n-8 do(I=I+random(1,R));
    I=saturate(Q+I^2);
    PM=append(PM,I);
    );

-- a Groebner basis of the ideals of PM up to degree 3
gbPM = for I in PM list gens gb (I,DegreeLimit=>{3}); 

-- a list of ideals of 6 couples of double points of Q embedded in P^n
PQ ={};
for i from 0 to 5 do(
    I=Q;
    for j from 0 to n-2 do(I=I+random(1,R));
    I=saturate(Q+I^2);
    PQ=append(PQ,I);
    );

-- a Groebner basis of the ideals of PM up to degree 3
gbPQ = for I in PQ list gens gb (I,DegreeLimit=>{3});


-- construction of a generic cubic
dimCubics = hilbertFunction(3,R);
A = KK[Variables=>dimCubics];
AR = A[gens R];
B3 = rsort flatten entries basis(3,AR);
genericCubic = sum for i from 0 to dimCubics-1 list A_i*B3#i;


-- imposing the containment of the generic cubic in the ideals
gbList = {gbL, gbM} | gbPL | gbPM | gbPQ;
equations = {};
for G in gbList do (
    gbG = forceGB sub(G,AR);
    r = genericCubic%gbG;
    equations = equations |  (for t in terms r list leadCoefficient t); 
);

E = gens gb ideal equations;
hF = dimCubics - numColumns E -- number of cubics of P^n containing the subscheme
13
\end{lstlisting}
{\bf\small Code for base cases $n=6,\dots,11$ in \autoref{lem:cubic4}}:
\begin{lstlisting}[basicstyle=\ttfamily\scriptsize\linespread{0.8}\selectfont]
-- Base case n=6
restart
KK=ZZ/17;
n=6;
R=KK[x_0..x_n]; -- the coordinate ring
q=x_5*x_6+sum for i from 0 to n-2 list x_i^2;
Q=ideal(q); -- the ideal of the quadric, L is empty in this case

-- a list of ideals of 6  couples of double points of Q embedded in P^n
PQ ={};
for i from 0 to n-1 do(
    I=Q;
    for j from 0 to n-2 do(I=I+random(1,R));
    I=saturate(Q+I^2);
    PQ=append(PQ,I);
    );

-- a Groebner basis of the ideals of PM up to degree 3
gbPQ = for I in PQ list gens gb (I,DegreeLimit=>{3});

-- construction of the scheme eta
eta'=saturate(Q+(ideal(x_0..x_5))^2);
eta=eta'+ideal(x_0);

-- a Groebner basis of eta

gbeta = gens gb eta;

-- construction of a generic cubic
dimCubics = hilbertFunction(3,R);
A = KK[Variables=>dimCubics];
AR = A[gens R];
B3 = rsort flatten entries basis(3,AR);
genericCubic = sum for i from 0 to dimCubics-1 list A_i*B3#i;


-- imposing the containment of the generic cubic in the ideals
gbList = {gbeta} | gbPQ;
equations = {};
for G in gbList do (
    gbG = forceGB sub(G,AR);
    r = genericCubic%gbG;
    equations = equations |  (for t in terms r list leadCoefficient t); 
);

E = gens gb ideal equations;
hF = dimCubics - numColumns E -- number of cubics of P^n containing the subscheme
7

-- Base case n=7
restart
KK=ZZ/17;
n=7;
R=KK[x_0..x_n]; -- the coordinate ring
Q=ideal(random(2,R)); -- the ideal of the quadric
L'=ideal(x_0..x_5); -- a codimension 6 linear subspace
L=saturate(L'+Q); -- the section of Q defined by L'

gbL = gens gb L; -- a Groebner basis of L

-- a list of ideals of 1 couples of double points of Q on L embedded in P^n
PL ={};
for i from 0 to 0  do(
    I=L;
    for j from 0 to n-8 do(I=I+random(1,R));
    I=saturate(Q+I^2);
    PL=append(PL,I);
    );

-- a Groebner basis of the ideals of PL up to degree 3
gbPL = for I in PL list gens gb (I,DegreeLimit=>{3}); 

-- a list of ideals of 7  couples of double points of Q embedded in P^n
PQ ={};
for i from 0 to n-1 do(
    I=Q;
    for j from 0 to n-2 do(I=I+random(1,R));
    I=saturate(Q+I^2);
    PQ=append(PQ,I);
    );

-- a Groebner basis of the ideals of PM up to degree 3
gbPQ = for I in PQ list gens gb (I,DegreeLimit=>{3});

-- construction of a generic cubic
dimCubics = hilbertFunction(3,R);
A = KK[Variables=>dimCubics];
AR = A[gens R];
B3 = rsort flatten entries basis(3,AR);
genericCubic = sum for i from 0 to dimCubics-1 list A_i*B3#i;


-- imposing the containment of the generic cubic in the ideals
gbList = {gbL} | gbPL | gbPQ;
equations = {};
for G in gbList do (
    gbG = forceGB sub(G,AR);
    r = genericCubic%gbG;
    equations = equations |  (for t in terms r list leadCoefficient t); 
);

E = gens gb ideal equations;
hF = dimCubics - numColumns E -- number of cubics of P^n containing the subscheme
8

-- Base case n=8
restart
KK=ZZ/17;
n=8;
R=KK[x_0..x_n]; -- the coordinate ring
q=x_7*x_8+sum for i from 0 to n-2 list x_i^2;
Q=ideal(q); -- the ideal of the quadric
L'=ideal(x_0..x_5); -- a codimension 6 linear subspace
L=saturate(L'+Q); -- the section of Q defined by L'

gbL = gens gb L; -- a Groebner basis of L

-- a list of ideals of 1 couples of double points of Q on L embedded in P^n
PL ={};
for i from 0 to 0  do(
    I=L;
    for j from 0 to n-8 do(I=I+random(1,R));
    I=saturate(Q+I^2);
    PL=append(PL,I);
    );

-- we construct the 3rd double point of Q on L embedded in P^n
I=L'+ideal(x_6,x_8);
I=saturate(Q+I^2);
PL=append(PL,I);

-- a Groebner basis of the ideals of PL up to degree 3
gbPL = for I in PL list gens gb (I,DegreeLimit=>{3}); 

-- a list of ideals of 8  couples of double points of Q embedded in P^n
PQ ={};
for i from 0 to n-1 do(
    I=Q;
    for j from 0 to n-2 do(I=I+random(1,R));
    I=saturate(Q+I^2);
    PQ=append(PQ,I);
    );

-- a Groebner basis of the ideals of PM up to degree 3
gbPQ = for I in PQ list gens gb (I,DegreeLimit=>{3});

-- construction of the scheme eta
eta'=saturate(Q+(ideal(x_0..x_7))^2);
eta=eta'+ideal(x_0,x_1,x_2,x_6,x_7);

-- a Groebner basis of eta

gbeta = gens gb eta;

-- construction of a generic cubic
dimCubics = hilbertFunction(3,R);
A = KK[Variables=>dimCubics];
AR = A[gens R];
B3 = rsort flatten entries basis(3,AR);
genericCubic = sum for i from 0 to dimCubics-1 list A_i*B3#i;


-- imposing the containment of the generic cubic in the ideals
gbList = {gbL,gbeta} | gbPL | gbPQ;
equations = {};
for G in gbList do (
    gbG = forceGB sub(G,AR);
    r = genericCubic%gbG;
    equations = equations |  (for t in terms r list leadCoefficient t); 
);

E = gens gb ideal equations;
hF = dimCubics - numColumns E -- number of cubics of P^n containing the subscheme
9

-- Base case n=9
restart
KK=ZZ/17;
n=9;
R=KK[x_0..x_n]; -- the coordinate ring
q=x_8*x_9+sum for i from 0 to n-2 list x_i^2;
Q=ideal(q); -- the ideal of the quadric
L'=ideal(x_0..x_5); -- a codimension 6 linear subspace
L=saturate(L'+Q); -- the section of Q defined by L'

gbL = gens gb L; -- a Groebner basis of L

-- a list of ideals of 2 couples of double points of Q on L embedded in P^n
PL ={};
for i from 0 to 1  do(
    I=L;
    for j from 0 to n-8 do(I=I+random(1,R));
    I=saturate(Q+I^2);
    PL=append(PL,I);
    );

-- we construct the 5th double point of Q on L embedded in P^n
I=L'+ideal(x_6,x_7,x_9);
I=saturate(Q+I^2);
PL=append(PL,I);

-- a Groebner basis of the ideals of PL up to degree 3
gbPL = for I in PL list gens gb (I,DegreeLimit=>{3}); 

-- a list of ideals of 9  couples of double points of Q embedded in P^n
PQ ={};
for i from 0 to n-1 do(
    I=Q;
    for j from 0 to n-2 do(I=I+random(1,R));
    I=saturate(Q+I^2);
    PQ=append(PQ,I);
    );

-- a Groebner basis of the ideals of PM up to degree 3
gbPQ = for I in PQ list gens gb (I,DegreeLimit=>{3});

-- construction of the scheme eta
eta'=saturate(Q+(ideal(x_0..x_8))^2);
eta=eta'+ideal(x_0,x_1,x_2,x_3,x_6,x_7,x_8);

-- a Groebner basis of eta

gbeta = gens gb eta;

-- construction of a generic cubic
dimCubics = hilbertFunction(3,R);
A = KK[Variables=>dimCubics];
AR = A[gens R];
B3 = rsort flatten entries basis(3,AR);
genericCubic = sum for i from 0 to dimCubics-1 list A_i*B3#i;


-- imposing the containment of the generic cubic in the ideals
gbList = {gbL,gbeta} | gbPL | gbPQ;
equations = {};
for G in gbList do (
    gbG = forceGB sub(G,AR);
    r = genericCubic%gbG;
    equations = equations |  (for t in terms r list leadCoefficient t); 
);

E = gens gb ideal equations;
hF = dimCubics - numColumns E -- number of cubics of P^n containing the subscheme
10

-- Base case n=10
restart
KK=ZZ/3;
n=10;
R=KK[x_0..x_n]; -- the coordinate ring
q=x_9*x_10+sum for i from 0 to n-2 list x_i^2;
Q=ideal(q); -- the ideal of the quadric
L'=ideal(x_0..x_5); -- a codimension 6 linear subspace
L=saturate(L'+Q); -- the section of Q defined by L'

gbL = gens gb L; -- a Groebner basis of L


-- a list of ideals of 3 couples of double points of Q on L embedded in P^n
PL ={};
for i from 0 to 2  do(
    I=L;
    for j from 0 to n-8 do(I=I+random(1,R));
    I=saturate(Q+I^2);
    PL=append(PL,I);
    );

-- we construct the 7th double point of Q on L embedded in P^n
I=L'+ideal(x_6,x_7,x_8,x_10);
I=saturate(Q+I^2);
PL=append(PL,I);

-- a Groebner basis of the ideals of PL up to degree 3
gbPL = for I in PL list gens gb (I,DegreeLimit=>{3}); 

-- a list of ideals of 10  couples of double points of Q embedded in P^n
PQ ={};
for i from 0 to n-1 do(
    I=Q;
    for j from 0 to n-2 do(I=I+random(1,R));
    I=saturate(Q+I^2);
    PQ=append(PQ,I);
    );

-- a Groebner basis of the ideals of PM up to degree 3
gbPQ = for I in PQ list gens gb (I,DegreeLimit=>{3});

-- construction of the scheme eta
eta'=saturate(Q+(ideal(x_0..x_9))^2);
eta=eta'+ideal(x_0,x_1,x_2,x_6,x_7);

-- a Groebner basis of eta

gbeta = gens gb eta;


-- construction of a generic cubic
dimCubics = hilbertFunction(3,R);
A = KK[Variables=>dimCubics];
AR = A[gens R];
B3 = rsort flatten entries basis(3,AR);
genericCubic = sum for i from 0 to dimCubics-1 list A_i*B3#i;


-- imposing the containment of the generic cubic in the ideals
gbList = {gbL,gbeta} | gbPL | gbPQ;
equations = {};
for G in gbList do (
    gbG = forceGB sub(G,AR);
    r = genericCubic%gbG;
    equations = equations |  (for t in terms r list leadCoefficient t); 
);

E = gens gb ideal equations;
hF = dimCubics - numColumns E -- number of cubics of P^n containing the subscheme
11

-- Base case n=11
restart
KK=ZZ/3;
n=11;
R=KK[x_0..x_n]; -- the coordinate ring
Q=ideal(random(2,R)); -- a general quadric
L'=ideal(x_0..x_5); -- a codimension 6 linear subspace
L=saturate(L'+Q); -- the section of Q defined by L'

gbL = gens gb L; -- a Groebner basis of L


-- a list of ideals of 5 couples of double points of Q on L embedded in P^n
PL ={};
for i from 0 to 4  do(
    I=L;
    for j from 0 to n-8 do(I=I+random(1,R));
    I=saturate(Q+I^2);
    PL=append(PL,I);
    );

-- a Groebner basis of the ideals of PL up to degree 3
gbPL = for I in PL list gens gb (I,DegreeLimit=>{3}); 

-- a list of ideals of 11  couples of double points of Q embedded in P^n
PQ ={};
for i from 0 to n-1 do(
    I=Q;
    for j from 0 to n-2 do(I=I+random(1,R));
    I=saturate(Q+I^2);
    PQ=append(PQ,I);
    );

-- a Groebner basis of the ideals of PM up to degree 3
gbPQ = for I in PQ list gens gb (I,DegreeLimit=>{3});


-- construction of a generic cubic
dimCubics = hilbertFunction(3,R);
A = KK[Variables=>dimCubics];
AR = A[gens R];
B3 = rsort flatten entries basis(3,AR);
genericCubic = sum for i from 0 to dimCubics-1 list A_i*B3#i;


-- imposing the containment of the generic cubic in the ideals
gbList = {gbL} | gbPL | gbPQ;
equations = {};
for G in gbList do (
    gbG = forceGB sub(G,AR);
    r = genericCubic%gbG;
    equations = equations |  (for t in terms r list leadCoefficient t); 
);

E = gens gb ideal equations;
hF = dimCubics - numColumns E -- number of cubics of P^n containing the subscheme
12
\end{lstlisting}
{\b\small Code for base cases $n=2,\dots,6$ in \autoref{thm:post-3}}
\begin{lstlisting}[basicstyle=\ttfamily\scriptsize\linespread{0.8}\selectfont]
-- Case n=2
restart
n=2;
KK=ZZ/11;
R=KK[x_0..x_n];
Q=ideal(x_0*x_1+x_2^2); --The quadric
P=ideal(x_1,x_2); --One point on the quadric
E=ideal(x_0,x_2);--One point on the quadric
L=ideal(random(1,R));
P12=saturate(L+Q);-- Two points on the quadric
K=intersect(saturate(P^2+Q),saturate(P12^2+Q),E);
hilbertFunction(3,R)-hilbertFunction(3,K)
3

--Case n=3
restart
n=3;
KK=ZZ/11;
R=KK[x_0..x_n];
Q=ideal(x_0*x_1+x_2^2+x_3^2); --The quadric
P=ideal(x_1,x_2,x_3); --One point on the quadric
E=ideal(x_0,x_2,x_3);--One point on the quadric
L1=ideal(random(1,R),random(1,R));
L2=ideal(random(1,R),random(1,R));

P12=saturate(L1+Q);-- Two points on the quadric
P34=saturate(L2+Q);-- Two points on the quadric
K=intersect(saturate(P^2+Q),saturate(P12^2+Q),saturate(P34^2+Q),E);
hilbertFunction(3,R)-hilbertFunction(3,K)
4

--Case n=4
restart
n=4;
KK=ZZ/11;
R=KK[x_0..x_n];
Q=ideal(x_0*x_1+x_2^2+x_3^2+x_4^2); --The quadric
P=ideal(x_1,x_2,x_3,x_4); --One point on the quadric
E=ideal(x_0,x_2,x_3,x_4);--One point on the quadric
L1=ideal(random(1,R),random(1,R),random(1,R));
L2=ideal(random(1,R),random(1,R),random(1,R));
L3=ideal(random(1,R),random(1,R),random(1,R));

P12=saturate(L1+Q);-- Two points on the quadric
P34=saturate(L2+Q);-- Two points on the quadric
P56=saturate(L3+Q);-- Two points on the quadric
K=intersect(saturate(P^2+Q),saturate(P12^2+Q),saturate(P34^2+Q),
saturate(P56^2+Q),saturate(E^2+Q+ideal(x_3,x_4)));
hilbertFunction(3,R)-hilbertFunction(3,K)
5

--Case n=5
restart
n=5;
KK=ZZ/11;
R=KK[x_0..x_n];
Q=ideal(x_0*x_1+x_2^2+x_3^2+x_4^2+x_5^2); --The quadric
L1=ideal(random(1,R),random(1,R),random(1,R),random(1,R));
L2=ideal(random(1,R),random(1,R),random(1,R),random(1,R));
L3=ideal(random(1,R),random(1,R),random(1,R),random(1,R));
L4=ideal(random(1,R),random(1,R),random(1,R),random(1,R));
L5=ideal(random(1,R),random(1,R),random(1,R),random(1,R));

P12=saturate(L1+Q);-- Two points on the quadric
P34=saturate(L2+Q);-- Two points on the quadric
P56=saturate(L3+Q);-- Two points on the quadric
P78=saturate(L4+Q);-- Two points on the quadric
P910=saturate(L5+Q);-- Two points on the quadric
K=intersect(saturate(P12^2+Q),saturate(P34^2+Q),saturate(P56^2+Q),
saturate(P78^2+Q),saturate(P910^2+Q));
hilbertFunction(3,R)-hilbertFunction(3,K)
6

--Case n=6
restart
n=6;
KK=ZZ/11;
R=KK[x_0..x_n];
Q=ideal(x_0*x_1+x_2^2+x_3^2+x_4^2+x_5^2+x_6^2); --The quadric
L1=ideal(random(1,R),random(1,R),random(1,R),random(1,R),random(1,R));
L2=ideal(random(1,R),random(1,R),random(1,R),random(1,R),random(1,R));
L3=ideal(random(1,R),random(1,R),random(1,R),random(1,R),random(1,R));
L4=ideal(random(1,R),random(1,R),random(1,R),random(1,R),random(1,R));
L5=ideal(random(1,R),random(1,R),random(1,R),random(1,R),random(1,R));
L6=ideal(random(1,R),random(1,R),random(1,R),random(1,R),random(1,R));

P12=saturate(L1+Q);-- Two points on the quadric
P34=saturate(L2+Q);-- Two points on the quadric
P56=saturate(L3+Q);-- Two points on the quadric
P78=saturate(L4+Q);-- Two points on the quadric
P910=saturate(L5+Q);-- Two points on the quadric
P1112=saturate(L6+Q);-- Two points on the quadric
P=ideal(x_0,x_2,x_3,x_4,x_5,x_6); --A point on the quadric
K=intersect(saturate(P12^2+Q),saturate(P34^2+Q),saturate(P56^2+Q),
saturate(P78^2+Q),saturate(P910^2+Q),saturate(P1112^2+Q),saturate(P^2+Q+ideal(x_2)));
hilbertFunction(3,R)-hilbertFunction(3,K)
7
\end{lstlisting}
{\b\small Code for base cases $n=4,5,6$ in \autoref{teo:post_fin}}
\begin{lstlisting}[basicstyle=\ttfamily\scriptsize\linespread{0.8}\selectfont]
-- Case n=4
restart
n=4;
KK=ZZ/11;
R=KK[x_0..x_n];
Q=ideal(x_0*x_1+x_2^2+x_3^2+x_4^2);

P = for i from 0 to 6 list 
(saturate(Q+(saturate(Q+ideal(random(1,R),random(1,R),random(1,R))))^2));
gbP = for I in P list gens gb (I,DegreeLimit=>{4});

-- generic quartic
dimQuartics = hilbertFunction(4,R);

A = KK[Variables=>dimQuartics];
AR = A[gens R];
B4 = rsort flatten entries basis(4,AR);
genericQuartic = sum for i from 0 to dimQuartics-1 list A_i*B4#i;

-- conditions
equations = {};
for G in gbP do (
    gbG = forceGB sub(G,AR);
    r = genericQuartic%gbG;
    equations = equations |  (for t in terms r list leadCoefficient t);
);

E = gens gb ideal equations;
hF = dimQuartics - numColumns E
15

--Case n=5
restart
n=5;
KK=ZZ/11;
R=KK[x_0..x_n];
Q=ideal(x_0*x_1+x_2^2+x_3^2+x_4^2+x_5^2);

P = for i from 0 to 9 list 
(saturate(Q+(saturate(Q+ideal(random(1,R),random(1,R),random(1,R))))^2));
P=append(P,saturate(Q+(ideal(x_1,x_2,x_3,x_4,x_5))^2));
gbP = for I in P list gens gb (I,DegreeLimit=>{4});

-- generic quartic
dimQuartics = hilbertFunction(4,R);

A = KK[Variables=>dimQuartics];
AR = A[gens R];
B4 = rsort flatten entries basis(4,AR);
genericQuartic = sum for i from 0 to dimQuartics-1 list A_i*B4#i;

-- conditions
equations = {}; -- condizioni per l'appartenenza
for G in gbP do (
    gbG = forceGB sub(G,AR);
    r = genericQuartic%gbG;
    equations = equations |  (for t in terms r list leadCoefficient t);
);

E = gens gb ideal equations;
hF = dimQuartics - numColumns E
21

--Case n=6
restart
n=6;
KK=ZZ/101;
R=KK[x_0..x_n];
Q=ideal(x_0*x_1+x_2^2+x_3^2+x_4^2+x_5^2+x_6^2);

P = for i from 0 to 14 list 
(saturate(Q+(saturate(Q+ideal(random(1,R),random(1,R),random(1,R))))^2));
P=append(P,saturate(Q+(ideal(x_1,x_2,x_3,x_4,x_5))^2));
gbP = for I in P list gens gb (I,DegreeLimit=>{4});

-- generic quartic
dimQuartics = hilbertFunction(4,R);

A = KK[Variables=>dimQuartics];
AR = A[gens R];
B4 = rsort flatten entries basis(4,AR);
genericQuartic = sum for i from 0 to dimQuartics-1 list A_i*B4#i;

-- conditions
equations = {}; -- condizioni per l'appartenenza
for G in gbP do (
    gbG = forceGB sub(G,AR);
    r = genericQuartic%gbG;
    equations = equations |  (for t in terms r list leadCoefficient t); 
);

E = gens gb ideal equations;
hF = dimQuartics - numColumns E
28
\end{lstlisting}
\end{document}